\documentclass[a4paper,11pt]{amsart}

\usepackage[utf8]{inputenc}
\usepackage{pifont}
\usepackage[english]{babel}
\usepackage{amsmath}
\usepackage{amsthm}
\usepackage{amsfonts,mathrsfs,relsize,bigints,stmaryrd}
\usepackage{amssymb}
\usepackage{anysize}
\usepackage{hyperref}
\usepackage{appendix}
\usepackage{mathtools}
\usepackage{enumerate}
\usepackage{graphicx}
\usepackage{xcolor}
\usepackage{ulem}
\makeatletter  \@addtoreset{equation}{section} \makeatother
\newtheorem{definition}{Definition}[section]

\newtheorem{lemma}{Lemma}[section]
\newtheorem{theorem}{Theorem}[section]
\newtheorem{corollary}{Corollary}[section]
\newtheorem{proposition}{Proposition}[section]

\newtheorem{remark}{Remark}[section]

\DeclareMathOperator{\N}{\mathbb{N}}
\DeclareMathOperator{\R}{\mathbb{R}}
\DeclareMathOperator{\C}{\mathbb{C}}
\DeclareMathOperator{\T}{\mathbb{T}}
\DeclareMathOperator{\TT}{\mathbb{T}}
\DeclareMathOperator{\D}{\mathbb{D}}
\DeclareMathOperator{\Z}{\mathbb{Z}}

\newcommand{\NN}{\mathbb{N}}

\newcommand{\dd}{\mathrm{d}}

\allowdisplaybreaks 

\begin{document}

\title[V-states for active scalar equations]{Unified  theory on  V-states structures  for active scalar equations}
\author{Taoufik Hmidi}
\address{IRMAR, Universit\'e de Rennes 1, Campus de Beaulieu, 35042 Rennes cedex, France}
\email{thmidi@univ-rennes1.fr}
\author{Liutang Xue}
\address{Laboratory of Mathematics and Complex Systems (MOE), School of Mathematical Sciences, Beijing Normal University, Beijing 100875, P.R. China}
\email{xuelt@bnu.edu.cn}
\author{Zhilong Xue}
\address{Laboratory of Mathematics and Complex Systems (MOE), School of Mathematical Sciences, Beijing Normal University, Beijing 100875, P.R. China}
\email{zhilongxue@mail.bnu.edu.cn}

{\thanks{The authors have been partially supported by National Key Research and Development Program of China (No. 2020YFA0712900).
T. Hmidi has been supported by Tamkeen under the NYU Abu Dhabi Research Institute grant.
L. Xue has been supported by National Natural Science Foundation of China (Nos. 12271045, 11771043).}}



\begin{abstract}
This paper revolves around the existence of V-states close to Rankine vortices for active scalar equations with completely monotone kernels. This allows to unify various results on this topic related to geophysical flows. A key ingredient is a new factorization formula for the spectrum using a universal function which is  independent of the model. This function  admits  several interesting properties allowing to track the spectrum distribution.  \end{abstract}

\maketitle

\tableofcontents

\section{Introduction}
In this paper we consider the Cauchy problem of the following two-dimensional (abbr. 2D) active scalar equation
\begin{eqnarray}\label{eq:ASE}	
\left\{\begin{array}{ll}
  \partial_t\omega + (v\cdot \nabla) \omega=0, &\,\,\text{   $(t,\mathbf{x})\in (0,\infty)\times \mathbf{D}$}, \\
  v=\nabla^\perp \psi, & \,\,\text{   $(t,\mathbf{x})\in (0,\infty)\times \mathbf{D}$}, \\
  \omega(0,\mathbf{x}) = \omega_0(\mathbf{x}),& \qquad\text{   $\mathbf{x}\in \mathbf{D}$},
\end{array}\right.
\end{eqnarray}
where $\mathbf{D}$ is either the whole space $\mathbb{R}^2$ or the unit disc $\mathbb{D}$, $\nabla^{\perp}=(\partial_2,-\partial_1)$, $v=(v_1,v_2)$ refers to the velocity field,
$\omega$ is a scalar field understood as vorticity or temperature or buoyancy of the fluid, and
the stream function $\psi$ is prescribed through the following relation
\begin{align}\label{eq:psi}
  \psi(\mathbf{x})=\int_{\mathbf{D}}K(\mathbf{x},\mathbf{y}) \omega(\mathbf{y}) \dd \mathbf{y}.
\end{align}
Hereafter we identify the complex plane $\mathbb{C}$ with $\R^2$. We also assume some symmetry conditions on the kernel $K$, through
\begin{align}\label{Sym01}
  K(\mathbf{x}, \mathbf{y})=K(\mathbf{y},\mathbf{x}),\quad K(\bar{\mathbf{x}},\bar{\mathbf{y}})=K(\mathbf{x},\mathbf{y}),
\end{align}
with $\bar{\mathbf{x}} \triangleq (x_1,-x_2)$ the reflection of $\mathbf{x}=(x_1,x_2)$, and
\begin{align}\label{Sym1}
  K(e^{i\theta} \mathbf{x}, e^{i\theta}\mathbf{y}) = K(\mathbf{x},\mathbf{y}),\quad \theta\in\R.
\end{align}
Then $\omega_0(\mathbf{x})={\bf{1}}_{b\mathbb{D}}(\mathbf{x})$, $b>0$ ($b\mathbb{D}\subset \mathbf{D}$)
is a stationary solution for the equation \eqref{eq:ASE}-\eqref{eq:psi}.

\noindent By taking different forms of the kernel $K$, the equation \eqref{eq:ASE}-\eqref{eq:psi}
includes several important hydrodynamic models as special cases. \vskip1mm
\begin{enumerate}[$\triangleright$]
\item Case $\mathbf{D} =\mathbb{R}^2$,
\begin{align*}
  K(\mathbf{x},\mathbf{y})=-\tfrac{1}{2\pi}\log |\mathbf{x}-\mathbf{y}|,\quad
  \textrm{that is},\quad \psi(\mathbf{x}) = (-\Delta)^{-1} \omega(\mathbf{x}),
\end{align*}
\eqref{eq:ASE}-\eqref{eq:psi} becomes the 2D Euler equation in the vorticity form,
which describes the motion of a 2D inviscid incompressible fluid and is a fundamental model in fluid dynamics. \vskip1mm
\item Case  $\mathbf{D}=\mathbb{R}^2$,
\begin{align*}
  K(\mathbf{x},\mathbf{y})={c_\beta}|\mathbf{x}-\mathbf{y}|^{-\beta},\;\beta\in (0,2),\quad
  \textrm{that is},\quad \psi(\mathbf{x}) = (-\Delta)^{-1+\frac{\beta}{2}} \omega(\mathbf{x}),
\end{align*}
with $c_\beta\triangleq \frac{\Gamma(\frac{\beta}{2})}{\pi2^{2-\beta}\Gamma(1-\frac{\beta}{2})} $,
\eqref{eq:ASE}-\eqref{eq:psi} is the inviscid generalized surface quasi-geostrophic (abbr. gSQG) equation.
In particular, for $\beta=1$, it is the surface quasi-geostrophic (abbr. SQG) equation which is a simplified model
to track the atmospheric circulation near the the tropopause \cite{HPGS} and the ocean dynamics in the upper layers \cite{LK06}.
This model in the range $0<\beta<1$ was introduced in \cite{CFMR} as an interpolation between the 2D Euler equation
and the SQG equation.
\vskip1mm
\item Case $\mathbf{D}=\mathbb{R}^2$,
\begin{align*}
  K(\mathbf{x},\mathbf{y})=\tfrac{1}{2\pi}\mathbf{K}_0(\varepsilon|\mathbf{x}-\mathbf{y}|),\;\varepsilon >0,\quad
  \textrm{that is},\quad \psi(\mathbf{x}) = ( - \Delta + \varepsilon^2)^{-1} \omega(\mathbf{x}),
\end{align*}
with $\mathbf{K}_0$ the modified Bessel function (see Subsection \ref{subsec:Bessel}),
\eqref{eq:ASE}-\eqref{eq:psi} is the quasi-geostrophic shallow water (abbr. QGSW) equation.
This model is derived asymptotically from the rotating shallow water equations in the limit of fast rotation
and small variation of free surface \cite{Val08}. The parameter $\varepsilon$ is known as the inverse `Rossby deformation length',
and small $\varepsilon$ physically corresponds to a nearly-rigid free surface.
\vskip1mm
\item Case $\mathbf{D}=\mathbb{D}$, consider the kernel
\begin{align*}
  K(\mathbf{x}, \mathbf{y}) =  - \tfrac{1}{2\pi} \log \left(\tfrac{|\mathbf{x} -\mathbf{y}|}{\big|1-\overline{\mathbf{x}}\mathbf{y}\big|}\right),
\end{align*}
which  describes  the Green function associated with the spectral Laplace operator $-\Delta$ in the unit disc $\mathbb{D}$
with Dirichlet boundary condition, and \eqref{eq:ASE}-\eqref{eq:psi} becomes the 2D Euler equation (in vorticity form) in the unit disc.
\end{enumerate}
\noindent Besides these examples, we refer to  Section \ref{sec:example}  for more
active scalar equations \eqref{eq:ASE}-\eqref{eq:psi}.

\noindent Owing to their substantial physical relevance and formal simplicity,
 active scalar equations \eqref{eq:ASE}-\eqref{eq:psi} have garnered considerable  attention over the  past decades.
Significant  progress has been achieved across multiple fronts.
The global well-posedness of classical solutions for 2D Euler equation in the whole space $\mathbb{R}^2$
or in any smooth bounded domain $\mathbf{D}$
is  well-known for a long time,  see for instance  \cite{MP94},
while it is still an open problem  for the  gSQG equation with $\beta\in(0,2)$. As to  the local  well-posedness issue in the framework of Sobolev spaces, it was explored in  \cite{CCCG12}
for  the whole space and in \cite{ConN18a} for  smooth bounded domains.
On the other hand, the $L^2$-weak solutions for gSQG equation are known to exist  globally in time, see \cite{Resnick,Mar08,LX19} for  $\R^2$
and \cite{ConN18b,NHQ18} for  smooth bounded domains.
Recently, their non-uniqueness aspect  in the plane  has been investigated in \cite{BSV19,Isett21}.\\
Another  significant  class of solutions extensively  studied  in the literature involves  the \textit{patch solutions},
which are solutions to \eqref{eq:ASE}-\eqref{eq:psi} with initial data in the form  of the characteristic function of a  bounded domain $D$, that is, $\omega_0(\mathbf{x})={\bf{1}}_D(\mathbf{x})$.
According to Yudovich \cite{Yud63}, the vorticity patch solution for 2D Euler equation in whole space is globally well-defined keeping during the motion  the form of  the patch structure. The patch problem initiated in 1980s revolves around the regularity persistence of the boundary. It aims to determine whether the initial regularity, for instance of type $C^{k,\gamma}$ with $k\in\mathbb{N}^\star$ and  $0<\gamma<1$,  can persist for all time. This problem was successfully tackled  by  Chemin \cite{Chemin93}, see also Bertozzi and Constantin \cite{BerC93} for another proof.
Similar results in half plane or within  smooth bounded domain were also obtained in \cite{Dep99,KRYZ,Kiselev19}.
The situation turns out to be more involved  for gSQG equation with $\beta\in (0,2)$. Here,  only local-in-time persistence in Sobolev spaces has been  established as proved in   \cite{CCCG12,Gan08,Rod05}.
In this case, some numerical experiments show strong evidence for the finite-time singularity formation,
see \cite{CFMR,SD14,SD19}. 
Finite-time singularity results with multi-signed patches in half plane in various range of $\beta$ has been accomplished
by Kiselev et al \cite{KRYZ,KYZ17} and also \cite{GanP21}. Very recently, ill-posedness results in various  H\"older and Sobolev spaces, associated with  the boundary  patches or to  the initial data,  have been established in \cite{Cor-Zo21,Cor-Zo22,KL23a,KL23}.
\vskip2mm

\noindent The main goal of this paper  is to construct  time periodic solutions in the patch form for  active scalar equations
\eqref{eq:ASE}-\eqref{eq:psi} with a general kernel  form $K$ that will cover most of the equations arising in geophysical flows.
This type of patch solutions are commonly known as \textit{V-states}, or \textit{relative equilibria} or \textit{rotating patches}. Their shape is not altered during the motion and can  be described through  a rigid body transformation.
By identifying $\mathbb{R}^2$ with the complex \mbox{plane $\mathbb{C}$} and
assuming that the center  of rotation is the origin,
the V-states take the form $\omega(\mathbf{x},t)=1_{D_t}(\mathbf{x})$, with
$D_t=e^{i\Omega t}D$, where $D\subset \mathbb{R}^2$ is a bounded domain.
The real number  $\Omega$ is called the \textit{angular velocity} of the rotating domain and  will play the role of a bifurcation parameter.

\noindent The V-states study  for  active scalar equations
\eqref{eq:ASE}-\eqref{eq:psi} has a long history and it is still an active area with intensive research. Over the last few decades, significant contributions at both analytical and numerical levels have shaped this field. The first example of rotating patches for Euler equations dates back to Kirchhoff \cite{Kirch},
who proved that any ellipse with semi-axis $a$ and $b$  rotates uniformly with  the angular velocity
$\Omega=\frac{ab}{a^2+b^2}$,
see also \cite[p. 232]{Lamb45}. 
About one century later, Deem and Zabusky \cite{DZ78} 
conducted numerical computations showcasing the existence of implicit V-states with $m$-fold symmetry.
This was analytically justified by Burbea \cite{Burbea82} using the  bifurcation theory and conformal parametrization. Actually,
the bifurcation from the  Rankine vortices (radial case)  occurs
at the angular velocities $\Omega=\frac{m-1}{2m}$ ($m\geqslant 2$).
Later, Hmidi, Mateu and Verdera \cite{HMV13} revisited this construction and
 show the $C^\infty$ boundary regularity and convexity of the bifurcated V-states close to Rankine vortices. The analyticity of the boundary has been recently explored  by
Castro, C\'ordoba and G\'omez-Serrano in \cite{CCG16b}, and its global version has been discussed by Hassainia, Masmoudi and Wheeler in \cite{HMW20}.

\noindent The V-states for the gSQG model in the whole plane was first investigated by  Hassainia and Hmidi in \cite{HH15} and confirmed a similar result to
 Burbea  for all  $\beta\in (0,1)$. Later,
Castro, C\'ordoba and G\'omez-Serrano \cite{CCG16} extended the construction  for the range $\beta\in [1,2)$ and proved
the $C^\infty$ boundary regularity; see also \cite{CCG16b}
for the real analyticity of the V-states boundary.\\
Similar rigid time periodic solutions for the   QGSW equation  was studied by Dritschel, Hmidi, and Renault \cite{DHR19}.
The  topic of V-states in radial domains with rigid boundary was initiated  by De la Hoz, Hassainia, Hmidi and  Mateu for 2D Euler equation in \cite{DHHM} and by  the authors of this paper to  gSQG equation
\cite{HXX23}.

\noindent Besides the above results, there are abundant papers in recent literature on the mathematical
study of V-states for the active scalar equation \eqref{eq:ASE}-\eqref{eq:psi} from various aspects.
For instance, a second family of countable branches bifurcate from Kirchhoff's ellipses was proved in \cite{CCG16b,HM16};
the existence of doubly connected V-states close to the annulus was established in \cite{DHMV16,HM16b,DHH16,Gom19,Rou23b};
concentrated multi vortices centered at regular n-gons or distributed
according to suitable periodic spatial patterns are analyzed in \cite{CQZZ,Garcia21,Garcia20,GS23,HW22,HM17}.
Very recently, the exploration of  time quasi-periodic vortex patches
for some  active scalar equations \eqref{eq:ASE}-\eqref{eq:psi} has been conducted
by employing  advanced tools from the KAM theory,  we refer to \cite{BHM23,Ber-Gan,GS-I,HHM23,HR22,HHR23,HR21,Rou23b}.
For other connected topics one can see \cite{Ao,CCG16,DHH16,GHM22,GHJ20,GHR23,GPSY,Gravejat,HM17}
and the references therein.
\vskip1mm
\noindent
In this paper we intend to develop a unified approach on the construction  of V-states
for the active scalar equation \eqref{eq:ASE}-\eqref{eq:psi} near Rankine vortices. 
More precisely, we shall apply the local bifurcation theory to construct time periodic patch solutions around
the Rankine vortices of \mbox{type $\mathbf{1}_{b\mathbb{D}}$,} with $b>0$ and  $b\mathbb{D}\subset \mathbf{D}$,
for the system \eqref{eq:ASE}-\eqref{eq:psi} by  imposing general assumptions on the kernel $K$,
which include all the aforementioned important models as special examples.
It should be emphasized that the explicit expression of $K$ plays a crucial role
to  the analysis in the previous works, especially along  the spectrum study where we need the monotinicity of the spectrum sequence.
\vskip0.5mm
\noindent
Before describing  our primary  contributions, we need to introduce the equations that govern  rotating simply connected patches. As we will see in Section \ref{sec:eq-linear}, we find it more convenient to parametrize  the boundary  of the V-states close to the
stationary solution $\mathbf{1}_{b\D}$ in terms of polar coordinates $\theta\in \R\mapsto \sqrt{b^2+2r(\theta)}e^{i\theta}$ with
$b>0$, such that $b \mathbb{D}\subset \mathbf{D}$.
The contour dynamics equation can be formulated as a nonlinear integro-differential equation $F(\Omega,r)=0$ with
\begin{align}\label{eq:main}
  \nonumber F(\Omega,r)&\triangleq \Omega r'(\theta)+ \partial_\theta \bigg(\int_0^{2\pi}\int_0^{R(\eta)}
  K(R(\theta)e^{i\theta},\rho e^{i\eta})\rho \dd \rho \dd \eta\bigg),\quad R(\theta)\triangleq \sqrt{b^2+2r(\theta)}\\
  &\triangleq \Omega r'(\theta)+F_1(r).
\end{align}
One can easily show that $F(\Omega,0) =0$ for all $\Omega \in\mathbb{R}$ and therefore the next task is to check that
the local bifurcation tools such as Crandall-Rabinowitz's theorem (see Theorem \ref{thm:C-R} below) applies in this framework.
\vskip1mm

\noindent
The first main result concerns the stream function $\psi$ associated with a convolution kernel 
\begin{align}\label{case:K-1}
K(\mathbf{x},\mathbf{y})=K_0(|\mathbf{x}-\mathbf{y}|),\quad  \forall\, \mathbf{x}, \mathbf{y}\in \mathbf{D},
\end{align}
where the function $t\in(0,\infty)\mapsto K_0(t)$  satisfies the following assumptions,
\begin{enumerate}
\item[($\mathbf{A}$1)] {\it Complete monotonicity}: the function $-K'_0$ is a nonzero completely monotone function (see Definition \ref{def:cmf}), equivalently,
there exists a non-negative measure $\mu$ on $[0, \infty)$ such that
\begin{align}\label{eq:K0prim}
  -K'_0(t)=\int_0^{\infty}e^{-tx }\dd \mu(x),\quad \forall t >0.
\end{align}
\item[($\mathbf{A}$2)] {\it Integrability assumption}: there exists a constant $a_0>0$ 
and some $\alpha\in (0,1)$ such that
\begin{equation}\label{cond:K0}
\begin{aligned}
  \int_0^{a_0} |K_0(t)| t^{-\alpha+\alpha^2}\dd t < \infty.
\end{aligned}
\end{equation}
\end{enumerate}
Note that the assumptions $(\mathbf{A}1)$-$(\mathbf{A}2)$ encompass as special examples the classical equations: Euler equations, gSQG and
QGSW equations, see Section \ref{sec:example} for more discussion.\\
 Our first main result reads as follows.

\begin{theorem}\label{thm:main}
Assume \eqref{case:K-1},
with $K_0$ satisfying the conditions $(\mathbf{A}\mathrm{1})$-$(\mathbf{A}\mathrm{2})$.
Then for any $m\in \mathbb{N}^\star$,
there exists a family of $m$-fold symmetric $V$-states  for the active scalar equation \eqref{eq:ASE}-\eqref{eq:psi}
bifurcating from the Rankine vortices
$\mathbf{1}_{b\D}(\mathbf{x})$, provided that $b\mathbb{D}\subset \mathbf{D},$ at the angular velocity
\begin{align}\label{spectral:Omega-whole}
  \Omega_{m,b}^0=\int_{\T} K_0\big(|2b\sin \tfrac{\eta}{2}|\big)\cos \eta\,\dd \eta
  -\int_{\T} K_0\big(|2b\sin \tfrac{\eta}{2}|\big)\cos(m\eta) \dd \eta.
\end{align}
\end{theorem}
\vskip2mm

\noindent
Motivated by the papers \cite{DHHM,HXX23} on the V-states in radial domains,
our second main result considers the perturbative case where the kernel involved in the stream function  takes  a  more general form  
\begin{align}\label{case:K-2}
  K(\mathbf{x},\mathbf{y})=K_0(|\mathbf{x}-\mathbf{y}|) + K_1(\mathbf{x},\mathbf{y}),
\end{align}
where $K_0$  satisfies $(\mathbf{A}1)$-$(\mathbf{A}2)$, whereas  $K_1$ satisfies
\begin{enumerate}
\item[($\mathbf{A}$3)] {\it{Regularity assumption}}: $K_1\in C^k_{\mathrm{loc}}({\mathbf{D}}^2)$ for some $k\geqslant 4$.
\item[($\mathbf{A}$4)] {\it{Symmetry  assumption}}: we assume that for any $ \mathbf{x},\mathbf{y}\in\mathbf{D},$
\begin{align*}
  K_1(\mathbf{x},\mathbf{y})= K_1(\mathbf{y},\mathbf{x}),
  \,K_1(\bar{\mathbf{x}},\bar{\mathbf{y}})=K_1(\mathbf{x},\mathbf{y}),\,K_1(e^{i\theta}\mathbf{x},e^{i\theta}\mathbf{y}) = K_1(\mathbf{x},\mathbf{y}), \forall \theta\in \R,
\end{align*}
where $\bar{\mathbf{x}} = (x_1, -x_2)$ is the reflection of $\mathbf{x} =(x_1,x_2)$.
\end{enumerate}
\begin{theorem}\label{thm:perturbative}
Consider the general case \eqref{case:K-2} with the  assumptions $(\mathbf{A}\mathrm{1})$--$(\mathbf{A}\mathrm{4})$.
Then there exists a sufficiently large number $m_0\in \mathbb{N}^\star$, such that for any $m\geqslant m_0$,
the  equation \eqref{eq:ASE}-\eqref{eq:psi} admits a family of $m$-fold symmetric V-states
bifurcating from the trivial solution $\mathbf{1}_{b\mathbb{D}}(\mathbf{x})$, provided that
$b\mathbb{D}\subset \mathbf{D}$,
at some  angular velocity $\Omega_{m,b}.$
\end{theorem}

\begin{remark} 
The angular velocity $\Omega_{m,b}$ in Theorem \ref{thm:perturbative} can be explicitly linked to the kernel as follows
\begin{equation*}
  \Omega_{m,b} = -b^{-1}\int_{0}^{2\pi}\int_0^b\Big(\nabla_\mathbf{x}K(be^{i\theta},\rho e^{i\eta})
  \cdot e^{i\theta}\Big)\rho \dd \rho \dd \eta - \int_0^{2\pi} K(b,be^{i\eta})e^{i m \eta}\dd \eta.
\end{equation*}
In particular, with the notation $G_1(\rho_1,\theta,\rho_2,\eta) \triangleq K_1 (\rho_1 e^{i\theta}, \rho_2 e^{i\eta})$, we also have
\begin{equation}\label{spectral:Omega-general}
\begin{aligned}
  \Omega_{m,b}=& \, \Omega_{m,b}^0
  - b^{-1} \int_0^{2\pi} \int_0^b \partial_{\rho_1} G_1(b, 0,\rho, \eta) \rho \dd \rho \dd \eta
  - \int_{\T} K_1(b,b e^{i\eta})\cos (m\eta)\dd \eta.
\end{aligned}
\end{equation}
\end{remark}

In the proof of Theorem \ref{thm:main}, our primary challenge lies in  exploring  the spectrum distribution of the linearized operator to the functional $F_1$ defined in \eqref{eq:main} at the equilibrium state. One of the  crucial ingredient is   the strict monotonicity of the spectrum
$(\Omega_{m,b}^0)_{m\in\mathbb{N}^\star}$ with respect \mbox{to $m$,} needed to get a one-dimensional kernel, which is a requisite condition stipulated   in Crandall-Rabinowitz's theorem, see Theorem \ref{thm:C-R}.\\[0.5mm]
Note that $\Omega_{m,b}^0$ has the expression \eqref{spectral:Omega-whole}
according to the analysis implemented in Subsection \ref{subsec:lin}.
Given  this  representation involving oscillating trigonometric functions, it  is not at all obvious whether this sequence exhibits a monotonic behavior with  general kernel function $K_0$.
A crucial discovery is that when  $K_0$ satisfies the assumption $(\mathbf{A}1)$, then we find an interesting    factorization of  the  spectrum  as follows, see Lemma \ref{lem:lamb-n},
\begin{align*}
  \Omega_{m,b}^0 = 2\int_0^{\infty}\big(\phi_1(bx)-\phi_m(bx)\big)\tfrac{\dd \mu(x)}{x}\quad\hbox{with}\quad
  \phi_m(x) \triangleq \int_0^{\pi}e^{-2\sin (\eta)x}e^{i2m\eta}\dd \eta,
\end{align*}
where $\mu$ is a nonnegative Borel measure. In this factorization, we make appeal to the  universal function  $\phi_m$ which is completely independent of the model and will    encode the key  feature of the spectral distribution. Especially, we show in  Proposition \ref{thm:mono} that  for each  $x>0$,
$\phi_m(x)$ is positive and the sequence $m\in \N^\star\mapsto \phi_m(x)$ is strictly decreasing, which yields in turn to the monotonicity of the spectrum. These properties on $\phi_m$ are not obvious and do not directly result from the definition of $\phi_m$ because the integrand undergoes  oscillations with   changes in sign. The crucial  point here  is that   $\phi_m$   solves a second order linear  differential equation with variable coefficients  given by  \eqref{eq:phi-ODE}. Then applying  an {\it ad hoc} comparison theorem result outlined in  \mbox{Lemma \ref{lem:comparison}} allows  to show  that
$\phi_m$ is positive and strictly decreasing in $m$.
Another serious difficulty lies on the proof of the strong regularity properties of $F(\Omega,r)$ needed in
Crandall-Rabinowitz's theorem.
Since we only impose an  integrability condition on $K_0$ through the  assumption $(\mathbf{A}2)$, 
the boundedness results in \cite{DHR19,HH15,HMV13,HXX23} related to singular kernel integrals with pointwise assumptions on the kernels
can not be directly used.
To circumvent this difficulty we establish   suitable results, see Lemma \ref{lem:int-operator} and Lemma \ref{lem:int2},  dealing with integral operators \eqref{eq:int-operator}
on the torus and use some  persistence regularity estimates employed  several  times to infer the required regularity for $F(\Omega,r)$ as  detailed in
 Subsection \ref{subsec:reg}.
The third delicate point in the proof is to check that $\partial_rF( \Omega_{m,b}^0,0)$ is of co-dimension one.
To this end, we shall use a Mikhlin type multiplier theorem stated  in  Lemma \ref{lem:multiplier-lemma} on the periodic framework, as described  in \mbox{Proposition \ref{propos:bifurcation}.} Another interesting result  is summarized in Proposition \ref{prop:asypt}  and \mbox{Corollary \ref{cor:asym}} where we derive the following spectrum  expansion: for each $N\in\N$ and $n\geqslant1$,
\begin{align*}
  \Omega^0_{n,b} = 2\int_0^{\infty}\phi_1(bx)\tfrac{\dd \mu(x)}{x}-2\sum_{k=0}^{N}\frac{1}{n^{2k+1}}\int_0^\infty \Psi_k(\tfrac{b x}{n}) \tfrac{\dd \mu(x)}{x}
  +\varepsilon_{n,N},
\end{align*}
where
\begin{align*}
  \Psi_0(x)=\frac{x}{1+x^2}, \quad
  \Psi_{k+1}(x)  = \frac{x^2}{4(1+x^2)}\Big(\Psi_k^{\prime\prime}(x)+\frac{1}{x}\Psi'_k(x)\Big),\quad \forall k\geqslant 0,
\end{align*}
and
\begin{align*}
  |\varepsilon_{n,N}|\leqslant \frac{C_{N,\delta}}{n^{2N+\frac{5}{3}}}\int_0^\infty \frac{x^{\delta-1}}{1+\frac{b x}{n}}  \dd \mu(x),\,
  \quad \forall \delta\in[0,\tfrac{1}{3}).
\end{align*}
This holds significant   consequences in classical analysis, illustrated in Section \ref{sec:example} through several examples stemming  from geophysical flows, see  Section \ref{sec:example}. The proof of the foregoing expansion  results on a rescaling argument coupled with an application of the Hankel transform.
\vskip1mm
\noindent
As to  the proof of Theorem \ref{thm:perturbative}, 
the main challenge is still to show the monotonicity of the spectrum sequence $(\Omega_{m,b})_{m\in \mathbb{N}^\star}$, which takes the form
\eqref{spectral:Omega-general} as shown in Subsection \ref{subsec:lin}. The idea is to perform perturbative arguments
where from the regularity assumption on  $K_1$ defined in \eqref{case:K-2}  we derive that
 the last term $\int_{\mathbb{T}} K_1(b,be^{i\eta}) \cos(m\eta) \dd \eta$ involved
in \eqref{spectral:Omega-general} decays in $m$ as $O(m^{-k})$ with some $k\in \mathbb{N}^\star$ that can be chosen large enough.
Thus, to derive the monotonicity property of the sequence $(\Omega_{m,b})_{m\geqslant1}$,
it is enough  to  analyze  the spectrum repartition and show an algebraic lower bound decay for $ \Omega_{m+1,b}^0 -  \Omega_{m,b}^0$.
To this end, we need a more careful quantitative study of the sequence  $(\phi_m(x))_{m\geqslant1}$. In Proposition \ref{eq:bound-diff-phi-n} we show that for every $m\geqslant 1$ and $x>0$,
\begin{align*}
  \frac{1}{2}\frac{(2m+1)x}{(m^2+x^2) \big((m+1)^2+x^2\big)}\leqslant \phi_m(x) - \phi_{m+1}(x)\leqslant
  4 \frac{(2m+1)x}{(m^2+x^2) \big((m+1)^2+x^2\big)}.
\end{align*}
From this, we find according to \eqref{eq:td-Omeg-lbd} a constant $c_*>0$ such that
\begin{align*}
  \Omega_{m+1,b}^0 - \Omega_{m,b}^0 &
  \geqslant \frac{c_* }{m^3}.
\end{align*}
This is the key point to get the spectrum monotonicity for large modes.
Notice that as a by-product of the spectral analysis, we also present a discussion in Section \ref{subsec:lam-conv} concerning  the convexity of  the spectrum $(\Omega^0_{m,b})_{m\geqslant1}$.
\\[0.5mm]
In the section \ref{sec:example}, we will delve into some  applications of Theorem \ref{thm:main} and Theorem \ref{thm:perturbative}. The 2D Euler equations, gSQG and QGSW equations in the whole space align seamlessly with \mbox{Theorem \ref{thm:main}.}
We point out  that the spectral study of the gSQG and QGSW equations as detailed  in \cite{DHR19,HH15} involves intricate analysis on  special functions. Nevertheless, with our  approach  those results are  easily derived yielding  new identities and estimates such as  \eqref{eq:lambda-gSQG}, \eqref{eq:lambda-gSQG-diff}, \eqref{eq:lamb-QGSW-bdd},
\eqref{eq:lamb-QGSW-diff}, \eqref{eq:lamb-QGSW2}.
The V-states for 2D Euler, gSQG and QGSW equations within  the unit disc $\mathbb{D}$ with rigid boundary condition fall under the scope of
Theorem \ref{thm:perturbative} allowing to get  the results outlined  in \cite{DHHM,HXX23}.
Notably, the application to  QGSW equation in $\mathbb{D}$ with rigid boundary condition  is a new contribution.

\vskip1mm
\noindent The remainder of this paper is organized as follows.
In the next section, we introduce the boundary equation modeling the V-states,
consider the linearization around the equilibrium state,
and give an important factorization formula of the spectrum  in terms of the universal function $\phi_n$.
In Section \ref{sec:phi_n}, we focus  on the analysis of some crucial properties of $\phi_n$.
We first prove a useful comparison theorem in Subsection \ref{subsec:comp-thm} allowing
to  derive the positivity and monotonicity of $\phi_n$ and  its asymptotic behavior, see Subsections
\ref{subsec:pos-mono} - \ref{subsec:asym1} respectively.
 This approach  offers  suitable tools in Subsections \ref{subsec:decay-phi} - \ref{subsec:lam-conv} to  track  the
decay rate of $\phi_n - \phi_{n+1}$ and the  convexity of the spectrum. In Section \ref{sec:main}, we give the detailed proofs for
\mbox{Theorem \ref{thm:main}} and Theorem \ref{thm:perturbative} by checking  the required conditions of Crandall-Rabinowitz's theorem.
In Section \ref{sec:example}, we present various examples that follow from  Theorems \ref{thm:main} and \ref{thm:perturbative},
and naturally deduce some interesting properties of the associated spectrum (most are new). 
In Section \ref{sec:tools}, we compile the tools used in the paper: completely monotone functions,
Bessel functions and Hankel transform, boundedness property of some integral operators on the torus,
and Crandall-Rabinowitz's theorem.
\vskip1mm

\noindent \textbf{Notation.} Throughout this paper, the following notation and convention will be used.
\begin{enumerate}[$\bullet$]
  \item The symbol $C$ denotes a positive constant that may change its value from line to line.
  \item We denote the unit disc by $\D$. The unit circle  is denoted by $\T$.
  \item The set $\N=\{0,1,2,\cdots\}$ is composed of nonnegative integers, and $\N^\star =\{1,2,\cdots\}$ only includes positive integers.
  \item Let $\mathbf{X}$ and $\mathbf{Y}$ be two Banach spaces.
  We denote by $\mathcal{L}(\mathbf{X}, \mathbf{Y})$ the space of all continuous linear maps $T : \mathbf{X}\rightarrow \mathbf{Y}$
  endowed with its usual strong topology.
\end{enumerate}

\section{Time periodic patches and linearization}\label{sec:eq-linear}
We have multiple goals in this section. First, we will describe in the context of the vortex patches the contour dynamics  in polar coordinates. Then, we will describe the linearized operator  around Rankine vortices, which are radial equilibrium states. This operator  takes the form of a Fourier multiplier, and its  spectrum within the framework  of completely monotone kernels will be factorized based on a Bessel-type universal function.
\subsection{Boundary equation}
Our primary focus lies in the motion of vortex patches concerning the active scalar equation \eqref{eq:ASE}-\eqref{eq:psi}.
Specifically, the solution takes the form  $\omega(t,\mathbf{x})=\mathbf{1}_{D_t}(\mathbf{x})$, where the domain $ D_t\subset \subset \mathbf{D}$
is a smooth perturbation of  the disc $b\,\mathbb{D}$, with $b>0$.
Note that when the domain $\mathbf{D} = \mathbb{D}$ is the unit disc,  then we impose  $0<b<1$ as in \cite{DHHM,HXX23}.
\\[1mm]
Our analysis will be centered on a specific patch solution within rotating domains, defined by
$$D_t = e^{i t \Omega} D$$
with some angular velocity $\Omega\in \mathbb{R}$. Clearly, this generates a time periodic solution with a \mbox{period $T=\frac{2\pi}{\Omega}.$} In this section, the kernel $K$ involved in the stream function \eqref{eq:psi} satisfies the properties \eqref{Sym01} and  \eqref{Sym1}.
Now, we will  parameterize the boundary $\partial D_t$  using the polar coordinates, as follows
\begin{align}\label{def:z}
  \mathbf{z}(t,\cdot):\mathbb{T}&\mapsto \partial D_t, \nonumber \\
  \theta &\mapsto e^{it \Omega} \mathbf{z}(\theta)\triangleq e^{it\Omega}\sqrt{b^2+2r(\theta)}e^{i\theta},
\end{align}
where $\mathbf{z}(\theta)\in \partial D$.
Denote by $\mathbf{n}(t,\mathbf{z}(t,\theta))\triangleq {i}\partial_\theta \mathbf{z}(t,\theta)$
an inward normal vector to the boundary $\partial D_{t}$ at the point $\mathbf{z}(t,\theta)$.
According to \cite[p. 174]{HMV13}, the vortex patch equation writes
\begin{align*}
  \partial_{t}\mathbf{z}(t,\theta)\cdot \mathbf{n}&=u(t,\mathbf{z}(t,\theta))\cdot\mathbf{n}\\
  &=-\partial_\theta[\psi(t,\mathbf{z}(t,\theta))],
\end{align*}
where $\psi$ is the stream function defined by \eqref{eq:psi}.
Then making a change of variables and using the symmetry property \eqref{Sym1}, we deduce that
\begin{align*}
  \psi(t,\mathbf{z}(t,\theta))& = \int_D K(e^{i t\Omega}\mathbf{z}(\theta),e^{i t\Omega}\mathbf{y})  \dd \mathbf{y} \\
  &=\int_D K(\mathbf{z}(\theta),\mathbf{y})  \dd \mathbf{y}.
\end{align*}
In addition,
\begin{align*}
  \partial_t \mathbf{z}(t,\theta)=i\Omega \mathbf{z}(t,\theta)
  = i\,\Omega\,e^{i t\Omega} \sqrt{b^2+2r(\theta)} e^{ i \theta}
\end{align*}
and
\begin{align*}
  \partial_{t}\mathbf{z}(t,\theta)\cdot \mathbf{n}(t,\mathbf{z}(t,\theta))
  =& \;\mbox{Im}\left(\partial_{t}\mathbf{z}(t,\theta)\,\overline{\partial_\theta \mathbf{z}(t,\theta)}\right) \\
  \stackrel{\eqref{def:z}}=&\, \Omega\, r^\prime(\theta).
\end{align*}
Thus we obtain the equation characterizing the boundary $\partial D$,
\begin{align}
  \Omega\, r'(\theta) = - \partial_\theta \left(\int_D K(\mathbf{z}(\theta),\mathbf{y}) \dd \mathbf{y}\right).
\end{align}
Using the polar coordinates gives
\begin{equation}\label{def:F_0[r]}
\begin{split}
  \int_D K(z(\theta),\mathbf{y})\dd \mathbf{y}
  &= \int_0^{2\pi} \int_0^{R(\eta)} K(R(\theta)e^{i\theta},\rho e^{i\eta}) \rho \dd \rho \dd \eta\\
  &\triangleq F_0[r](\theta) ,
  \qquad\quad\textrm{with}\quad R(\theta)\triangleq \sqrt{b^2+ 2 r(\theta)},
\end{split}
\end{equation}
thus we arrive at
\begin{equation}\label{eq:F}
\begin{split}
  F(\Omega,r)
  \triangleq \Omega \,r'(\theta)+ \partial_\theta F_0[r](\theta)  = 0 .
\end{split}
\end{equation}
Notice that Rankine vortices $\mathbf{1}_{b\mathbb{D}}(\mathbf{x})$ are stationary solutions of the equation \eqref{eq:F}, that is,
\begin{equation*}
  F(\Omega,0)\equiv 0,\quad \forall\, \Omega \in \mathbb{R}.
\end{equation*}
This property follows  easily  from the fact that $F[0]$ is rotationally invariant according to \eqref{Sym1}.

\subsection{Linearization}\label{subsec:lin}
In this section the kernel $K$ in  \eqref{eq:psi} satisfies \eqref{case:K-2} together with the properties ($\mathbf{A}$3) and  ($\mathbf{A}$4).
Linearizing the rotating patch equation \eqref{eq:F}, we obtain
\begin{equation}\label{eq:F_Lin_r}
  \begin{aligned}
		\partial_r F(\Omega,r)h(\theta)
		& =\Omega h'(\theta)
		+ \partial_\theta\left[\frac{h(\theta)}{R(\theta)}\int_0^{2\pi} \int_0^{R(\eta)}
		\left(\nabla_{\mathbf{x}} K\big(R(\theta)e^{i\theta},  \rho e^{i\eta}\big)
		\cdot e^{i\theta} \right)\rho\, \dd\rho \dd\eta\right]\\
		&\quad + \partial_\theta \left(\int_{\mathbb{T}}  K \big(R(\theta)e^{i\theta},
		R(\eta) e^{i\eta}\big) h(\eta) \dd\eta \right)\\
		&\triangleq \partial_{\theta}\Big( \big(\Omega+V[r](\theta)\big) h(\theta)+\mathcal{L}[r](h)(\theta)\Big).
	\end{aligned}
\end{equation}
From \eqref{case:K-1} we infer
\begin{align*}
  \nabla_{\mathbf{x}}K_0(|\mathbf{x}-\mathbf{y}|)=-\nabla_{\mathbf{y}}K_0(|\mathbf{x}-\mathbf{y}|),
\end{align*}
which implies that
\begin{align*}
  V[r](\theta) & = \frac{1}{R(\theta)} \int_0^{2\pi} \int_0^{R(\eta)}
  \left(\nabla_{\mathbf{x}} K_0 \big(|R(\theta)e^{i\theta}-\rho e^{i\eta}|\big) \cdot e^{i\theta} \right) \rho\,
  \dd \rho \dd \eta +  V_1[r](\theta) \nonumber \\
  & = - \frac{1}{R(\theta)} \int_0^{2\pi} \int_0^{R(\eta)}
  \left(\nabla_{\mathbf{y}} K_0 \big(|R(\theta)e^{i\theta}-\rho e^{i\eta}|\big) \cdot e^{i\theta} \right) \rho\,
  \dd \rho \dd \eta + V_1[r](\theta) \nonumber \\
  & = - \frac{1}{R(\theta)} \iint_D
  \left(\nabla_{\mathbf{y}} K_0 \big(|R(\theta)e^{i\theta}-\mathbf{y}|\big) \cdot e^{i\theta} \right) \dd \mathbf{y}
  + V_1[r](\theta),
\end{align*}
with
\begin{align*}
  V_1[r](\theta) \triangleq \frac{1}{R(\theta)} \int_0^{2\pi} \int_0^{R(\eta)}
  \left(\nabla_{\mathbf{x}}K_1(R(\theta)e^{i\theta},\rho e^{i\eta})\cdot e^{i\theta}\right)\rho \dd \rho \dd \eta.
\end{align*}
By using the Gauss-Green theorem, we rewrite $V[r](\theta)$ as
\begin{equation}\label{def:Vr}
\begin{split}
  V[r](\theta)
  & = - \frac{1}{R(\theta)} \int_{\mathbb{T}} K_0\big(\big|R(\theta)e^{i\theta} - R(\eta)e^{i\eta}\big|\big)
  \big(-i\partial_\eta(R(\eta)e^{i\eta}) \big)\cdot e^{i\theta} \dd\eta + V_1[r](\theta) \\
  & \triangleq V_0[r](\theta) + V_1[r](\theta).
\end{split}
\end{equation}
Hence, by setting $G_1(\rho_1,\theta,\rho_2,\eta) \triangleq K_1(\rho_1 e^{i\theta},\rho_2 e^{i\eta})$ and using \eqref{Sym1}, \eqref{def:Vr}, \eqref{eq:G1-prop1}, \eqref{eq:K-der},
at the equilibrium state $r=0$ one has $V[0]$ is a constant independent of $\theta$ and
\begin{align}\label{exp:V0}
  V[0](\theta)=&\, b^{-1}\int_0^{2\pi}\int_0^b\left(\nabla_{\mathbf{x}} K\big(be^{i\theta},  \rho e^{i\eta}\big)
  \cdot e^{i\theta} \right)\rho\, \dd\rho \dd\eta\nonumber \\
  =&- \int_{\mathbb{T}}K_0\big( |be^{i\theta} - b e^{i\eta}|\big)
  \left( e^{i\eta} \cdot e^{i\theta} \right) \dd\eta
  + b^{-1}\int_0^{2\pi}\int_0^b\left(\nabla_{\mathbf{x}} K_1\big(be^{i\theta},  \rho e^{i\eta}\big)
  \cdot e^{i\theta} \right)\rho\, \dd\rho \dd\eta  \nonumber \\
  = & -\int_{\TT} K_0 \big( |b -be^{i\eta}|\big)\cos (\eta)\dd \eta
  + b^{-1}\int_0^{2\pi}\int_0^b\partial_{\rho_1}G_1(b,\theta,\rho,\eta)\rho\, \dd \rho \dd \eta \nonumber \\
  = & -\int_{\TT} K_0 \big( |b -be^{i\eta}|\big) e^{i\eta}\dd \eta
  + b^{-1}\int_0^{2\pi}\int_0^b\partial_{\rho_1}G_1(b,0,\rho,\eta)\rho\, \dd \rho \dd \eta.
\end{align}
In addition, we get by virtue of assumption ($\mathbf{A}$4),
\begin{align}\label{exp:L0h}
  \mathcal{L}[0](h)(\theta)
  = \int_{\mathbb{T}} K(be^{i\theta},b e^{i\eta}) h(\eta) \dd \eta
  = \int_{\mathbb{T}}  K \big(b, be^{i\eta}\big) h(\theta+\eta) \dd\eta.
\end{align}
It is easy to check that the operator $\mathcal{L}[0]$ is a Fourier multiplier. Actually,  for every smooth function $h(\theta)=\sum_{n\in\mathbb{Z}} h_n e^{i n\theta}$,
\begin{align}\label{def:lambda-nb}
  \mathcal{L}[0] (h)(\theta)=\sum_{n\in\mathbb{Z}}  \Lambda_{n,b}\,  h_n e^{i n\theta},\qquad
  \Lambda_{n,b}\triangleq\int_{\mathbb{T}} K( b,be^{i\eta}) e^{i  n\eta}\dd\eta.
\end{align}
Notice that  $\Lambda_{n,b}=\Lambda_{-n,b}$ (owing to \eqref{Sym01}) and the spectrum of  $\mathcal{L}[0]$ is discrete and given by
\begin{align*}
  \textnormal{sp}(\mathcal{L}[0])=\big\{\Lambda_{n,b}, n\in \mathbb{N} \big\}.
\end{align*}
Denoting that
\begin{align*}
  d_r\mathcal{L}[r](h,w) & \triangleq \Big(\frac{\dd}{\dd s}\mathcal{L}[r+sw](h)\Big)\Big|_{s=0} \\
  & = \int_{\T} \bigg(\nabla_{\mathbf{x}}K\big(R(\theta)e^{i\theta},R(\eta)e^{i\eta}\big)\cdot
  \Big(\tfrac{w(\theta)e^{i\theta}}{R(\theta)}\Big)
  + \nabla_{\mathbf{y}}K\big(R(\theta)e^{i\theta},R(\eta)e^{i\eta}\big)\cdot
  \Big(\tfrac{w(\eta)e^{i\eta}}{R(\eta)}\Big)\bigg)\dd \eta,
\end{align*}
and using the chain rule, we find
\begin{equation}\label{exp:parLr=0}
\begin{split}
  \partial_\theta \Big(\mathcal{L}[r](h)(\theta)\Big)\Big|_{r=0}
  & = \Big(d_r\mathcal{L}[r](h,r')(\theta)\Big) \Big|_{r=0} + \partial_\theta\Big(\mathcal{L}[0](h)(\theta)\Big) \\
  & = d_r\mathcal{L}[0](h,0)(\theta) + \mathcal{L}[0](h')(\theta)
  = \mathcal{L}[0](h')(\theta).
\end{split}
\end{equation}
Similarly, we obtain
\begin{align}\label{exp:par-Vr=0}
  \partial_\theta \Big(V[r](\theta) h(\theta)\Big)\Big|_{r=0} = \Big(d_r V[r](\theta) r'(\theta)h(\theta)\Big)\Big|_{r=0} +
  \partial_\theta\Big( V[0](\theta)h(\theta)\Big)  = V[0]h'(\theta).
\end{align}
Consequently, provided that $(\Lambda_{n,b})_{n\in \mathbb{N}^\star}$ is strictly monotone with respect to $n$,
the kernel of $\partial_r F(\Omega,0)$ is nontrivial if and only if (see Subsection \ref{subsec:spect-study} for more discussion)
\begin{align}\label{eq:spectral-set}
  \Omega\in\Big\{-V[0]-\Lambda_{n,b},\, n\in\mathbb{N}^\star\Big\}.
\end{align}
\vskip0.5mm
\noindent In the particular case where   $K(\mathbf{x},\mathbf{y})=K_0(|\mathbf{x}-\mathbf{y}|)$, one gets
\begin{align}\label{lambda_n}
  \Lambda_{n,b}=\lambda_{n,b} \triangleq \int_{\T}K_0(|b-b e^{i\eta}|) e^{in\eta} \dd \eta
  =  2\int_{0}^{\pi}K_0\big( 2b\sin \eta\big) e^{i 2 n\eta}\dd\eta.
\end{align}

\subsection{Spectrum factorization}
The main goal is to factorize the spectrum $\lambda_{n,b}$ given by \eqref{lambda_n}
using a universal function when the kernel $-K_0^\prime$ is completely monotone as in the assumption $(\mathbf{A}1)$.
More precisely, we have the following key result.
\begin{lemma}\label{lem:lamb-n}
Assume that $K(\mathbf{x},\mathbf{y}) = K_0(|\mathbf{x}-\mathbf{y}|)$ with the assumption $(\mathbf{A}1)$ being satisfied.
Then for every $n\in\mathbb{N}^\star$, $\lambda_{n,b}$ given by \eqref{lambda_n} satisfies
\begin{align}\label{eq:lamb-n}
  \lambda_{n,b}=2\int_0^{\infty}\phi_{n}(bx)\tfrac{\dd {\mu}(x)}{x},
\end{align}
with
\begin{align}\label{def:phi-n}
  \phi_n(x) \triangleq \int_0^\pi e^{-2x \sin(\eta)}e^{i {2 n\eta}}\dd\eta.
\end{align}
\end{lemma}

\begin{proof}[Proof of Lemma \ref{lem:lamb-n}]
Under the assumption $(\mathbf{A}1)$, and according to Theorem \ref{lem:bernstein},
we infer the existence of a Borel measure ${\mu}$ on $[0, \infty)$ such that
\begin{align}\label{eq:K0prim2}
  -K'_0(t)=\int_0^{\infty}e^{-tx }\dd {\mu}(x),\quad \forall t >0.
\end{align}
Integrating \eqref{eq:K0prim2} with respect to $t$-variable, and using Fubini's theorem we obtain
\begin{align}\label{Expression-K}
  \nonumber K_0(t) = & K_0(2b) - \int_{2b}^t \int_0^{\infty}e^{-\tau x} \,\dd {\mu}(x)\dd \tau \\
  \nonumber =&K_0(2b)-\int_0^{\infty} \int_{2b}^t e^{- \tau x}\dd \tau \dd {\mu}(x) \\
  =&K_0(2b)+\int_0^{\infty}\tfrac{e^{-t x}-e^{-2bx}}{x}\dd {\mu}(x).
\end{align}
By virtue of Fubini's theorem and \eqref{lambda_n}, we can rewrite the spectrum $\lambda_{n,b}$ as
\begin{align*}
  \lambda_{n,b} =&  2\int_0^\pi \int_0^\infty
  \Big(\frac{e^{-2b\,x \sin \eta }-e^{-2bx}}{x}\dd {\mu}(x)\Big) e^{i2n\eta}\dd \eta \nonumber \\
  = &2\int_0^\infty \int_0^\pi \big( e^{-2bx \sin \eta }
  - e^{-2bx} \big) e^{i2n\eta} \,\dd \eta\, \tfrac{\dd{\mu}(x)}{x} \nonumber \\
  =& 2\int_0^{\infty}\phi_{n}(bx)\tfrac{\dd {\mu}(x)}{x}.
\end{align*}
This achieves the proof of the desired result.
\end{proof}

\section{Analysis of the universal function $\phi_n$}\label{sec:phi_n}
In this section, we shall study various properties of the real-valued function $\phi_n$,
which is defined by \eqref{def:phi-n}.
In \eqref{eq:lamb-n}, we encountered the universal function $\phi_n$ which naturally emerges  in the analysis of the spectrum $\lambda_{n,b}$
of the linearized operator $\partial_r F(\Omega,0)$. The positivity and monotonicity of $\phi_n$, together with
its  asymptotic behavior  and the rate of  decay of $\phi_n - \phi_{n+1}$ are pivotal elements  in the spectral study. We plan to explore these aspects
along  the Subsections \ref{subsec:pos-mono} - \ref{subsec:decay-phi}. Additionally, we leverage some of the properties of $\phi_n$ to introduce a lemma regarding the convexity of $(\lambda_{n,b})_{n\in\mathbb{N}^\star}$ in Subsection \ref{subsec:lam-conv}. This lemma pertains to a specific class of nonnegative measures.
\\[1mm]
Defining $\phi_n$ as in \eqref{def:phi-n} through an integral featuring oscillating trigonometric functions in the integrand makes it challenging to establish  the aforementioned properties, such as positivity or the monotonity. Fortunately, we discover that $\phi_n$ obeys  an ordinary differential equation \eqref{eq:phi-ODE}, which significantly helps us in establishing the desired properties of $\phi_n$. We basically employ  suitable  comparison principles  to \eqref{eq:phi-ODE} as it will be stated in Subsection \ref{subsec:comp-thm}.
\subsection{Comparison theorem}\label{subsec:comp-thm}
We intend to detail a  comparison principle that serves as the cornerstone for establishing
several qualitative and quantitative properties of $\phi_n$.
\begin{lemma}\label{lem:comparison}
Let $a, b:(0,\infty)\to (0,\infty)$ be two given continuous functions and $f\in C^2((0,\infty))$ be a solution to
\begin{eqnarray*}
\left\{\begin{array}{ll}
  f^{\prime\prime}(x) + a(x)f^\prime(x) - b(x)f(x)\leqslant 0, \quad\forall  x>0,\\
   f(0)\geqslant 0,\quad \lim\limits_{x\to \infty}f(x)\geqslant 0.\end{array}\right.
\end{eqnarray*}
Then $f$ is non-negative on $(0,\infty)$, that is, $f(x)\geqslant 0$.
In addition, if $f$  satisfies
\begin{align*}
  f^{\prime\prime}(x) + a(x)f^\prime(x) - b(x)f < 0,\quad \forall  x>0,
\end{align*}
then $f$ is strictly positive on $(0,\infty)$.
\end{lemma}

\begin{proof}[Proof of Lemma \ref{lem:comparison}]
We will start with  proving the first  statement. For this aim, we shall argue by contradiction. Assume that $f$ takes strictly negative values at some points of $(0,\infty)$.
Then in light of the assumptions $f(0)\geqslant 0$ and $\lim\limits_{x\to \infty}f(x)\geqslant 0,$ one can find some $x_0>0$ such that
\begin{align*}
  \inf_{x>0}f(x)=f(x_0) < 0.
\end{align*}	
Hence,
\begin{align}\label{prop:f-inf}
  f^\prime(x_0)=0, \quad f^{\prime\prime}(x_0)\geqslant 0.
\end{align}
Coming back to the differential inequality we find
\begin{align*}
  f''(x_0)\leqslant b(x_0)f(x_0) < 0,
\end{align*}
which is a contradiction.\\
For the second assertion, we assume that $f$ takes non-positive values at some points of $(0,\infty)$,
then there exists some $x_0>0$ so that $\inf_{x>0} f(x) = f(x_0) \leqslant 0 $ which satisfies \eqref{prop:f-inf},
but using the strict  differential inequality  gives $f''(x_0) < b(x_0) f(x_0) \leqslant 0$, and it yields a contradiction.\\
This concludes the proof of the desired result.
\end{proof}

\subsection{Positivity and monotonicity of $\phi_n$}\label{subsec:pos-mono}
This subsection is dedicated to exploring the    application of the comparison theorem in establishing some   qualitative properties of $\phi_n$ introduced in \eqref{def:phi-n}. We shall show the following result.
\begin{proposition}\label{thm:mono}
  For every $n\geqslant 1$ and $x>0$, $\phi_n(x)>0$ and the map $n\mapsto \phi_n(x)$ is strictly decreasing.
\end{proposition}

\begin{proof}[Proof of Proposition \ref{thm:mono}]
For $\mathbf{z} \in \mathbb{C}$, define
\begin{align*}
  \Phi_n(\mathbf{z}) \triangleq \frac{1}{\pi} \int_0^\pi e^{i(- \mathbf{z}\sin \eta + 2n\eta) }\dd \eta.
\end{align*}
Recall the Anger and Weber functions defined successively by, see 8.580 in \cite{GR15},
\begin{align*}
  \mathbf{J}_\nu(\mathbf{z}) = \frac{1}{\pi} \int_0^\pi \cos(\nu \eta - \mathbf{z} \sin \eta) \dd \eta\quad 
  \hbox{and}\quad \mathbf{E}_\nu(\mathbf{z}) = \frac{1}{\pi} \int_0^\pi \sin(\nu\eta - \mathbf{z} \sin \eta) \dd \eta.
\end{align*}
Then we find
\begin{align*}
  \Phi_n(\mathbf{z})=\mathbf{J}_{2n}(\mathbf{z})+i \mathbf{E}_{2n}(\mathbf{z})
\end{align*}
and
\begin{align}\label{Rela-1}
  \phi_n(x)=\pi\Phi(-2 i x).
\end{align}
Now, it is a classical fact that  the functions $\mathbf{J}_{2n}(\mathbf{z})$ and $\mathbf{E}_{2n}(\mathbf{z})$  satisfy the following ODEs, for instance see 8.584 in \cite{GR15},
\begin{align*}
  \mathbf{J}_{2n}^{\prime\prime}(\mathbf{z}) + \mathbf{z}^{-1} \mathbf{J}_{2n}^\prime(\mathbf{z}) + \left(1-\tfrac{4n^2}{\mathbf{z}^2}\right)
  \mathbf{J}_{2n}(\mathbf{z}) = 0
\end{align*}
and
\begin{align*}
  \mathbf{E}_{2n}^{\prime\prime}(\mathbf{z}) + \mathbf{z}^{-1}\mathbf{E}_{2n}^\prime(\mathbf{z}) + \left(1-\tfrac{4n^2}{\mathbf{z}^2}\right) \mathbf{E}_{2n}(\mathbf{z}) = -\frac{2}{\pi \mathbf{z}}\cdot
\end{align*}
It follows that
\begin{align*}
  \Phi_n^{\prime\prime}(\mathbf{z}) + \mathbf{z}^{-1}\Phi_n^\prime(\mathbf{z}) +\left(1-\tfrac{4n^2}{\mathbf{z}^2}\right)\Phi_n(\mathbf{z})
  = - \frac{2 i}{\pi \mathbf{z}}\cdot
\end{align*}
This implies by virtue of \eqref{Rela-1} that
\begin{align}\label{eq:phi-ODE}
  \phi_n^{\prime\prime}(x) + x^{-1}\phi_n^\prime(x) - 4\left(1+\tfrac{n^2}{x^2}\right)\phi_n(x) = -\frac{4}{ x}, \quad  x>0.
\end{align}
On the other hand, one may get  from Riemann-Lebesgue's lemma applied with \eqref{def:phi-n} that
\begin{align}\label{eq:phi-n0}
  \forall n\geqslant 1,\quad \phi_n(0)=0,\quad \textrm{and}\quad  
  \lim_{x\to\infty}\phi_n(x)=0.
\end{align}
Hence, Lemma \ref{lem:comparison} guarantees that
\begin{align}\label{def-T1}
  \forall x>0, \quad \phi_n(x)>0.
\end{align}
\vskip1mm
\noindent Now we show that for any $x>0$ the sequence $n\mapsto \phi_n(x)$ is strictly decreasing.
For this aim, we define
\begin{align*}
  \chi_n(x) \triangleq \phi_n(x) - \phi_{n+1}(x).
\end{align*}
Then using  the equation \eqref{eq:phi-ODE} we find
\begin{align}\label{d-ff-2}
  \chi_n^{\prime\prime}(x) + x^{-1}\chi_n^\prime(x) - 4\left(1+\tfrac{n^2}{x^2}\right)\chi_n(x)
  = - 4 \tfrac{2n+1}{x^2}(x) \phi_{n+1}(x),
  \quad x>0,
\end{align}
with
\begin{align*}
  \chi_n(0)=0, \quad  
  \lim_{x\to\infty}\chi_n(x) = 0.
\end{align*}
Thus, \eqref{def-T1} and Lemma \ref{lem:comparison} ensure that
\begin{align*}
  \forall x>0, \quad \chi_n(x) > 0,
\end{align*}
which implies the strict monotonicity of $\phi_n$.
This concludes the proof of the desired results.
\end{proof}

\subsection{Asymptotic structure of $\phi_n$}\label{subsec:asym1}
The next goal is to explore   the asymptotic behavior of $\phi_n$ with respect to $n$. This will be the crucial step in describing  the asymptotic behavior of the spectrum given through \eqref{eq:lamb-n}.
For this purpose, we shall rescale the function $\phi_n$ as follows,
\begin{align}\label{def:varphi-n}
  \phi_n(x) \triangleq \tfrac{1}{n}\varphi_n(\tfrac{x}{n}).
\end{align}
Then from \eqref{eq:phi-ODE} we easily find  that
\begin{align}\label{eq:varphi_n}
  \tfrac{1}{n^2}\left(\varphi''_n(x)+\tfrac{1}{x}\varphi_n'(x)\right)
  - 4\left(1+\tfrac{1}{x^2}\right)\varphi_n(x)=-\tfrac{4}{x},\quad x>0.
\end{align}
In the following, we plan to provide an expansion formula of $\varphi_n(x)$ in terms of $\frac{1}{n}$.

\begin{proposition}\label{prop:asypt}
For every $x>0$, $n\geqslant 1$, $N\in \mathbb{N}$, we have
\begin{align}\label{eq:decom-phi-n}
  \varphi_n(x)=\sum_{k=0}^{N}\frac{1}{n^{2k}}\Psi_k(x)+g_{n,N}(x),
\end{align}
with
\begin{align}
  \Psi_0(x)&=\frac{x}{1+x^2}, \label{eq:psi-rela1} \\
  \Psi_{k+1}(x) & = \frac{x^2}{4(1+x^2)}\Big(\Psi_k^{\prime\prime}(x)+\frac{1}{x}\Psi'_k(x)\Big),\quad \forall {k\in\mathbb{N},} \label{eq:psi-rela2}
\end{align}
and
\begin{align}\label{eq:gnN}
  |g_{n,N}(x)|\leqslant \frac{C}{n^{2N+\frac{2}{3}-\delta}}\frac{x^\delta}{1+x},
  \quad \forall \delta\in [0,\tfrac{1}{3}),
\end{align}
where $C= C(N,\delta)>0$ is independent of $n$ and $x$.
\end{proposition}

\begin{proof}[Proof of Proposition \ref{prop:asypt}]
We define the second order differential operators
\begin{align}\label{def:L0f}
  \mathbf{L}_0f(x) \triangleq f''(x)+\tfrac{1}{x}f'(x),
\end{align}
and
\begin{align}\label{def:Lf}
  \mathbf{L}f(x) \triangleq \tfrac{1}{n^2} \left(f^{\prime\prime}(x)+\tfrac{1}{x}f^\prime(x)\right)
  -4 \left(1+\tfrac{1}{x^2} \right)f(x) .
\end{align}
Putting this ansatz \eqref{eq:decom-phi-n}
into the equation \eqref{eq:varphi_n}, we obtain that
\begin{align*}
  \sum_{k=0}^N \frac{1}{n^{2k+2}} \mathbf{L}_0 \Psi_k(x) - \frac{4(x^2 + 1)}{x^2}
  \sum_{k=0}^N \frac{1}{n^{2k}} \Psi_k(x) + \mathbf{L} g_{n,N}(x) = -\frac{4}{x},
\end{align*}
that is,
\begin{align*}
  & \sum_{k=0}^{N-1} \frac{1}{n^{2k+2}} \Big( \mathbf{L}_0 \Psi_k(x) - \frac{4(x^2 + 1)}{x^2} \Psi_{k+1}(x)\Big)
  + \frac{\mathbf{L}_0 \Psi_N(x)}{n^{2N+2}}
  + \mathbf{L} g_{n,N}(x)  \\
  & = \frac{4(x^2 + 1)}{x^2}\Psi_0(x) -\frac{4}{x}.
\end{align*}
Taking advantage of the relations \eqref{eq:psi-rela1}-\eqref{eq:psi-rela2} gives the error equation
\begin{align}
  \mathbf{L}g_{n,N}(x) =-\frac{1}{n^{2N+2}}\mathbf{L}_0\Psi_N(x).
\end{align}
Now, let us consider
\begin{align*}
  \widetilde{\mathbf{L}}_0 f(x) \triangleq \frac{x^2}{4(1+x^2)}\mathbf{L}_0 f(x),
\end{align*}
then we write
\begin{align*}
  \Psi_{k+1}(x) = \widetilde{\mathbf{L}}_0\Psi_k(x), \quad \textrm{and} \quad
  \mathbf{L} g_{n,N}(x)=-\frac{1}{n^{2N+2}}\mathbf{L}_0 \widetilde{\mathbf{L}}^N_0\Psi_0(x)
  \triangleq \frac{1}{n^{2N+2}}F_N(x).
\end{align*}
By straightforward computations, using for instance an induction argument, we obtain
\begin{align}\label{Esti-q}
  |\Psi_n(x)|\leqslant\frac{C_n|x|}{(1+x^2)^{n+1}}, \quad\textrm{and}\quad |F_N(x)|\leqslant \frac{C_N}{|x|(1+x^2)^{N+1}}\cdot
\end{align}
Concerning the equation of $g_{n,N}$, it can be written in the form
\begin{align}\label{eq:g_N}
  g''_{n,N}(x)+\frac{1}{x}g'_{n,N}(x)-\frac{4n^2}{x^2}g_{n,N}(x)-4n^2g_{n,N}(x) = \frac{1}{n^{2N}} F_N(x)\cdot
\end{align}
In view of \eqref{eq:decom-phi-n} and the relation $\varphi_n(x) = n\, \phi_n(nx)$, we claim that
\begin{align*}
  \lim_{x\to 0}g_{n,N}(x)=0,\quad \lim_{x\to \infty} x^{\frac{1}{2}}(|g_{n,N}(x)|+|g'_{n,N}(x)|)=0,\quad \forall n,\,N\in \mathbb{N}.
\end{align*}
Indeed, this can be directly justified by the dominated convergence theorem, noticing that for every $x>0$,
\begin{align*}
  |\phi_n(x)| + | \phi_n'(x)| \leqslant 3 \int_0^\pi e^{-2\sin (\eta)x}\dd \eta \leqslant
  6\int_0^{\frac{\pi}{2}} e^{-\frac{4}{\pi}\eta x}\dd \eta \leqslant \min\big\{3\pi, \tfrac{6}{x} \big\},
\end{align*}
and
\begin{align*}
  \forall k\in \mathbb{N},\quad \lim_{x\to 0}\Psi_k(x)=0,\quad \lim_{x\to \infty} x^{\frac{1}{2}}(|\Psi_k(x)|+|\Psi'_k(x)|)=0.
\end{align*}
To estimate this error function, we find it convenient to use the Hankel transform, for more details see Subsection \ref{subsec:Bessel}.
Then applying the Hankel transform $\mathcal{H}_{2n}$ to the equation \eqref{eq:g_N} and using \eqref{hankel:diff}, we get
\begin{align*}
  (-r^2-4n^2)(\mathcal{H}_{2n}\,g_{n,N})(r)=\frac{1}{n^{2N}} \mathcal{H}_{2n}F_N(r).
\end{align*}
In light of the inverse formula \eqref{hankel:inverse} we deduce that
\begin{align}\label{eq:gnN2}
  \forall x>0,\quad g_{n,N}(x)=-\frac{1}{n^{2N}}\int_0^{\infty}\frac{r}{r^2+4n^2} \big(\mathcal{H}_{2n}F_N\big)(r) J_{2n}(xr)\dd r.
\end{align}
Via a change of variable and integration by parts, we find that
\begin{align}\label{decom:gnN}
  g_{n,N}(x)=&-\frac{1}{n^{2N}}\int_0^{\infty}\frac{r}{r^2+4n^2x^2}(\mathcal{H}_{2n}F_N)\big(\tfrac{r}{x}\big)J_{2n}(r)\dd r \nonumber\\
  =&\frac{1}{n^{2N}}\int_0^{\infty}\Big(\frac{r^{-2n}}{r^2+4n^2x^2}(\mathcal{H}_{2n}F_N)\big(\tfrac{r}{x}\big)\Big)^\prime r^{2n+1}
  J_{2n+1}(r)\dd r \nonumber \\
  =& - \frac{1}{n^{2N}}\int_0^{\infty}\frac{2n(\mathcal{H}_{2n}F_N)
  \big(\frac{r}{x}\big)}{r^2+4n^2x^2} J_{2n+1}(r)\dd r
  - \frac{1}{n^{2N}}\int_0^\infty \frac{2r^2(\mathcal{H}_{2n}F_N)
  \big(\frac{r}{x}\big)}{(r^2+4n^2x^2)^2} J_{2n+1}(r)\dd r \nonumber \\
  &\; + \frac{1}{n^{2N}}\int_0^{\infty}\frac{r(\mathcal{H}_{2n}F_N)'(\frac{r}{x})}{(r^2+4n^2x^2)x}J_{2n+1}(r)\dd r \nonumber \\
  \triangleq & \, \textnormal{I} + \textnormal{II}+\textnormal{III},
\end{align}
where in the second line we have used the classical identity $r^{2n+1}J_{2n}(r)= \big(r^{2n+1}J_{2n+1}(r)\big)'$.
\\[1mm]
We point out  that in \cite{Landau00,Olenko06}, it was proved  the existence of  an absolute constant $C_0>0$ independent of $n,x$ so that
\begin{align*}
  |J_n(x)|\leqslant C_0 \min\Big\{n^{-\frac{1}{3}}, x^{-\frac{1}{3}}\Big\},\quad \forall n\in \mathbb{N}, x>0.
\end{align*}
Combining this estimate with the definition \eqref{def:HankelTrans} and \eqref{Esti-q} yields
\begin{align}\label{es:H2nFN}
  |(\mathcal{H}_{2n}F_{N})(r)|\leqslant & \int_0^\infty x |F_N(x)| |J_{2n+1}(x r)|\dd x \nonumber \\
  \leqslant & C_0 \min\Big\{ r^{-\frac{1}{3}}\int_0^{\infty}x^{\frac{2}{3}}|F_N(x)|\dd x,\,
  n^{-\frac{1}{3}} \int_0^\infty x |F_N(x)|\dd x\Big\} \nonumber \\
  \leqslant & C_N\min \Big\{r^{-\frac{1}{3}}\int_0^{\infty}\frac{1}{x^{\frac{1}{3}}(1+x^2)^{N+1}}
  \dd r,\,
  n^{-\frac{1}{3}} \int_0^\infty \frac{1}{(1+x^2)^{N+1}}\dd x\Big\} \nonumber \\
  \leqslant & C_N \min\Big\{r^{-\frac{1}{3}}, n^{-\frac{1}{3}} \Big\}.
\end{align}
Similarly, using the relation \eqref{bessel:recurrence} allows to get
\begin{align*}
  |(\mathcal{H}_{2n}F_{N})'(r)| & \leqslant \int_0^{\infty} x^2 |F_N(x)| |J_{2n}'(x r)|\dd x \\
  & \leqslant \int_0^{\infty} x^2 |F_N(x)| \Big( |J_{2n+1}(x r)| + |J_{2n-1}(x r)| \Big) \dd x \\
  & \leqslant C_N r^{-\frac{1}{3}}.
\end{align*}
Using the interpolation inequlaity and \eqref{bessel:recurrence}
\begin{align}\label{eq:Jn-fact}
  |J_n(x)|\leqslant |J_n(x)|^{1-\delta}|J_n(x)-J_n(0)|^\delta\leqslant C_0 n^{-\frac{1}{3}} x^\delta,\quad \delta\in [0,1], n\geqslant 1,
\end{align}
together with a change of variables,
we infer that for every $\delta\in [0,\frac{1}{3})$,
\begin{align*}
  |\textnormal{I}|&\leqslant C_N\frac{ n^{1-\frac{1}{3}}}{n^{2N}} \int_0^\infty \frac{x^{\frac{1}{3}} r^{- \frac{1}{3} + \delta}}{r^2+4n^2x^2}\dd r\\
  &\leqslant \frac{C_N}{n^{2N+\frac{2}{3}-\delta}} \frac{1}{x^{1-\delta}} \int_0^\infty \frac{s^{\delta-\frac13}}{1+s^2}\dd s\\
 & \leqslant \frac{C_N}{n^{2N+\frac{2}{3}-\delta}}\frac{1}{x^{1-\delta}}\cdot
\end{align*}
Proceeding in the same way, we successively get
\begin{align*}
  |\textnormal{II}|\leqslant C_N \frac{n^{-\frac{1}{3}}}{n^{2N}}
  \int_0^{\infty}\frac{x^{\frac{1}{3}}r^{-\frac{1}{3} +\delta}}{r^2+4n^2x^2}\dd r
  \leqslant \frac{C_N }{n^{2N+\frac{5}{3}-\delta}}\frac{1}{x^{1-\delta}}
\end{align*}
and
\begin{align*}
  |\textnormal{III}|\leqslant C_N \frac{1}{n^{2N}x}
  \int_0^{\infty}\frac{r^{1-\frac{1}{3} + \delta}x^{\frac{1}{3}}}{r^2+4n^2x^2} n^{-\frac{1}{3}}\dd r
  \leqslant  \frac{C_{N,\delta}}{n^{2N+\frac{2}{3}-\delta}}\frac{1}{x^{1-\delta}}\cdot
\end{align*}
Combining these estimates with \eqref{decom:gnN}, we obtain
\begin{align*}
  |g_{n,N}(x)|\leqslant  \frac{C_{N,\delta}}{n^{2N+\frac{2}{3} -\delta}}\frac{1}{x^{1-\delta}}\cdot
\end{align*}
In addition, using \eqref{eq:gnN2} and \eqref{eq:Jn-fact},
we also have that for every $\delta\in [0,\frac{1}{3})$,
\begin{align*}
  |g_{n,N}(x)|\leqslant & \frac{1}{n^{2N}}\int_0^{\infty}\frac{r}{r^2+4n^2}|\mathcal{H}_{2n}F_N(r)||J_{2n}(xr)|\dd r\\
  \leqslant & \frac{C_N}{n^{2N}}\int_0^\infty \frac{r^{\frac{2}{3}}n^{-\frac{1}{3}} (xr)^\delta}{r^2+4n^2}\dd r\\
  \leqslant & C_N\frac{x^{\delta}}{n^{2N+\frac{2}{3}-\delta}}\int_0^{\infty}\frac{r^{\frac{2}{3}+\delta}}{r^2+1}\dd r \\
  \leqslant & C_{N,\delta} \frac{x^{\delta}}{n^{2N+\frac{2}{3}-\delta}}\cdot
\end{align*}
Therefore, collecting the above two estimates yields to  the last point of  the proposition as desired.
Hence, the proof is completed.
\end{proof}

\noindent As an immediate  consequence of Proposition \ref{prop:asypt} and \eqref{def:varphi-n}, \eqref{eq:lamb-n}, we have the following results on the asymptotic representation of the universal function $\phi_n$ and the spectrum.
\begin{corollary}\label{cor:asym}
Let $b\in(0,1)$ and $\delta\in [0,\frac{1}{3})$. Then, for any $n\geqslant 1$ and $N\geqslant 0$, the following statements hold true.
\begin{enumerate}[(1)]
\item
We have
\begin{align*}
  \forall x> 0,\quad \phi_n(x)=\sum_{k=0}^{N}\frac{1}{n^{2k+1}}\Psi_k(\tfrac{x}{n})+\tfrac{1}{n}g_{n,N}(\tfrac{x}{n}),
\end{align*}
with
\begin{align*}
  \big|\tfrac{1}{n}g_{n,N}(\tfrac{x}{n})\big|\leqslant \frac{C_{N,\delta}}{n^{2N+\frac{5}{3}}} \frac{x^\delta}{1+\frac{x}{n}}.
\end{align*}
\item We have
\begin{align*}
  \lambda_{n,b} = 2\sum_{k=0}^{N}\frac{1}{n^{2k+1}}\int_0^\infty \Psi_k(\tfrac{b x}{n}) \tfrac{\dd \mu(x)}{x}
  +\varepsilon_{n,N},
\end{align*}
with
\begin{align*}
  |\varepsilon_{n,N}|\leqslant \frac{C_{N,\delta}}{n^{2N+\frac{5}{3}}}\int_0^\infty \frac{x^{\delta-1}}{1+\frac{b x}{n}}
  \dd \mu(x).
\end{align*}
\end{enumerate}
\end{corollary}

\subsection{The decay rate of $\phi_n-\phi_{n+1}$}\label{subsec:decay-phi}
Our main goal in  this section is to provide  an explicit lower/upper bound for $\chi_n \triangleq \phi_n-\phi_{n+1}$,
which is useful in handling the perturbative argument employed  in the proof of Theorem \ref{thm:perturbative}.
\\[1mm]
According to the differential equation \eqref{d-ff-2}, the function $\phi_{n+1}$ contributes on the source term, and  thus
we shall need  some pointwise controls for $\phi_{n+1}$ or $\varphi_{n+1}$ in order to estimate $\chi_n$.
We note that, by  choosing $N=0$ in Proposition \ref{prop:asypt}, we obtain
\begin{align*}
  \forall n\geqslant 1,x>0,\quad \varphi_n(x)=\frac{x}{1+x^2}+g_{n,0}(x),
\end{align*}
with $\displaystyle{\lim_{n\to \infty}\lVert g_{n,0}\rVert_{L^{\infty}(\mathbb{R}^+)} =0}$.
Thus, for sufficiently large $n$, $\varphi_n(x)$ remains  close to $\frac{x}{1+x^2}$.
We will see that  by analyzing  carefully the differential equation governing  $\varphi_n(x)-\frac{x}{1+x^2}$,
we can show the following lower/upper bound of $\varphi_n(x)$, which gives a more precise version of that result.
\begin{lemma}\label{lem:bound-varphi}
For every $x>0$ and $n\geqslant 1$, the following inequalities hold true
\begin{align*}
  \frac{4n^2}{4n^2 +1} \frac{x}{1+x^2}\leqslant \varphi_n(x)\leqslant \frac{4n^2}{4n^2 -1} \frac{x}{1+x^2}
\end{align*}
and
\begin{align*}
  \frac{4n^2}{4 n^2 +1} \frac{x}{n^2+x^2} \leqslant \phi_n(x)\leqslant \frac{4n^2}{4n^2 -1} \frac{x}{n^2 + x^2}\cdot
\end{align*}
\end{lemma}

\begin{remark}
One may expect that there exist some constants $c_1,c_2>0$ such that at least for sufficiently large $n$,
\begin{align*}
  \frac{c_1}{n^2}\Psi_1(x)\leqslant \varphi_n(x)-\frac{x}{1+x^2}\leqslant \frac{c_2}{n^2}\Psi_1(x),\quad x>0.
\end{align*}
However, numerical experiments indicate that such inequalities do not hold even for very small $c_1$ or very large $c_2$.
On the other hand, if we define
\begin{align*}
  w_n(x) \triangleq \varphi_n(x)-\frac{x}{1+x^2}-\frac{c}{n^2}\Psi_1(x),\quad c>0,
\end{align*}
we have
\begin{align*}
  \tfrac{1}{n^2}\big(w_n''(x)+\tfrac{1}{x}w_n'(x)\big) - 4(1+\tfrac{1}{x^2})w_n(x)
  =\tfrac{c-1}{n^2}\tfrac{(x^2-3)^2-8}{(1+x^2)^3x} - \tfrac{c}{n^4}\big( \Psi_1''(x)+\tfrac{1}{x}\Psi_1'(x)\big).
\end{align*}
The leading term on the right-hand side
$\frac{c-1}{n^2} \frac{(x^2-3)^2-8}{(1+x^2)^3x}= \frac{4(c-1)}{n^2} \big(1+\tfrac{1}{x^2} \big)\Psi_1(x)$
does not have a definite sign, so that the comparision test seen in Lemma \ref{lem:comparison}  does not apply in this context.
\end{remark}

\begin{proof}[Proof of Lemma \ref{lem:bound-varphi}]
Define
\begin{align*}
  f_n(x) \triangleq \varphi_n(x)-\frac{c\, x}{1+x^2},
\end{align*}
where $c>0$ is a constant that will be chosen later according to $n$.
From \eqref{eq:varphi_n} and \eqref{eq:phi-n0} we deduce by straightforward computations that
\begin{align*}
  \frac{1}{n^2}\Big(f_n''(x)+\tfrac{1}{x}f'_n(x)\Big) - 4\Big(1+\tfrac{1}{x^2}\Big) f_n
  = -(1-c) \frac{4}{x} - \frac{c}{n^2} \frac{(x^2-3)^2-8}{(1+x^2)^3x},
\end{align*}
with
\begin{align*}
  f_n(0)=0,\quad 
  \quad \textrm{and}\quad \lim_{x\to \infty}f_n(x)=0.
\end{align*}
Next, we shall use the following bounds
\begin{align*}
  \forall x\geqslant 0,\quad \quad -1\leqslant-6\frac{x^2}{(1+x^2)^3}\leqslant\frac{(x^2-3)^2-8}{(1+x^2)^3}\leqslant \frac{x^4+1}{(1+x^2)^2}\leqslant 1
\end{align*}
leading to
\begin{align*}
  -\frac{c}{n^2}\frac{1}{x} \leqslant -\frac{c}{n^2}\frac{(x^2-3)^2-8}{(1+x^2)^3x}\leqslant \frac{ c}{n^2} \frac{1}{x}
\end{align*}
and
\begin{align*}
  \big(c-1 - \tfrac{c}{4n^2} \big) \frac{4}{x}  \leqslant -(1-c) \frac{4}{x} - \frac{c}{n^2} \frac{(x^2-3)^2-8}{(1+x^2)^3x}
  \leqslant -\big(1- c - \tfrac{c}{4n^2} \big) \frac{4}{x}\cdot
\end{align*}
Hence, by choosing $c=\frac{4n^2}{4n^2 +1}$ and $c=\frac{4 n^2 }{4n^2 -1}$  and applying Lemma \ref{lem:comparison},
we obtain the lower and upper bounds for $\varphi_n(x)$, respectively.
Combined with the relation \eqref{def:varphi-n}, it gives the required lower/upper bounds for  $\phi_n(x)$.
\end{proof}

\noindent The next goal is to estimate the difference $\phi_n -\phi_{n+1}$ that will be used later to explore the spectrum distribution.
\begin{proposition}\label{prop:chi-n}
For every $x>0$ and $n\geqslant 1$, we have
\begin{equation}\label{eq:bound-diff-phi-n}
\begin{aligned}
  \frac{1}{2}\frac{(2n+1)x}{(n^2+x^2) \big((n+1)^2+x^2\big)}\leqslant \phi_{n}(x) - \phi_{n+1}(x)\leqslant
  4 \frac{(2n+1)x}{(n^2+x^2) \big((n+1)^2+x^2\big)}\cdot
\end{aligned}
\end{equation}
\end{proposition}

\begin{proof}[Proof of Proposition $ \ref{prop:chi-n}$]
We shall first prove the following result: for any $x>0$, $n\geqslant 1$, we have
\begin{align}\label{low-up bound}
  \frac{4(n+1)^2}{4(n+1)^2 +1} r_n(x) \leqslant \phi_{n}(x) - \phi_{n+1}(x) \leqslant \frac{4(n+1)^2}{4(n+1)^2 -1}  r_n(x),
\end{align}
where $r_n(x)$ is a  solution to the equation
\begin{align}\label{eq:rn}
  r_n''(x)+ \tfrac{1}{x} r'_n(x)-4(1+\tfrac{n^2}{x^2})r_n(x)=-\tfrac{4(2n+1)}{x\big((n+1)^2+x^2\big)},
\end{align}
supplemented with the boundary conditions
\begin{align}\label{eq:rn-bdr}
  r_n(0)=0,\quad\textrm{and}\quad \lim_{x\rightarrow \infty} r_n(x) = 0.
\end{align}
The construction of  $r_n$ can be done using Hankel transform. Indeed, applying \eqref{hankel:diff} to \eqref{eq:rn} yields
\begin{align*}
  -s^2\mathcal{H}_{2n}r_n(s)-4\mathcal{H}_{2n}r_n(s)=-\mathcal{H}_{2n}\Big(\tfrac{4(2n+1)}{x\big((n+1)^2+x^2\big)}\Big).
\end{align*}
Thus,
\begin{align*}
  \mathcal{H}_{2n}r_n(s)=\frac{1}{s^2+4}\mathcal{H}_{2n}\Big(\tfrac{4(2n+1)}{x\big((n+1)^2+x^2\big)}\Big)(s).
\end{align*}
Then in view of \eqref{hankel:inverse}, we find
\begin{align}\label{eq:new_r_n}
  r_n(x)=\mathcal{H}_{2n}\bigg(\frac{1}{s^2+4}\mathcal{H}_{2n}\Big(\tfrac{4(2n+1)}{x\big((n+1)^2+x^2\big)}\Big)(s)\bigg)(x).
\end{align}
By the definition of $\mathcal{H}_{2n}$ in  \eqref{def:HankelTrans} and arguing as for getting  the estimate \eqref{es:H2nFN} we get \eqref{eq:rn-bdr}.
\\[1mm]
Now we define $h_n(x) \triangleq \phi_{n}(x) - \phi_{n+1}(x)  - \frac{4(n+1)^2}{4(n+1)^2 +1} r_n(x)$ and
\begin{align*}
  \mathbf{T}_n f(x)\triangleq f''(x) + \tfrac{1}{x}f'(x) - 4\big(1+\tfrac{n^2}{x^2}\big)f(x) .
\end{align*}
Then thanks to \eqref{d-ff-2}, $h_n(x)$ satisfies
\begin{align*}
  \mathbf{T}_nh_n(x) 
  =-4\frac{2n+1}{x^2}\phi_{n+1}(x)- \frac{4(n+1)^2}{(n+1)^2 +1} \mathbf{T}_nr_n(x),
\end{align*}
with
\begin{align*}
  h_n(0)=0, \quad\textrm{and}\quad \lim_{x\rightarrow \infty} h_n(x) = 0.
\end{align*}
Lemma \ref{lem:bound-varphi} ensures that
\begin{align*}
  - 4 \frac{2n+1}{x^2} \phi_{n+1}(x) \leqslant - \frac{4(n+1)^2}{4(n+1)^2 +1} \frac{4(2n+1)}{x \big( (n+1)^2+x^2\big)}\cdot
\end{align*}
Thus
\begin{align*}
  -4\frac{2n+1}{x^2}\phi_{n+1}(x) - \frac{4(n+1)^2}{4(n+1)^2 +1} \mathbf{T}_nr_n(x) \leqslant 0.
\end{align*}
Taking advantage of Lemma \ref{lem:comparison}, we find that $h_n(x)\geqslant 0$ for every $x>0$ and $n\geqslant 1$,
which leads to the desired inequality
\begin{align*}
  \phi_{n}(x) - \phi_{n+1}(x) \geqslant \frac{4(n+1)^2}{4(n+1)^2 +1} r_n(x).
\end{align*}
Performing a similar argument, that we shall omit here,  one can prove the other inequality
\begin{align*}
  \phi_{n}(x) - \phi_{n+1}(x) \leqslant \frac{4(n+1)^2}{4(n+1)^2 -1} r_n(x).
\end{align*}
This achieves  the proof of \eqref{low-up bound}.
\\[1mm]
Next we shall investigate some lower and upper bound for $r_n$. We shall first deal with the following rescaled function
$R_n(x)\triangleq n\, r_n(nx)$, which satisfies the equation, see \eqref{def:Lf},
\begin{align*}
  \mathbf{L}R_n(x)= \frac{1}{n^2}\big(R_n''(x)+ \tfrac{1}{x} R_n'(x)\big) - 4\big(1+\tfrac{1}{x^2}\big) R_n(x)
  = -\tfrac{4(2n+1)}{x\big((n+1)^2+n^2x^2\big)}\cdot
\end{align*}
Our primary goal is to derive the pointwise lower/upper bound of $R_n(x)$. To this end, we define
\begin{align*}
  H_n(x) \triangleq \frac{x}{1+x^2} \frac{2n+1}{(n+1)^2+n^2x^2 }\cdot
\end{align*}
We plan to show the following result
\begin{align}\label{Es-R-H}
  \frac{16}{25} H_n(x) \leqslant R_n(x) \leqslant \frac{8}{3} H_n(x),
\end{align}
which implies in turn that
\begin{align*}
  \frac{16}{25}  \frac{x}{n^2+x^2} \frac{2n+1}{(n+1)^2+x^2}
  \leqslant r_n(x) \leqslant \frac{8}{3}  \frac{x}{n^2+x^2} \frac{2n+1}{(n+1)^2+x^2}\cdot
\end{align*}
Then combining this estimate with \eqref{low-up bound} gives  the desired result of Proposition \ref{prop:chi-n}.
\\[1mm]
Now, let us move to the proof of \eqref{Es-R-H}.
Set
\begin{align*}
  \widetilde{f}_n(x) \triangleq R_n(x)-c\,H_n(x),
\end{align*}
with some constant $c>0$ that will be carefully chosen later.
Then, direct computations, using the notation \eqref{def:L0f}, imply
\begin{align*}
  \mathbf{L}\widetilde{f}_n(x)=-(1-c)\frac{4(2n+1)}{x \big((n+1)^2+n^2x^2\big)} - c\frac{1}{n^2} \mathbf{L}_0 H_n(x),
\end{align*}
and we note that
\begin{align*}
  \widetilde{f}_n (0) = 0, \quad \textrm{and}\quad \lim_{x\rightarrow \infty} \widetilde{f}_n(x) = 0.
\end{align*}
According to Lemma \ref{lem:comparison}, in order to obtain that $\widetilde{f}_n(x)\geqslant 0$ or $\widetilde{f}_n (x) \leqslant 0$,
we only need to let the right-hand side of above equation be non-positive or non-negative.
Next, we plan to compute
\begin{align*}
  \frac{1}{n^2}\mathbf{L}_0 H_n(x) = \frac{1}{n^2} \big(H_n''(x) + \frac{1}{x} H_n'(x) \big).
\end{align*}
From straightforward computations we get  for every $x>0$ and $n\geqslant 1$,
\begin{align}\label{Hprime}
  \frac{1}{n^2} H_n'(x) = &  \frac{1}{n^2}\frac{1- x^2}{(1+x^2)^2} \frac{2n+1}{(n+1)^2+n^2x^2 }
  - \frac{x}{1+x^2} \frac{(2n+1)\, 2x}{\big((n+1)^2+n^2x^2\big)^2} \\
  \nonumber \geqslant & - \frac{2n + 1}{(n+1)^2 + n^2 x^2} \bigg( \frac{x^2}{(1+x^2)^2} + \frac{2x^2}{(1+x^2)}
  \frac{1}{(n+1)^2 + n^2 x^2}\bigg) \\
 \nonumber \geqslant & - \frac{3}{4}\,\frac{2n + 1}{(n+1)^2 + n^2 x^2}\cdot
\end{align}
Direct computations yield
\begin{align}\label{H-second}
  \frac{1}{n^2}H_n''(x) = & \frac{1}{n^2} \frac{(-2x)(3-x^2)}{(1+x^2)^3}\frac{2n+1}{(n+1)^2+n^2x^2}
  - \frac{1-x^2}{(1+x^2)^2}\frac{(2n+1)2x}{\big((n+1)^2+n^2x^2\big)^2} \\
  \nonumber& - \frac{x}{(1+x^2)^2}\frac{4(2n+1)}{\big((n+1)^2+n^2x^2\big)^2}
  + \frac{x^3}{1+x^2}\frac{2(2n+1)4 n^2}{\big((n+1)^2+n^2x^2\big)^3}.
\end{align}
It follows that
\begin{align*}
  \frac{1}{n^2}H_n''(x)
  \geqslant & - \frac{4(2n+1)}{x \big((n+1)^2 + n^2 x^2\big)}
  \bigg( \frac{3 x^2 }{2 n^2(1+x^2)^3}  + \frac{ 3x^2 }{2(1+x^2)^2} \frac{1}{(n+1)^2 + n^2 x^2} \bigg) \\
  \geqslant &  -\frac{3}{8} \frac{4(2n+1)}{x\big((n+1)^2 + n^2 x^2\big)} ,
\end{align*}
where  we have used the inequalities $\frac{x^2}{(1+x^2)^2} \leqslant \frac{1}{4}$
and $\frac{ x^2}{(1+x^2)^3} \leqslant \frac{4}{27}$.
Hence we find
\begin{align*}
  \mathbf{L} \widetilde{f}_n(x) \leqslant \frac{4(2 n+1)}{x\big((n+1)^2 + n^2 x^2\big)} \big( -(1-c) + \tfrac{9c}{16}\big),
\end{align*}
and choosing $c =  \frac{16}{25}$ gives $\mathbf{L}\widetilde{f}_n \leqslant 0 $.
Then, Lemma \ref{lem:comparison} implies that $\widetilde{f}_n(x) \geqslant 0$,
that is,
\begin{align*}
  \forall\, x\geqslant 0,\quad R_n(x)\geqslant \frac{16}{25} H_n(x).
\end{align*}
Now, we move to the proof  of the second estimate of \eqref{Es-R-H}. First, we observe from \eqref{Hprime} that
\begin{align*}
  \frac{1}{n^2} H_n'(x) \leqslant  \frac{1}{4}\,\frac{4(2n + 1)}{(n+1)^2 + n^2 x^2}\cdot
\end{align*}
In addition, we deduce from \eqref{H-second}
\begin{align*}
  \frac{1}{n^2} H_n''(x) & \leqslant \frac{4(2n+1)}{x \big((n+1)^2 + n^2 x^2\big)}
  \bigg( \frac{ x^4 }{2 n^2(1+x^2)^3} +  \frac{ x^4 }{2(1+x^2)^2} \frac{1}{(n+1)^2 + n^2 x^2}\bigg) \\
  & \quad + \frac{4(2n+1)}{x \big((n+1)^2 + n^2 x^2\big)}
  \times \frac{ 2 x^2 }{1+x^2} \frac{n^2 x^2}{\big((n+1)^2 + n^2 x^2\big)^2} \\
  & \leqslant \frac{3}{8} \frac{4(2n+1)}{x \big((n+1)^2 + n^2 x^2\big)} ,
\end{align*}
where we have used the fact that $\frac{x^4}{(1+ x^2)^3} \leqslant \frac{4}{27}$ and
$\frac{n^2 x^2}{((n+1)^2 + n^2 x^2)^2} \leqslant \frac{1}{16}$.
Then we have
\begin{align*}
  \mathbf{L}\widetilde{f}_n(x) \geqslant  \frac{4(2 n+1)}{x\big((n+1)^2 + n^2 x^2\big)}
  \big( - (1-c) - \tfrac{5c}{8} \big).
\end{align*}
Choosing $c= \frac{8}{3}$ guarantees that $\mathbf{L}\widetilde{f}_n \geqslant0$.
Therefore, we conclude in view of Lemma \ref{lem:comparison} that
$\widetilde{f}_n(x) \leqslant 0$, that is,
\begin{align*}
  R_n(x)\leqslant \frac{8}{3} H_n(x),
\end{align*}
as stated in \eqref{Es-R-H}. This achieves the proof of Proposition \ref{prop:chi-n}.
\end{proof}

\subsection{Spectrum convexity}\label{subsec:lam-conv}
In the forthcoming lemma, we intend  to discuss a  result concerning the convexity of spectrum $(\lambda_{n,b})_{n\in\mathbb{N}^\star}$
associated with a class of measures $\mu$ with suitable densities. Our result reads as follows.
\begin{lemma}\label{lem:convex-lem}
  Let $\lambda_{n,b}$ be given by \eqref{eq:lamb-n} with $\dd \mu(x)=x f(x)\dd x$,
$f(x)\geqslant 0$ and $f\in C^2(\R^+)$.
If there exists some constant $C>0$ such that the following conditions hold
\begin{enumerate}[(1)]
\item $\displaystyle{\limsup_{x\to 0^+}x |f(x)| +\limsup_{x\to 0^+}x^2 |f'(x)|\leqslant C}$,
\item $\displaystyle{\lim_{x\to +\infty}f(x)=0}$ and $\displaystyle{\lim_{x\to +\infty}x f'(x)=0}$,
\item $\forall x>0,\quad f''(x)\geqslant 0$,
\end{enumerate}
then we have
\begin{align*}
  \forall\, n\geqslant 2,\quad \lambda_{n+1,b}+\lambda_{n-1,b}-2\lambda_{n,b}\geqslant 0.
\end{align*}
\end{lemma}

\begin{proof}[Proof of Lemma \ref{lem:convex-lem}]
First, recalling that $\phi_n$ is defined by \eqref{def:phi-n} and using the fact that
\begin{align*}
  \forall \eta\in\R,\quad -e^{i2\eta} + 2-e^{-i2\eta} = 4\sin^2 \eta,
\end{align*}
we get  the following identity,
\begin{align}\label{eq:convex-id-phi-n}
 \forall n\geqslant 2,\quad  \frac{\dd^2 \phi_n(x)}{\dd x^2}=-\phi_{n+1}(x) + 2\phi_n(x) - \phi_{n-1}(x).
\end{align}
Combining together \eqref{eq:lamb-n} with  \eqref{eq:convex-id-phi-n} allows to get
\begin{align*}
  \lambda_{n+1,b}+\lambda_{n-1,b}-2\lambda_{n,b} = & 2 \int_0^{\infty} \big(\phi_{n+1}(bx) + \phi_{n-1}(bx)
  - 2 \phi_n(b x) \big) f(x) \dd x \\
  = & - 2 \int_0^{\infty}\phi_n''(b x)f(x)\dd x \\
  = & - \frac{2}{b^2} \int_0^{\infty} \big(\phi_n(b x) - \phi_n'(0)b x \big)''f(x) \dd x.
\end{align*}
Integration by parts, using the above assumptions on $f$ and the fact that
$|\phi_n'(x)| \leqslant C$,
\begin{align*}
  \forall n\geqslant 2,\quad \phi''_n(0)=&\int_0^{\pi}4(\sin \eta)^2e^{i2n\eta}\dd \eta
  = 2\int_0^{\pi}(1-\cos 2\eta)\cos (2n\eta) \dd \eta  = 0,
\end{align*}
we obtain
\begin{align*}
  \lambda_{n+1,b}+\lambda_{n-1,b}-2\lambda_{n,b}
  =&\, \frac{2}{b^2}\int_0^{\infty}\big(\phi_n(bx)-\phi_n'(0)bx \big)' f'(x) \dd x \\
  =&- \frac{2}{b^2} \int_0^{\infty} \big(\phi_n(bx)-\phi_n'(0) b x \big) f''(x)\dd x.
\end{align*}
Now, define  the function
\begin{align*}
  \widetilde{h}_n(x)\triangleq \phi_n'(0)x -\phi_n(x) = \frac{4x }{4n^2-1} -\phi_n(x).
\end{align*}
We intend to  prove $\widetilde{h}_n(x)\geqslant 0$ for any $x\geqslant 0$,
which implies in turn the desired result of \mbox{Lemma \ref{lem:convex-lem}.}
By straightforward computations we find
\begin{align*}
  \widetilde{h}_n''(x) + \tfrac{1}{x} \widetilde{h}_n'(x) - 4 \big(1+\tfrac{n^2}{x^2}\big) \widetilde{h}_n(x)
  = - \tfrac{16x}{4n^2-1} \leqslant 0,
\end{align*}
and
\begin{align*}
  \widetilde{h}_n(0) = \widetilde{h}_n'(0)=0,\quad \textrm{and}\quad \quad \lim_{x\to \infty}\widetilde{h}_n(x)= \infty.
\end{align*}
Applying Lemma \ref{lem:comparison} implies
\begin{align*}
  \forall\, x\geqslant 0,\qquad \widetilde{h}_n(x) \geqslant 0.
\end{align*}
This concludes the proof of the positivity of $\widetilde{h}_n(x)$ and achieves the desired result.
\end{proof}

\section{Proof of the main theorems}\label{sec:main}

In this section, we will apply Crandall-Rabinowitz's theorem to prove the existence of time-periodic solution
for the active scalar equation \eqref{eq:ASE}-\eqref{eq:psi}.
We consider the kernel $K(\mathbf{x},\mathbf{y}) = K_0 (|\mathbf{x}-\mathbf{y}|) + K_1 (\mathbf{x},\mathbf{y})$,
and if $K_1\equiv 0$, it corresponds to the case treated in Theorem \ref{thm:main},
and if $K_1 \not\equiv 0$, it is the case studied in Theorem \ref{thm:perturbative}.
Below, we always identify the complex plane $\C$ with $\mathbb{R}^2$.
\\[1mm]
Before proceeding with  the  proofs, we collect some useful facts in polar coordinates.
Denote by
\begin{align}\label{def:G}
  G_1(\rho_1,\theta,\rho_2,\eta) \triangleq K_1(\rho_1e^{i\theta},\rho_2e^{i\eta}),
  \quad G(\rho_1,\theta,\rho_2,\eta) \triangleq K(\rho_1e^{i\theta},\rho_2e^{i\eta}),
\end{align}
then thanks to ($\mathbf{A}4$) we have
\begin{equation}\label{eq:G1-prop1}
\begin{split}
  G_1(\rho_1,-\theta,\rho_2,-\eta) & = G_1(\rho_1,\theta,\rho_2,\eta),\quad \\
  G_1(\rho_1,\theta+\theta',\rho_2,\eta+\theta') & = G_1(\rho_1,\theta,\rho_2,\eta),\quad
  \forall \theta'\in \mathbb{R}.
\end{split}
\end{equation}
Hence, we get in particular
\begin{align}\label{eq:G1-parity}
  G_1(\rho_1,0,\rho_2,-\eta)=G_1(\rho_1,0,\rho_2,\eta),
\end{align}
and differentiating at $\theta^\prime=0$ the second identity  in \eqref{eq:G1-prop1} yields
\begin{equation}\label{eq:G-deriv}
\begin{split}
  \partial_\theta G_1(\rho_1,\theta,\rho_2,\eta) & =-\partial_\eta G_1(\rho_1,\theta,\rho_2,\eta), \\
  \quad \partial_\theta G(\rho_1,\theta,\rho_2,\eta) &=-\partial_\eta G(\rho_1,\theta,\rho_2,\eta).
\end{split}
\end{equation}
By setting $\mathbf{x}= \rho_1 e^{i\theta}$ and $\mathbf{y} = \rho_2 e^{i\eta}$,
we get from straightforward computations
\begin{align}\label{eq:K-der}
		\nabla_{\mathbf{x}} K(\mathbf{x},\mathbf{y})=\begin{pmatrix}
		\partial_{\rho_1}G\cos \theta - \partial_{\theta}G \frac{\sin \theta}{\rho_1}\\
		\partial_{\rho_1}G\sin \theta + \partial_{\theta}G  \frac{\cos \theta}{\rho_1}
	\end{pmatrix},
\end{align}
\begin{align}\label{eq:K-der2}
     \nabla_{\mathbf{y}} K (\mathbf{x},\mathbf{y})=\begin{pmatrix}
		\partial_{\rho_2}G \cos \eta - \partial_{\eta}G \frac{\sin \eta}{\rho_2}\\
		\partial_{\rho_2}G\sin \eta+\partial_{\eta}G  \frac{\cos \eta}{\rho_2}
	\end{pmatrix}.
\end{align}
We note that the  above identities  hold true when  $(K,G)$ is replaced by $(K_1,G_1)$.
\\[1mm]
Next, we shall introduce the function spaces that will be used in  the bifurcation arguments.
For $m\in \mathbb{N}^\star$ and $\alpha\in(0,1)$, we define
\begin{align}\label{def:Xm}
  \mathbf{X}=\mathbf{X}_m\triangleq
  \Big\{ f\in C^{2-\alpha}(\T): f(\theta)=\sum_{n\geqslant 1}b_n \cos (nm\theta), b_n\in \R, \theta\in \T \Big\},
\end{align}
and
\begin{align}\label{def:Ym}
  \mathbf{Y}= \mathbf{Y}_m
  \triangleq \Big\{ f\in C^{1-\alpha}(\T): f(\theta)=\sum_{n\geqslant 1}b_n \sin (nm\theta), b_n\in \R, \theta\in \T \Big\},
\end{align}
equipped with the usual norms. For $\epsilon_0>0$, we denote by $\mathbf{B}_{\epsilon_0}$ the open ball of $\mathbf{X}_m$  centered \mbox{at $0$} and of  radius $\epsilon_0$, that is,
\begin{align*}
  \mathbf{B}_{\epsilon_0}\triangleq \big\{f\in \mathbf{X}_m:\lVert f\rVert_{\mathbf{X}_m}<\epsilon_0\big\}.
\end{align*}

\subsection{Strong regularity}\label{subsec:reg}
This aim of this part is to explore the strong regularity of the functional $F$ described by \eqref{eq:F}.
We have the following result.
\begin{proposition}\label{prop:regularity-F1}
Let $m\geqslant 1$, $\alpha\in(0,1)$ and $\mathbf{X}_m$ and $\mathbf{Y}_m$ the spaces  given by \eqref{def:Xm}-\eqref{def:Ym}.
There \mbox{exists  $\epsilon_0>0$} small enough such that the following statements hold true.
\begin{enumerate}[(1)]
\item[(1)] $F:\mathbb{R}\times \mathbf{B}_{\epsilon_0}\to \mathbf{Y}_m$ is well-defined.
\item[(2)] $F:\mathbb{R}\times \mathbf{B}_{\epsilon_0}\to \mathbf{Y}_m$ is of class $C^1$.
\item[(3)] The partial derivative $\partial_{\Omega}\partial_r F: \mathbb{R}\times \mathbf{B}_{\epsilon_0}\to
\mathcal{L}(\mathbf{X}_m,\mathbf{Y}_m)$
exists and is continuous.
\end{enumerate}
\end{proposition}

\begin{proof}[Proof of Proposition \ref{prop:regularity-F1}]
{\bf{(1)}} Using Gauss-Green theorem (similarly as deriving \eqref{def:Vr}), we can rewrite $\partial_{\theta}F_0[r]$ as
\begin{align}\label{eq:F_1}
  \partial_{\theta}F_0[r]=(F_{00}[r] + F_{01}[r])\cdot \partial_{\theta}(R(\theta)e^{i\theta})
\end{align}
where 
\begin{align*}
  F_{00}[r](\theta) & = \int_0^{2\pi}\int_0^{R(\eta)}\nabla_{\mathbf{x}} K_0\big(|R(\theta) e^{i\theta} -\rho e^{i\eta}|\big)
  \rho \,\dd\rho \dd\eta \\
  & = \int_{\mathbb{T}} K_0\big(|R(\theta)e^{i\theta}-R(\eta)e^{i\eta}|\big)
  \partial_\eta(R(\eta)e^{i\eta}) \dd \eta,
\end{align*}
and
\begin{align*}
  F_{01}[r](\theta) \triangleq
  \int_0^{2\pi}\int_0^{R(\eta)}\nabla_{\mathbf{x}} K_1(R(\theta) e^{i\theta},\rho e^{i\eta})
  \rho \,\dd\rho \dd\eta.
\end{align*}
Since $\partial_{\theta}(R(\theta)e^{i\theta})= \big(\frac{r'(\theta)}{R(\theta)} e^{i\theta}
+ R(\theta) i e^{i\theta}\big)\in C^{1-\alpha}(\T)$, then from \eqref{eq:F} and \eqref{eq:F_1} and  in order to show
$F(\Omega,r)\in C^{1-\alpha}(\mathbb{T})$, we only need to check that
\begin{align}\label{reg:holder}
  \theta\in \mathbb{T}\mapsto F_{00}[r] , F_{01}[r] \in C^{1-\alpha}(\mathbb{T}).
\end{align}
Next we plan to prove \eqref{reg:holder}.
First, by letting $\epsilon_0>0$ small enough, we have that for every $r\in \mathbf{B}_{\epsilon_0}$
and for every $\theta,\eta\in \mathbb{R}$,
\begin{align}\label{ineq:basic}
  b \big|\sin \tfrac{\theta-\eta}{2}\big|\leqslant |R(\theta)e^{i\theta}-R(\eta)e^{i\eta}|
  \leqslant 3 b \big|\sin \tfrac{\theta-\eta}{2}\big|.
\end{align}
Indeed, this is quite similar to \cite[Eq. (59)]{HXX23}: according to the following estimates
\begin{align*}
  R(\theta)e^{i\theta} - R(\eta)e^{i\eta} = b e^{i\eta} \big(e^{i(\theta -\eta)} - 1\big)
  + \Big(\big(R(\theta)e^{i\theta}
  -b e^{i\theta} \big) - \big(R(\eta)e^{i\eta} -b e^{i\eta} \big) \Big),
\end{align*}
and $|\partial_\theta\big(R(\theta)e^{i\theta}-b e^{i\theta}\big)| \leqslant \frac{2\epsilon_0}{\sqrt{b^2 - 2\epsilon_0}}$,
and
\begin{equation}\label{eq:sin-theta}
\tfrac{2|\theta|}{\pi }\leqslant |e^{i\theta}-1|=|2\sin \tfrac{\theta}{2}|\leqslant |\theta|,\quad\text{ for }\; |\theta|\leqslant \pi,
\end{equation}
we can get \eqref{ineq:basic} for every $|\theta-\eta|\leqslant \pi $ by letting $\epsilon_0>0$
sufficiently small; and the general case follows from the periodicity property.
Now, define
\begin{align*}
  \mathbf{k}_1(\theta,\eta) \triangleq  K_0\big(|R(\theta)e^{i\theta}-R(\eta)e^{i\eta}|\big).
\end{align*}
Using the monotonicity of $K_0(t)$ and \eqref{ineq:basic}, we deduce that
\begin{align*}
  |\mathbf{k}_1(\theta,\theta+\eta)| \leqslant \max\Big\{\big|K_0(b|\sin \tfrac{\eta}{2}|)\big|,
  \big|K_0(3b|\sin \tfrac{\eta}{2}|)\big|\Big\}
  \triangleq H_1\big(|\sin \tfrac{\eta}{2} |\big),
\end{align*}
and from \eqref{cond:K0} (noting that \eqref{cond:K0} implies $\int_0^{a_0} |K_0(t)|\dd t <\infty$) and \eqref{eq:sin-theta} we get
\begin{equation}\label{ineq:K_0-int}
\begin{split}
  \int_{\mathbb{T}} H_1\big(\big|\sin \tfrac{\eta}{2} \big|\big) \dd \eta
  \leqslant & 2\max\Big\{ \int_0^\pi |K_0(\tfrac{b}{\pi}\eta)|\dd \eta,
  \int_0^\pi |K_0(\tfrac{3b}{2}\eta)|\dd \eta \Big\} \\
  \leqslant & \frac{C_0}{b} \int_0^{\frac{3\pi}{2}b} |K_0(t)|\dd t<\infty.
\end{split}
\end{equation}
Noticing that $\partial_{\eta}(R(\eta)e^{i\eta}) = \big(\frac{r'(\eta)}{R(\eta)} e^{i\eta}
+ R(\eta) i e^{i\eta}\big)\in C^{1-\alpha}(\T)$, and the estimate
\begin{align*}
  \Big|\partial_\theta\Big(R(\theta)e^{i\theta}-R(\theta+\eta)e^{i(\theta+\eta)}\Big)\Big|
  \leqslant C |\eta|^{1-\alpha} \leqslant C \pi^{1-\alpha} \big|\sin \tfrac{\eta}{2}\big|^{1-\alpha},
  \quad \forall |\eta|\leqslant \pi,
\end{align*}
combined with the periodic property of $R$ leads to
\begin{align}\label{eq:fact0}
  \Big|\partial_\theta\Big(R(\theta)e^{i\theta}-R(\theta+\eta)e^{i(\theta+\eta)}\Big)\Big|
  \leqslant C \big|\sin \tfrac{\eta}{2}\big|^{1-\alpha},\quad \forall \eta \in\mathbb{R}.
\end{align}
Then, we use \eqref{ineq:basic} to deduce that
\begin{align*}
  \big|\partial_\theta\big(\mathbf{k}_1(\theta,\theta+\eta)\big)\big|\leqslant
  & C \big| K'_0\big(|R(\theta)e^{i\theta}-R(\theta+\eta)e^{i(\theta+\eta)}|\big)\big|
  \Big|\partial_\theta \Big(R(\theta)e^{i\theta}-R(\theta+\eta)e^{i(\theta+\eta)}\Big)\Big| \\
  \leqslant & C \big|K'_0\big(b|\sin \tfrac{\eta}{2}|\big)\big|
  \big|\sin \tfrac{\eta}{2}\big|^{1-\alpha}\triangleq H_2\big(\big|\sin \tfrac{\eta}{2}\big|\big).
\end{align*}
In addition, in view of Lemma \ref{lem:int2} (with $\beta = -\alpha(1- \alpha)$) and \eqref{cond:K0}, \eqref{eq:sin-theta}, we have
\begin{equation}\label{ineq:K-0-int-diff-1}
\begin{split}
  \int_{\mathbb{T}} \big(H_1(|\sin \tfrac{\eta}{2}|) \big)^\alpha
  \big(H_2(|\sin \tfrac{\eta}{2}|) \big)^{1-\alpha} \dd \eta
  & \leqslant C
  \int_0^\pi |K_0(C_1\eta)|^\alpha |K'_0(b \eta)|^{1-\alpha} \eta^{(1-\alpha)^2}\dd \eta  \\
  & \leqslant C \int_0^{(\frac{3b}{2}\vee 1)\pi} |K_0(t)| t^{-\alpha+\alpha^2}\dd t + C < \infty,
\end{split}
\end{equation}
where $C_1$ equals either $\frac{b}{\pi}$ or $\frac{3b}{2}$.
Hence, gathering \eqref{ineq:K_0-int}, \eqref{ineq:K-0-int-diff-1} with Lemma \ref{lem:int-operator} implies
\begin{align}\label{es:F00-Hold}
  \|F_{00}[r]\|_{C^{1-\alpha}}\leqslant C \|\partial_\eta (R(\eta) e^{i\eta})\|_{C^{1-\alpha}} \leqslant C.
\end{align}
\vskip0.5mm
\noindent For the remaining result in \eqref{reg:holder}, using the  assumption ($\mathbf{A}$3),
one can easily show that
\begin{align}\label{es:F01}
  \|F_{01}[r]\|_{C^1(\T)} \leqslant
  C \sup_{\mathbf{x},\mathbf{y}\in B(0,b+\sqrt{2\epsilon_0})} \Big(|\nabla_{\mathbf{x}}K_1(\mathbf{x},\mathbf{y})|+
  |\nabla^2_{\mathbf{x}} K_1(\mathbf{x},\mathbf{y})| \|\partial_\theta(R(\theta)e^{i\theta})\|_{L^\infty} \Big)
  \leqslant C,
\end{align}
which guarantees that $F_{01}[r]$ belongs to
$C^{1-\alpha}(\mathbb{T})$, as desired.
\\[1mm]
Now, we prove that $F(\Omega,r)$ given by \eqref{eq:F} has the series expansion as in $\mathbf{Y}_m$.
Indeed, noting that
under the assumption $(\mathbf{A}4)$, the kernel  $K$ satisfies
\begin{align*}
  K(\bar{\mathbf{x}},\bar{\mathbf{y}}) = K(\mathbf{x},\mathbf{y}),\quad
  K(e^{i\theta} \mathbf{x}, e^{i\theta} \mathbf{y}) = K(\mathbf{x},\mathbf{y}),\;\;\forall \theta\in\mathbb{R},
\end{align*}
we can argue as in  \cite[p. 27]{HXX23} to deduce that
\begin{align*}
  F_0[r](-\theta)= F_0[r](\theta),\quad F_0[r](\theta + \tfrac{2\pi}{m}) = F_0[r](\theta),
\end{align*}
which leads to
\begin{align*}
  F_0[r](\theta) = \sum_{n=0}^\infty c_{nm} \cos(nm \theta),\quad c_{nm}\in \mathbb{R},
\end{align*}
and consequently, $F(\Omega,r)$ has the desired expansion formula.
Therefore, by taking  $\epsilon_0>0$ small enough we conclude that $F(\Omega,r)\in \mathbf{Y}_m$.
\\[1mm]
{\bf{(2)}} It is obvious to see that $\partial_{\Omega}F(\Omega,r)=r'$ is a continuous mapping.
So we only need to show that $\partial_rF(\Omega,r)$ is continuous with respect to $r$.
In view of \eqref{eq:F_Lin_r}-\eqref{def:Vr}, we have
\begin{align}\label{eq:F_r}
  \partial_r F(\Omega,r) h(\theta)
  & = \partial_\theta \Big( \big(\Omega+V[r](\theta)\big) h(\theta)+\mathcal{L}[r](h)(\theta)\Big) \nonumber \\
  & = \partial_\theta \Big( \big(\Omega + V_0[r](\theta) + V_1[r](\theta) \big) h(\theta)+\mathcal{L}[r](h)(\theta)\Big) \nonumber \\
  & = \Omega h'(\theta) +  \partial_\theta\Big( \mathcal{I}_0[r](h) + \mathcal{I}_1[r](h)  \Big),
\end{align}
where
\begin{align*}
  \mathcal{I}_0[r](h)(\theta) & \triangleq \int_{\mathbb{T}} K_0\big(|R(\theta)e^{i\theta}
  - R(\eta)e^{i\eta}|\big) h(\eta) \dd \eta \\
  &\quad +\int_{\mathbb{T}} K_0\big(|R(\theta)e^{i\theta} - R(\eta)e^{i\eta}|\big)
  \big(i\partial_\eta \big(R(\eta)e^{i\eta}\big)\big)
  \cdot \big(\tfrac{h(\theta)}{R(\theta)}e^{i\theta}\big)  \dd \eta,
\end{align*}
and
\begin{align}\label{I1-iedntity}
  \nonumber \mathcal{I}_1[r](h)(\theta) \triangleq& \int_0^{2\pi} \int_0^{R(\eta)}
  \nabla_{\mathbf{x}} K_1 \big(R(\theta)e^{i\theta},  \rho e^{i\eta}\big)\cdot \big(\tfrac{h(\theta)}{R(\theta)}e^{i\theta}\big)
  \rho \dd \rho \dd \eta\\
  & + \int_{\mathbb{T}} K_1(R(\theta)e^{i\theta},R(\eta)e^{i\eta})h(\eta)\dd \eta.
\end{align}
For the term $\partial_\theta\mathcal{I}_0$, by using the notation $\nabla_\mathbf{x} K_0 (|\mathbf{x}-\mathbf{y}|) = K_0'(|\mathbf{x}-\mathbf{y}|) \frac{\mathbf{x} -\mathbf{y}}{|\mathbf{x} -\mathbf{y}|}$,
we decompose it as follows
\begin{align}\label{Pipaa1}
  \nonumber\partial_\theta \mathcal{I}_0[r](h)(\theta)   &=\int_{\mathbb{T}} K_0\big(|R(\theta)e^{i\theta}
  - R(\eta)e^{i\eta}|\big) \big(i\partial_\eta \big(R(\eta)e^{i\eta}\big)\big)
  \cdot \partial_\theta\big(\tfrac{h(\theta)}{R(\theta)}e^{i\theta}\big)  \dd \eta \\
  \nonumber& \quad +\partial_\theta\big(R(\theta)e^{i\theta}\big)
  \cdot\int_{\mathbb{T}}\nabla_{\mathbf{x}} K_0\big(|R(\theta)e^{i\theta} - R(\eta)e^{i\eta}|\big)
  \mathbf{w}(\theta,\eta)  \dd \eta\\
  &\triangleq \mathcal{I}_{00}[r](h)(\theta)+ \mathcal{I}_{01}[r](h)(\theta),
\end{align}
with
\begin{align*}
  \mathbf{w}(\theta,\eta)&=h(\eta)+\big(i\partial_\eta \big(R(\eta)e^{i\eta}\big)\big)
 \cdot \big(\tfrac{h(\theta)}{R(\theta)}e^{i\theta}\big)\\
 &=h(\eta)+\tfrac{h(\theta)}{R(\theta)}\partial_\eta \big(R(\eta)\sin(\theta-\eta)\big)\\
 &=\underbrace{h(\eta)-\tfrac{h(\theta)}{R(\theta)} R(\eta)\cos(\theta-\eta)}_{\triangleq \mathbf{w}_2(\theta,\eta)}+\underbrace{\tfrac{h(\theta)}{R(\theta)}\tfrac{r^\prime(\eta)}{R(\eta)}\sin(\theta-\eta)}_{\triangleq\mathbf{w}_3(\theta,\eta)}.
\end{align*}
The estimate of $\mathcal{I}_{00}[r](h)(\theta)$ is similar to that of $F_{00}[r]$ in \eqref{es:F00-Hold}. Actually, using the product laws in $C^{1-\alpha}(\T)$ we have
\begin{align}\label{ineq:bound-I-01}
  \lVert \mathcal{I}_{00}[r](h)\rVert_{C^{1-\alpha}}\leqslant C \big\|\partial_\eta(R(\eta)e^{i\eta})\big\|_{C^{1-\alpha}}
  \big\|\partial_\theta\big(\tfrac{h(\theta)}{R(\theta)} e^{i\theta}\big)\big\|_{C^{1-\alpha}}
  \leqslant C\lVert h\rVert_{C^{2-\alpha}}.
\end{align}
For $\mathcal{I}_{01}[r](h)$, since $\theta\in\T\mapsto \partial_\theta \big(R(\theta)e^{i\theta} \big) \in C^{1-\alpha}(\T)$,
then  using the product laws we get
\begin{equation}\label{P-loi}
\begin{split}
  \|\mathcal{I}_{01}[r](h)\|_{C^{1-\alpha}} &\leqslant C \Big\|\underbrace{\int_{\T} \nabla_{\mathbf{x}}K_0 \big(|R(\theta) e^{i\theta}
  - R(\eta)e^{i\eta}|\big) \,\mathbf{w}_2(\theta,\eta) \dd \eta}_{{\triangleq \mathcal{I}_{02}[r](h)}}  \Big\|_{C^{1-\alpha}}\\
  &\quad+C \Big\|\underbrace{\int_{\T} \nabla_{\mathbf{x}}K_0 \big(|R(\theta) e^{i\theta}
  - R(\eta)e^{i\eta}|\big) \,\mathbf{w}_3(\theta,\eta) \dd \eta}_{\triangleq \mathcal{I}_{03}[r](h)}  \Big\|_{C^{1-\alpha}}.
\end{split}
\end{equation}
We define
\begin{align*}
  \mathbf{k}_2(\theta,\eta)  & \triangleq
  \nabla_{\mathbf{x}}K_0\big(|R(\theta)e^{i\theta}-R(\eta)e^{i\eta}|\big) \mathbf{w}_2(\theta,\eta) \\
  & = K_0' \big(|R(\theta)e^{i\theta} - R(\eta) e^{i\eta}|\big)
  \tfrac{R(\theta)e^{i\theta} - R(\eta) e^{i \eta}}{|R(\theta)e^{i\theta} - R(\eta) e^{i\eta}|}
  \mathbf{w}_2(\theta,\eta) .
\end{align*}
Notice that
\begin{align*}
  \mathbf{w}_2(\theta,\theta+\eta)=h(\theta+\eta)-h(\theta)+h(\theta)\tfrac{R(\theta)-R(\theta+\eta)\cos \eta}{R(\theta)},
\end{align*}
and
\begin{align*}
  \partial_\theta\big(\mathbf{w}_2(\theta,\theta +\eta)\big)
  & = h'(\theta+\eta)-h'(\theta) + \partial_\theta\Big(\tfrac{h(\theta)}{R(\theta)}\Big)
  \big(R(\theta) - R(\theta+\eta) \cos\eta \big) \\
  & \quad + \tfrac{h(\theta)}{R(\theta)} \Big(\tfrac{r'(\theta)}{R(\theta)}
  -\tfrac{r'(\theta+\eta)\cos\eta}{R(\theta+\eta)} \Big).
\end{align*}
Arguing as for \eqref{eq:fact0}, we deduce that
\begin{align}\label{ineq:w-0-th}
  |\mathbf{w}_2(\theta,\theta+\eta)|\leqslant C \lVert h\rVert_{C^1}\big|\sin \tfrac{\eta}{2}\big|
\end{align}
and
\begin{align}\label{ineq:w-1-th}
  \big|\partial_\theta\big(\mathbf{w}_2(\theta,\theta+\eta)\big)\big|
  \leqslant C\lVert h\rVert_{C^{2-\alpha}} \big|\sin \tfrac{\eta}{2}\big|^{1-\alpha}.
\end{align}
Thanks to the non-increasing property of $|K_0'|$ and \eqref{ineq:basic},
we deduce that
\begin{align}\label{est:k_2}
  |\mathbf{k}_2(\theta,\theta+\eta)|\leqslant C \big|K'_0\big(\tfrac{b}{2}|\sin \tfrac{\eta}{2}|\big)\big| |\sin\tfrac{\eta}{2}|\lVert h\rVert_{C^1}
  \triangleq C H_3\big(|\sin \tfrac{\eta}{2}|\big) \lVert h\rVert_{C^1}.
\end{align}
Applying the estimates \eqref{eq:int-fk-fact} and \eqref{es:f-weig} allows to get
\begin{equation}\label{ineq:K_0-diff-1-1}
\begin{aligned}
  \int_{\T} H_3\big(|\sin \tfrac{\eta}{2}|\big)\dd \eta
  \leqslant  \int_0^\pi |K'_0(\tfrac{b}{2\pi} \eta)| \eta \,\dd\eta
  \leqslant  C \int_0^\pi |K_0(t)| \dd t + C .
\end{aligned}
\end{equation}
By using \eqref{ineq:basic}, \eqref{eq:fact0}, \eqref{ineq:w-0-th}-\eqref{ineq:w-1-th} and the non-increasing property of $|K_0'|$, $|K_0''|$, together with Lemma \ref{lem-high-deriv}-(1),
we infer that
\begin{align}\label{eq:par-k2}
  \big|\partial_\theta\big(\mathbf{k}_2(\theta,\theta+\eta)\big)\big|
  \leqslant &\,C\Big(|K'_0(b|\sin \tfrac{\eta}{2}|)| \big|\sin \tfrac{\eta}{2}\big|^{-\alpha}
  +|K''_0(b |\sin \tfrac{\eta}{2}|)|  \big|\sin \tfrac{\eta}{2}\big|^{1-\alpha}\Big)
  \lVert h\rVert_{C^1}\big|\sin \tfrac{\eta}{2}\big| \nonumber \\
  &+C\lVert h\rVert_{C^{2-\alpha}}|K'_0(b|\sin \tfrac{\eta}{2})| \big |\sin \tfrac{\eta}{2}\big|^{1-\alpha} \nonumber\\
  \leqslant &\,C\lVert h\rVert_{C^{2-\alpha}}\Big(|K'_0(b|\sin \tfrac{\eta}{2}|)|
  +|K''_0(b|\sin \tfrac{\eta}{2}|)| \big|\sin\tfrac{\eta}{2}\big|\Big)|\sin \tfrac{\eta}{2}|^{1-\alpha} \nonumber \\
  \leqslant& C\lVert h\rVert_{C^{2-\alpha}}|K'_0(\tfrac12b|\sin \tfrac{\eta}{2}|)||\sin \tfrac{\eta}{2}|^{1-\alpha}\triangleq  C\lVert h\rVert_{C^{2-\alpha}} H_4\big(|\sin \tfrac{\eta}{2}|\big).
\end{align}
Lemma \ref{lem:int2}, estimates \eqref{es:f-weig}, \eqref{eq:int-fk-fact} and assumption $(\mathbf{A}2)$ ensure that
\begin{align}\label{ineq:K_0-diff-1-2}
  \int_{\T} \big(H_3(|\sin \tfrac{\eta}{2}|)\big)^\alpha
  \big(H_4(|\sin \tfrac{\eta}{2}|)\big)^{1-\alpha}\dd \eta
  &\leqslant C \int_{\T} \big|K_0'(\tfrac{b}{2} |\sin\tfrac{\eta}{2}|) \big|
  \, |\sin \tfrac{\eta}{2}|^{\alpha   + (1-\alpha)^2} \dd \eta \nonumber \\
  & \leqslant C \int_0^\pi \big|K_0'(\tfrac{b}{2\pi} \eta) \big|
  \, \eta^{1 -\alpha   + \alpha^2} \dd \eta
  \nonumber \\
  & \leqslant C \int_0^\pi |K_0(t)|t^{-\alpha+\alpha^2}\dd t + C .
\end{align}
Hence according to \eqref{ineq:K_0-diff-1-1}, \eqref{ineq:K_0-diff-1-2} and Lemma \ref{lem:int-operator}, we find
\begin{align}\label{es:J1}
  \Big\| \mathcal{I}_{02}[r](h)  \Big\|_{C^{1-\alpha}}
  \leqslant C \lVert h\rVert_{C^{2-\alpha}}.
\end{align}
For $\mathcal{I}_{03}[r](h)$, we set
\begin{align*}
  \mathbf{k}_3(\theta,\eta) \triangleq \nabla_{\mathbf{x}}K_0\big(|R(\theta)e^{i\theta}-R(\eta)e^{i \eta}|\big)\sin (\theta-\eta).
\end{align*}
In a similar way as for deriving \eqref{est:k_2} and \eqref{eq:par-k2}, we have
\begin{align*}
  |\mathbf{k}_3(\theta,\theta+\eta)|\leqslant C H_3\big(|\sin\tfrac{\eta}{2}|\big),
  \quad \big|\partial_\theta \big(\mathbf{k}_3(\theta,\theta+\eta)\big)\big|\leqslant C H_4\big(|\sin \tfrac{\eta}{2}|\big).
\end{align*}
Lemma \ref{lem:int-operator} and \eqref{ineq:K_0-diff-1-1}, \eqref{ineq:K_0-diff-1-2} guarantee that
\begin{align}\label{es:int-k3}
  \Big\lVert \int_{\mathbb{T}} \mathbf{k}_3(\theta,\eta) \tfrac{r'(\eta)}{R(\eta)}
  \dd \eta\Big\rVert_{C^{1-\alpha}}\leqslant C \Big\lVert \tfrac{r'(\eta)}{R(\eta)}\Big\rVert_{C^{1-\alpha}}\leqslant C.
\end{align}
Hence, it follows from \eqref{P-loi} and the product laws in $C^{1-\alpha}(\T)$ that
\begin{equation}\label{ineq:bound-I-04}
\begin{aligned}
  \lVert \mathcal{I}_{03}[r](h)\rVert_{C^{1-\alpha}}\leqslant C\lVert h\rVert_{C^{1-\alpha}}.
\end{aligned}
\end{equation}
Putting together \eqref{es:J1} and  \eqref{ineq:bound-I-04} yields
\begin{align}\label{I01-Es}
  \Big\|\partial_\theta \mathcal{I}_0[r]h  \Big\|_{C^{1-\alpha}}\leqslant C \lVert h\rVert_{C^{2-\alpha}}.
\end{align}
Let us now move to the term $\mathcal{I}_1$ defined in \eqref{I1-iedntity}. Then one gets
\begin{align}\label{decom:parI1}
  \partial_\theta\mathcal{I}_1[r]h(\theta) & = \int_0^{2\pi} \int_0^{R(\eta)}
  \nabla_{\mathbf{x}} K_1 \big(R(\theta)e^{i\theta},  \rho e^{i\eta}\big)
  \rho \dd \rho \dd \eta \cdot
  \partial_\theta \big(\tfrac{h(\theta)}{R(\theta)}e^{i\theta}\big) \nonumber \\
  & \quad + \partial_\theta\big(R(\theta) e^{i\theta}\big)\cdot \int_0^{2\pi} \int_0^{R(\eta)}\nabla_{\mathbf{x}}^2
  K_1 \big(R(\theta)e^{i\theta},  \rho e^{i\eta}\big) \rho \dd \rho \dd \eta \cdot
  \big(\tfrac{h(\theta)}{R(\theta)}e^{i\theta}\big) \nonumber \\
  & \quad + \partial_\theta\big(R(\theta)e^{i\theta} \big) \cdot\int_{\mathbb{T}}
  \nabla_{\mathbf{x}}K_1(R(\theta)e^{i\theta},R(\eta)e^{i\eta})h(\eta)\dd \eta \nonumber \\
  & \triangleq \,\mathcal{I}_{11}[r](h)(\theta) + \mathcal{I}_{12}[r](h)(\theta) + \mathcal{I}_{13}[r](h)(\theta).
\end{align}
Since $K_1\in C_{\text{loc}}^{4}(\mathbf{D}^2)$, then  arguing as for the estimate of  $F_{01}[r]$ in \eqref{es:F01},
one can easily show that
\begin{align*}
  \|\mathcal{I}_{11}[r](h)\|_{C^{1-\alpha}} & \leqslant C
  \Big\|\int_0^{2\pi}\int_0^{R(\eta)} \nabla_{\mathbf{x}}K_1\big(R(\theta)e^{i\theta},\rho e^{i\eta}\big)
  \rho\dd\rho\dd\eta \Big\|_{C^1} \big\|\partial_\theta\big(\tfrac{h(\theta)}{R(\theta)}e^{i\theta}\big) \big\|_{C^{1-\alpha}} \\
  & \leqslant C \|h\|_{C^{2-\alpha}}.
\end{align*}
In a similar way, we find
\begin{align}\label{ineq:bound-I-1}
  \nonumber\lVert \partial_\theta\mathcal{I}_1[r]h\rVert_{C^{1-\alpha}}&\leqslant \|\mathcal{I}_{11}[r](h)\|_{C^{1-\alpha}}
  + \|\mathcal{I}_{12}[r](h)\|_{C^{1-\alpha}} + \|\mathcal{I}_{13}[r](h)\|_{C^{1-\alpha}}\\
  & \leqslant C\lVert h\rVert_{C^{2-\alpha}}.
\end{align}
Therefore,  by collecting \eqref{eq:F_r}-\eqref{I01-Es} and \eqref{ineq:bound-I-1} we infer
\begin{align*}
  \lVert \partial_rF(\Omega,r)h\rVert_{C^{1-\alpha}}\leqslant C\lVert h\rVert_{C^{2-\alpha}}.
\end{align*}
\vskip1mm
\noindent The next goal is to prove that for given $\Omega\in \R$, the mapping $r\mapsto \partial_rF(\Omega,r)\in \mathcal{L}(\mathbf{X}_m,\mathbf{Y}_m)$
is continuous. 
Thanks to \eqref{eq:F_r}, \eqref{Pipaa1}, \eqref{decom:parI1}, it suffices to show that, for every
$r_1,r_2\in \mathbf{B}_{\epsilon_0}$ as $\lVert r_1-r_2\rVert_{C^{2-\alpha}}\to 0$,
\begin{align}\label{ineq:contin-whole}
  \sup_{\lVert h\rVert_{C^{2-\alpha}}\leqslant 1}
  \bigg( \sum_{j=0}^1 \big\lVert \mathcal{I}_{0j}[r_1](h)- \mathcal{I}_{0j}[r_2](h) \big\rVert_{C^{1-\alpha}}
  +\sum_{j=1}^3\big\lVert \mathcal{I}_{1j}[r_1](h)- \mathcal{I}_{1j}[r_2](h) \big\rVert_{C^{1-\alpha}}  \bigg) \to 0.
\end{align}
Denote by $R_j(\theta) \triangleq \sqrt{b^2 + 2 r_j(\theta)}$, $j=1,2$.
For $\mathcal{I}_{00}$ given by \eqref{Pipaa1}, we get
\begin{align*}
   &\|\mathcal{I}_{00}[r_1](h) - \mathcal{I}_{00}[r_2](h)\|_{C^{1-\alpha}}  \leqslant C \Big\|\int_{\mathbb{T}} \mathbf{k}_4(\theta,\eta) \partial_\eta\big(R_1(\eta) e^{i\eta}\big) \dd \eta\Big\|_{C^{1-\alpha}}
  \big\|\partial_\theta\big(\tfrac{h(\theta)}{R_1(\theta)}e^{i\theta}\big) \big\|_{C^{1-\alpha}} \nonumber \\
  & \quad\qquad  + C \Big\|\int_{\T} K_0\big(|X_2(\theta,\eta)|\big)
  \partial_\eta\big(\big(R_1(\eta) - R_2(\eta)\big)e^{i\eta}\big) \dd \eta \Big\|_{C^{1-\alpha}}
  \big\|\partial_\theta\big(\tfrac{h(\theta)}{R_1(\theta)}e^{i\theta}\big) \big\|_{C^{1-\alpha}}  \\
  & \quad \qquad +  C \Big\|\int_{\T} K_0\big(|X_2(\theta,\eta)|\big) \partial_\eta\big(R_2(\eta)e^{i\eta}\big) \dd \eta \Big\|_{C^{1-\alpha}} \big\|\partial_\theta\big(\tfrac{h(\theta)}{R_1(\theta)}e^{i\theta}\big)
  -\partial_\theta\big(\tfrac{h(\theta)}{R_2(\theta)}e^{i\theta}\big)\big\|_{C^{1-\alpha}}, \nonumber
\end{align*}
where
\begin{align}\label{def:Xj}
  X_j(\theta,\eta) \triangleq  R_j(\theta)e^{i\theta}-R_j(\eta)e^{i\eta},\quad j=1,2,
\end{align}
and
\begin{align*}
  \mathbf{k}_4(\theta,\eta) \triangleq K_0 \big(|X_1(\theta,\eta)|\big) - K_0 \big(|X_2(\theta,\eta)|\big).
\end{align*}
Hence,
taking advantage of  the estimates
\begin{align*}
  \|R_1(\eta)-R_2(\eta)\|_{C^{2-\alpha}} \leqslant C \|r_1 - r_2\|_{C^{2-\alpha}},
\end{align*}
and
\begin{align*}
  \big\|\partial_\theta\big(\tfrac{h(\theta)}{R_1(\theta)}e^{i\theta}\big)
  - \partial_\theta\big(\tfrac{h(\theta)}{R_2(\theta)} e^{i\theta}\big) \big\|_{C^{1-\alpha}}
  \leqslant C \|h\|_{C^{2-\alpha}} \|r_1 - r_2\|_{C^{2-\alpha}},
\end{align*}
together with \eqref{es:F00-Hold} we deduce that
\begin{equation}\label{es:I01-r}
\begin{split}
  \|\mathcal{I}_{00}[r_1](h) - \mathcal{I}_{00}[r_2](h)\|_{C^{1-\alpha}}
  & \leqslant C \lVert h\rVert_{C^{2-\alpha}} \lVert r_1-r_2\rVert_{C^{2-\alpha}} \\
  & \quad +  C \|h\|_{C^{2-\alpha}} \Big\|\int_{\mathbb{T}} \mathbf{k}_4(\theta,\eta)\partial_\eta\big( R_1(\eta) e^{i\eta}\big) \dd \eta\Big\|_{C^{1-\alpha}}.
\end{split}
\end{equation}
The next goal is to estimate $\mathbf{k}_4$ and without loss of generality we can assume that $|X_1(\theta,\theta+\eta)|\leqslant |X_2(\theta,\theta+\eta)|$. Then, according to  Lemma \ref{lem-high-deriv}-(2) we deduce the inequality
\begin{align*}
  |\mathbf{k}_4(\theta,\theta+\eta)| \leqslant
  \big(|X_1(\theta,\theta+\eta)|-|X_2(\theta,\theta+\eta)|\big)K_0^\prime(|X_1(\theta,\theta+\eta)|).
\end{align*}
Using the triangle inequality together with \eqref{ineq:basic} yields
\begin{align*}
  |\mathbf{k}_4(\theta,\theta+\eta)| \leqslant (|X_1(\theta,\theta+\eta)-X_2(\theta,\theta+\eta)|\big)
  \big|K_0^\prime( b |\sin \tfrac{\eta}{2}|)\big|.
\end{align*}
Applying Taylor's formula and using  the $2\pi$-periodicity, implying that we can assume $|\eta|\leqslant \pi$,
\begin{align}\label{eq:X1-X2}
  \nonumber |X_1(\theta,\theta+\eta) - X_2(\theta,\theta+\eta)| & = \Big|\int_0^1 \partial_\theta
  \Big(R_1(\theta + \tau \eta ) e^{i (\theta + \tau \eta)} - R_2(\theta + \tau \eta) e^{i (\theta + \tau \eta)} \Big)
  \cdot \eta \dd \tau \Big| \\
  & \leqslant C_0 \|R_1 - R_2\|_{C^1} |\eta|
  \leqslant C_0 \|r_1 - r_2\|_{C^1} |\sin \tfrac{\eta}{2}|.
\end{align}
Therefore we deduce that
\begin{align}\label{eq:bound-k4-0-th}
 |\mathbf{k}_4(\theta,\theta+\eta) \leqslant C \|r_1 - r_2\|_{C^1}\big|K_0^\prime( \tfrac12b\big|\sin \tfrac{\eta}{2}\big|)\big| |\sin \tfrac{\eta}{2}|\triangleq C\lVert r_1-r_2\rVert_{C^1}H_3\big(|\sin \tfrac{\eta}{2}|\big).
\end{align}
Now, we shall estimate the derivative $  \partial_\theta \big(\mathbf{k}_4(\theta,\theta+\eta)\big)$ which takes the form
\begin{align*}
  \partial_\theta \big(\mathbf{k}_4(\theta,\theta+\eta)\big)
  & = \Big(K_0'\big(|X_1(\theta,\theta+\eta)|\big) - K_0'\big(|X_2(\theta,\theta+\eta)| \big) \Big)
  \tfrac{X_1(\theta,\theta+\eta)}{|X_1(\theta,\theta+\eta)|} \cdot \partial_\theta \big(X_1(\theta,\theta+\eta)\big) \\
  & \quad + K_0'\big(|X_2(\theta,\theta+\eta)| \big) \Big( \tfrac{X_1(\theta,\theta+\eta)}{|X_1(\theta,\theta+\eta)|}
  - \tfrac{X_2(\theta,\theta+\eta)}{|X_2(\theta,\theta+\eta)|}  \Big) \cdot \partial_\theta \big(X_1(\theta,\theta+\eta)\big) \\
  & \quad + K_0'\big(|X_2(\theta,\theta+\eta)|\big) \tfrac{X_2(\theta,\theta+\eta)}{|X_2(\theta,\theta+\eta)|} \cdot
  \Big(\partial_\theta \big(X_1(\theta,\theta+\eta)\big) - \partial_\theta \big(X_2(\theta,\theta+\eta)\big) \Big) .
\end{align*}
As $R_1\in C^{2-\alpha}(\T),$ then we deduce that for  $|\eta|\leqslant \pi$,
\begin{align*}
  \big|\partial_\theta \big(X_1(\theta,\theta+\eta)\big)\big| & \leqslant \big| R_1'(\theta) e^{i\eta}
  - R_1'(\theta+\eta) e^{i(\theta+\eta)} \big|
  + \big|R_1(\theta) i e^{i\theta} - R_1(\theta+\eta) i e^{i(\theta+\eta)} \big| \\
  & \leqslant C \big(\|R_1'\|_{C^{1-\alpha}} + \|R_1\|_{C^{1-\alpha}} \big) |\eta|^{1-\alpha} \\
  & \leqslant C \|r_1\|_{C^{2-\alpha}} |\sin \tfrac{\eta}{2}|^{1-\alpha}.
\end{align*}
Similarly, we get
\begin{align*}
   \big|\partial_\theta \big( X_1(\theta,\theta+\eta)\big) - \partial_\theta \big(X_2(\theta,\theta+\eta)\big) \big|&  \leqslant \big|\big(R_1'(\theta) - R_2'(\theta)\big)
  - \big(R_1'(\theta +\eta) - R_2'(\theta+\eta)\big) e^{i\eta}\big| \\
  & \quad +  \big|\big(R_1(\theta) - R_2(\theta)\big)
  - \big(R_1(\theta +\eta) - R_2(\theta+\eta)\big)  e^{i\eta}\big| \\
  & \leqslant C \big(\|R_1' - R_2'\|_{C^{1-\alpha}} + \|R_1 - R_2\|_{C^{1-\alpha}} \big) |\eta|^{1-\alpha}\\
 & \leqslant C \|r_1 -r_2 \|_{C^{2-\alpha}} |\sin\tfrac{\eta}{2}|^{1-\alpha}.
\end{align*}
We may assume $|X_1(\theta,\theta+\eta)|\leqslant|X_2(\theta,\theta+\eta)|$. Then, applying  Lemma \ref{lem-high-deriv}-(2) with the triangle inequality allows us to get,
\begin{align*}
  \big|K_0'(|X_1(\theta,\theta+\eta)|) - K_0'(|X_2(\theta,\theta+\eta)|) \big|\leqslant \big(|X_1(\theta,\theta+\eta)-X_2(\theta,\theta+\eta)|\big)\big|K_0^{\prime\prime}(|X_1(\theta,\theta+\eta)|)\big|.
\end{align*}
Thus, we obtain by virtue of \eqref{eq:X1-X2}, the monotonicity of $|K_0^{\prime\prime}|$ and \eqref{ineq:basic},
\begin{align*}
   \big|K_0'(|X_1(\theta,\theta+\eta)|) - K_0'(|X_2(\theta,\theta+\eta)|) \big|
  & \leqslant  C_0 \|r_1 - r_2\|_{C^1} |\sin \tfrac{\eta}{2}|\big|K_0''(b|\sin\tfrac{\eta}{2}|)\big|.
\end{align*}
Using Lemma \ref{lem-high-deriv}-(1) gives
\begin{align*}
   \big|K_0'(|X_1(\theta,\theta+\eta)|) - K_0'(|X_2(\theta,\theta+\eta)|) \big|
  & \leqslant  C_0 \|r_1 - r_2\|_{C^1} \big|K_0'(\tfrac12b|\sin\tfrac{\eta}{2}|)\big|.
\end{align*}
Putting together the foregoing estimates we get by straightforward computations
\begin{align}\label{eq:bound-k4-1-th}
  \big|\partial_\theta \big(\mathbf{k}_4(\theta,\theta+\eta)\big)\big|
  \leqslant &\, C \lVert r_1-r_2\rVert_{C^{2-\alpha}}
 \big|K_0'(\tfrac12b|\sin \tfrac{\eta}{2}|) \big| \, |\sin \tfrac{\eta}{2}|^{1-\alpha}
   \nonumber \\
  \triangleq  &C\,\lVert r_1-r_2\rVert_{C^{2-\alpha}}H_4\big(|\sin \tfrac{\eta}{2} |\big).
\end{align}
Hence, \eqref{eq:bound-k4-0-th}, \eqref{eq:bound-k4-1-th}, \eqref{ineq:K_0-diff-1-1},
\eqref{ineq:K_0-diff-1-2} and Lemma \ref{lem:int-operator} ensure that
\begin{align*}
  \Big\|\int_{\mathbb{T}} \mathbf{k}_4(\theta,\eta) \partial_\eta(R_1(\eta) e^{i\eta}) \dd \eta\Big\|_{C^{1-\alpha}}
  &\leqslant C \|r_1-r_2\|_{C^{2-\alpha}} \|\partial_\eta(R_1(\eta)e^{i\eta})\|_{C^{1-\alpha}} \\
  &\leqslant C \|r_1-r_2\|_{C^{2-\alpha}},
\end{align*}
and consequently,
\begin{align}\label{eq:I01-r2}
  \|\mathcal{I}_{00}[r_1](h) - \mathcal{I}_{00}[r_2](h)\|_{C^{1-\alpha}}\leqslant C\lVert r_1-r_2\rVert_{C^{2-\alpha}}
  \lVert h\rVert_{C^{2-\alpha}}.
\end{align}
For $\mathcal{I}_{01}[r](h)$ given by \eqref{Pipaa1}, using \eqref{es:J1} and the estimates that
$\|R_i(\theta)e^{i\theta}\|_{C^1(\T)}\leqslant C$ and
\begin{align*}
  \Big\lVert\partial_{\theta}(R_1(\theta)e^{i\theta})-\partial_{\theta}(R_2(\theta)e^{i\theta})\Big\rVert_{C^{1-\alpha}}
  \leqslant C\lVert r_1-r_2\rVert_{C^{2-\alpha}},
\end{align*}
we find in a similar way as  for deriving \eqref{es:I01-r}
\begin{equation}\label{es:I01diff}
\begin{split}
  \|\mathcal{I}_{01}[r_1](h) - \mathcal{I}_{01}[r_2](h)\|_{C^{1-\alpha}}
  & \leqslant C\lVert h\rVert_{C^{2-\alpha}} \lVert r_1-r_2\rVert_{C^{2-\alpha}}  \\
  & \quad +C\sum_{k=2}^{3}\lVert \mathcal{I}_{0k}[r_1](h)-\mathcal{I}_{0k}[r_2](h)\rVert_{C^{1-\alpha}}.
\end{split}
\end{equation}
From \eqref{P-loi}, we denote
\begin{align}\label{es:I02-1}
  \lVert \mathcal{I}_{02}[r_1](h)-\mathcal{I}_{02}[r_2](h)\rVert_{C^{1-\alpha}}
  = \Big\|\int_{\T} \mathbf{k}_5(\theta,\eta)\dd\eta \Big\|_{C^{1-\alpha}},
\end{align}
where
\begin{align*}
  \mathbf{w}_{2j}(\theta,\eta) \triangleq h(\eta)-\tfrac{R_j(\eta)\cos (\theta-\eta)}{R_j(\theta)}h(\theta),\quad j=1,2,
\end{align*}
and
\begin{align*}
  \mathbf{k}_5(\theta,\eta)\triangleq
  \nabla_{\mathbf{x}} K_0\big(X_1(\theta,\eta)\big) \mathbf{w}_{21}(\theta,\eta)
  - \nabla_{\mathbf{x}} K_0\big(X_2(\theta,\eta)\big) \mathbf{w}_{22}(\theta,\eta).
\end{align*}
We rewrite $\mathbf{k}_5(\theta,\theta+\eta)$ as follows
\begin{align*}
  \mathbf{k}_5(\theta,\theta+\eta) & =
  \Big(K_0' (|X_1(\theta,\theta+\eta)|) - K_0'(|X_2(\theta,\theta+\eta)|)\Big)
  \tfrac{X_1(\theta,\theta+\eta)}{|X_1(\theta,\theta+\eta)|} \mathbf{w}_{21}(\theta,\theta+\eta) \\
  & \quad +  K_0'(|X_2(\theta,\theta+\eta)|) \Big( \tfrac{X_1(\theta,\theta+\eta)}{|X_1(\theta,\theta+\eta)|}
  - \tfrac{X_2(\theta,\theta+\eta)}{|X_2(\theta,\theta+\eta)|}\Big)\,\mathbf{w}_{21}(\theta,\theta+\eta) \\
  & \quad + \nabla_{\mathbf{x}} K_0\big( |X_2(\theta,\theta+\eta|)\big)\,\Big(\mathbf{w}_{21}(\theta,\theta+\eta)
  - \mathbf{w}_{22}(\theta,\theta+\eta) \Big).
\end{align*}
By the identity
\begin{align*}
  &\mathbf{w}_{21}(\theta,\theta+\eta)-\mathbf{w}_{22}(\theta,\theta+\eta)
  = h(\theta)\cos (\eta)\Big(\tfrac{R_2(\theta+\eta)}{R_2(\theta)}-\tfrac{R_1(\theta+\eta)}{R_1(\theta)}\Big)\\
  &=h(\theta)\cos (\eta)
  \tfrac{R_1(\theta)\big((R_2-R_1)(\theta+\eta)-(R_2-R_1)(\theta)\big)
  +\big(R_2(\theta)-R_1(\theta)\big)\big(R_1(\theta)-R_1(\theta+\eta)\big)}{R_1(\theta)R_2(\theta)},
\end{align*}
we deduce that
\begin{align}\label{ineq:w1-w2-0-th}
  |\mathbf{w}_{21}(\theta,\theta+\eta)-\mathbf{w}_{22}(\theta,\theta+\eta)|\leqslant
  C\lVert h\rVert_{L^{\infty}}\lVert r_1-r_2\rVert_{C^1}\big|\sin \tfrac{\eta}{2}\big|,
\end{align}
and
\begin{align}\label{ineq:w1-w2-1-th}
  \Big|\partial_{\theta}\big(\mathbf{w}_{21}(\theta,\theta+\eta)\big)- \partial_\theta\big(\mathbf{w}_{22}(\theta,\theta+\eta)\big)\Big|
  \leqslant C\lVert h\rVert_{C^1}\lVert r_1-r_2\rVert_{C^{2-\alpha}}\big|\sin \tfrac{\eta}{2}\big|^{1-\alpha}.
\end{align}
Arguing as in  \eqref{est:k_2} and \eqref{eq:bound-k4-1-th}, and using \eqref{ineq:w-0-th}, \eqref{ineq:w1-w2-0-th} and Lemma \ref{lem-high-deriv}-(1), we infer that
\begin{align*}
  |\mathbf{k}_5(\theta,\theta+\eta)|\leqslant &\, C \lVert r_1-r_2\rVert_{C^1}\lVert h\rVert_{C^1}
  \Big(\big|K_0'(b|\sin \tfrac{\eta}{2}|)\big| |\sin \tfrac{\eta}{2}|
  + \big| K''_0(b|\sin \tfrac{\eta}{2}|)\big|\,|\sin \tfrac{\eta}{2}|^2\Big)\\
  \leqslant&\, C \lVert r_1-r_2\rVert_{C^1}\lVert h\rVert_{C^1}
  \big|K_0'(\tfrac12b|\sin \tfrac{\eta}{2}|)\big| |\sin \tfrac{\eta}{2}|
  \\
  =& \,C \lVert r_1-r_2\rVert_{C^1}\lVert h\rVert_{C^1}H_3\big(|\sin \tfrac{\eta}{2}|\big).
\end{align*}
In a similar way to  \eqref{eq:par-k2} and \eqref{eq:bound-k4-1-th}, and after some tedious computations, we arrive at
\begin{align*}
  \big|\partial_\theta \big(\mathbf{k}_5(\theta,\theta+\eta)\big)\big| \leqslant & \,C\lVert r_1-r_2\rVert_{C^{2-\alpha}}
  \lVert h\rVert_{C^{2-\alpha}}\bigg( \sum_{j=1}^{3}\big|K_0^{(j)}(b|\sin \tfrac{\eta}{2}|)\big|\,
  |\sin \tfrac{\eta}{2}|^{j-\alpha} \bigg) \\
  \leqslant &\, C\lVert r_1-r_2\rVert_{C^{2-\alpha}}
  \lVert h\rVert_{C^{2-\alpha}}\big|K'_0(\tfrac12b|\sin \tfrac{\eta}{2}|)\big|\,
  |\sin \tfrac{\eta}{2}|^{1-\alpha} \\
  =&\,C \lVert r_1-r_2\rVert_{C^{2-\alpha}} \lVert h\rVert_{C^{2-\alpha}}
  H_4\big(|\sin \tfrac{\eta}{2}|\big).
\end{align*}
Hence, \eqref{ineq:K_0-diff-1-1}, \eqref{ineq:K_0-diff-1-2} and Lemma \ref{lem:int-operator} implies that
\begin{align*}
  \Big\|\int_{\T}\mathbf{k}_5(\theta,\eta)\dd\eta \Big\|_{C^{1-\alpha}}
  \leqslant C \|r_1-r_2\|_{C^{2-\alpha}} \|h\|_{C^{2-\alpha}}.
\end{align*}
Putting it together  with \eqref{es:I02-1} yields
\begin{align}\label{es:I02-2}
  \|\mathcal{I}_{02}[r_1](h) - \mathcal{I}_{02}[r_2](h)\|_{C^{1-\alpha}} \leqslant C \|r_1 -r_2\|_{C^{2-\alpha}}
  \|h\|_{C^{2-\alpha}}.
\end{align}
\vskip0.5mm
\noindent The estimate of $\mathcal{I}_{03}[r_1](h)-\mathcal{I}_{03}[r_2](h)$ is quite similar  to the preceding one, using in particular \eqref{es:int-k3},
and we shall just state the final result omitting the details,
\begin{align}\label{es:I03}
  \|\mathcal{I}_{03}[r_1](h)-\mathcal{I}_{03}[r_2](h)\|_{C^{1-\alpha}}
  \leqslant C \|r_1-r_2\|_{C^{2-\alpha}} \|h\|_{C^{2-\alpha}}.
\end{align}
\vskip1mm
\noindent Next, the estimation of $\mathcal{I}_{1j}[r_1](h)-\mathcal{I}_{1j}[r_2](h)$ ($j=1,2,3$) is more straightforward, and one gets
\begin{align}\label{ineq:contin-I-1}
  \sum_{j=1,2,3}\lVert \mathcal{I}_{1j}[r_1](h)-\mathcal{I}_{1j}[r_2](h)\rVert_{C^{1-\alpha}}
  \leqslant C \lVert r_1-r_2\rVert_{C^{2-\alpha}} \lVert h\rVert_{C^{2-\alpha}}.
\end{align}
Below, we only give the proof of the estimate for $\mathcal{I}_{11}$,
since the remaining terms  are similar.
Indeed, noting from   \eqref{decom:parI1} that
\begin{align*}
  \mathcal{I}_{11}[r_1](h)(\theta) & - \mathcal{I}_{11}[r_2](h)(\theta) = \int_0^{2\pi}
  \int_{R_2(\eta)}^{R_1(\eta)}\nabla_{\mathbf x} K_1\big(R_1(\theta) e^{i\theta},
  \rho e^{i\eta}\big)  \rho \dd\rho \dd\eta \cdot \partial_\theta \Big(\tfrac{h(\theta)}{R_1(\theta)}e^{i\theta} \Big) \\
  &  + \big(R_1(\theta)e^{i\theta} - R_2(\theta) e^{i\theta} \big)\cdot \Pi(\theta)
  \cdot \partial_\theta\Big(\tfrac{h(\theta)}{R_1(\theta)}e^{i\theta} \Big)  \\
  & + \int_0^{2\pi} \int_0^{R_2(\eta)} \nabla_{\mathbf{x}} K_1\big(R_2(\theta)e^{i\theta},\rho e^{i\eta}\big)\,\rho\dd \rho \dd\eta
  \cdot \Big(\partial_\theta\Big(\tfrac{h(\theta)}{R_1(\theta)}e^{i\theta}\Big)
  -\partial_\theta\Big(\tfrac{h(\theta)}{R_2(\theta)}e^{i\theta}\Big) \Big),
\end{align*}
with
\begin{align*}
  \theta\in\mathbb{R}\mapsto \Pi(\theta) \triangleq \int_0^1 \int_0^{2\pi} \int_0^{R_2(\eta)} \nabla_{\mathbf{x}}^2
  K_1\big(s R_1(\theta)e^{i\theta} +(1-s)R_2(\theta)e^{i\theta}, \rho e^{i\eta} \big)
  \rho\, \dd \rho \dd \eta\dd s \in C^1(\T),
\end{align*}
and using the $C^4_{\mathrm{loc}}$-smoothness of $K_1$ yields
\begin{align*}
  \|\mathcal{I}_{11}[r_1](h) - \mathcal{I}_{11}[r_2](h)\|_{C^{1-\alpha}(\T)} \leqslant C \|r_1 - r_2\|_{C^{2-\alpha}} \|h\|_{C^{2-\alpha}}.
\end{align*}
\vskip0.5mm
\noindent Therefore, gathering the above estimates we conclude \eqref{ineq:contin-whole},
allowing to get the desired  result on the the continuity of $\partial_r F(\Omega,r)$.
\\[1mm]
\textbf{(3)} Since $\partial_\Omega \partial_r F(\Omega,r) h(\theta) = h'(\theta)$, then the  regularity result follows immediately.
\end{proof}

\subsection{Spectral study}\label{subsec:spect-study}
In this subsection we focus on the spectral study of the linearized operator  at zero,
given by $\partial_rF(\Omega,0)$.
Consider
\begin{align*}
  \theta\in\R\mapsto h(\theta)=\sum\limits_{n=1}^{\infty}a_n\cos (nm\theta)\in \mathbf{X}_m,\,\,a_n\in \R.
\end{align*}
Then, according to \eqref{exp:V0} and \eqref{def:lambda-nb}-\eqref{exp:par-Vr=0}, we have
\begin{align}\label{eq:F-h-general}
  \partial_{r}F(\Omega,0) h(\theta)
  & = \big(\Omega + V[0] \big)h'(\theta) + \mathcal{L}[0](h')(\theta) \nonumber \\
  & = - \sum_{n=1}^{\infty}a_n \big(\Omega-\Omega_{nm,b}\big)nm \sin(nm\theta),
\end{align}
where $\Omega_{n,b}$ ($n\in \mathbb{N}^\star$) satisfies
\begin{equation}\label{eq:Omega-whole}
	\Omega_{n,b}= \Omega_{n,b}^0 + \Omega_{n,b}^1,
\end{equation}
with
\begin{equation}\label{def:tild-Omeg}
\begin{split}
  \Omega_{n,b}^0 & \triangleq  \int_{\T}K_0(2b|\sin \tfrac{\eta}{2}|)\cos \eta\,\dd \eta
  -\int_{\T} K_0(2b|\sin \tfrac{\eta}{2}|) \cos (n\eta)\dd \eta \\
  & = \lambda_{1,b} - \lambda_{n,b},
\end{split}
\end{equation}
and
\begin{align}\label{def:bar-Omeg}
  \Omega_{n,b}^1 \triangleq & - b^{-1} \int_0^{2\pi} \int_0^b \partial_{\rho_1}G_1(b,0,\rho,\eta) \rho \dd \rho \dd \eta
  -\int_{\T} K_1(b,b e^{i\eta})\cos (n\eta) \dd \eta.
\end{align}
In particular, if $K(\mathbf{x},\mathbf{y}) = K_0(|\mathbf{x}-\mathbf{y}|)$,
then $\Omega_{n,b} = \Omega_{n,b}^0 = \lambda_{1,b} -\lambda_{n,b}$ with $\lambda_{n,b}$ given by \eqref{lambda_n}.
\\[1mm]
Lemma \ref{lem:lamb-n} and the results in Section \ref{sec:phi_n} imply the following crucial properties of $\Omega_{n,b}^0$.
\begin{lemma}\label{lem:Omega}
Let $K_0$ be a smooth function satisfying the assumptions $(\mathbf{A}1)$-$(\mathbf{A}2)$, see \eqref{eq:K0prim} and \eqref{cond:K0}.
Then the following statements hold true.
\begin{enumerate}[(1)]
\item For any $n\in \mathbb{N}^\star$, we have
\begin{align*}
  \lambda_{1,b} - \frac{8n^2}{4n^2-1} \int_0^\infty \frac{b}{n^2+(bx)^2}\dd \mu(x)
  \leqslant \Omega_{n,b}^0
  \leqslant \lambda_{1,b} -\frac{8n^2}{4n^2+1} \int_0^\infty \frac{b}{n^2+(bx)^2} \dd \mu(x),
\end{align*}
and
\begin{align*}
  \lim_{n\to \infty} \Omega_{n,b}^0 = \int_{\T} K_0\big(2b|\sin \tfrac{\eta}{2}|\big) \cos \eta\, \dd \eta.
\end{align*}
\item The map $n\in \N^\star\mapsto \Omega_{n,b}^0$ is strictly increasing and
\begin{align*}
  \frac{1}{2}D_n\leqslant \Omega_{n+1,b}^0 - \Omega_{n,b}^0 \leqslant 4 D_n,
\end{align*}
with
\begin{align*}
  D_n \triangleq \int_0^{\infty}\frac{b}{n^2+(bx)^2}\frac{2n+1}{(n+1)^2+(bx)^2}\dd \mu(x).
\end{align*}
\end{enumerate}
\end{lemma}

\begin{proof}[Proof of Lemma \ref{lem:Omega}]
In light of \eqref{def:tild-Omeg} and Lemma \ref{lem:lamb-n},
we have a useful  formula for $\Omega_{n,b}^0$ in terms of $\phi_n$ given by \eqref{def:phi-n},
\begin{align}\label{eq:Omega-whole-s}
  \Omega_{n,b}^0 = \lambda_{1,b}-\lambda_{n,b}
  = 2 \int_0^\infty \big(\phi_1(b x)-\phi_n (b x) \big) \tfrac{\dd \mu(x)}{x}.
\end{align}
\vskip0.5mm
\noindent Hence, the statement $(1)$ follows directly from Lemma \ref{lem:bound-varphi}.\\
 As to the estimate of the point $(2)$, it  can be deduced from \eqref{eq:bound-diff-phi-n} and \eqref{eq:Omega-whole-s}.
In addition, since $\mu$ is a nonnegative measure and is not zero measure,
there exist some $0< d<\infty$ and $c_*>0$ such that $\mu([0,d])\geqslant c_* >0$.
Then, we obtain  the strict monotonicity of $(\Omega_{n,b}^0)_{n\in \mathbb{N}^\star}$, that is,
\begin{align}\label{eq:td-Omeg-lbd}
  \Omega_{n+1,b}^0 - \Omega_{n,b}^0 & \geqslant
  \frac{1}{2}\int_0^d \frac{b}{n^2+(bx)^2}\frac{2n+1}{(n+1)^2+(bx)^2}\dd \mu(x) \nonumber \\
  & \geqslant \frac{c_*}{2} \frac{b}{n^2+(bd)^2}\frac{2n+1}{(n+1)^2+(bd)^2}
  \geqslant \frac{c_*' }{n^3},
\end{align}
with $c_*'>0$ depending only on $c_*$, $d$ and $b$.
\end{proof}

\noindent Next, we intend  to  show the monotonicity of the sequence $(\Omega_{n,b})$ for large modes.
\begin{lemma}\label{lem:Omega-pertubative}
Consider the general case \eqref{case:K-2} with $K_0$ and $K_1$ satisfying the assumptions $ (\mathbf{A}\mathrm{1})$-$(\mathbf{A}\mathrm{4})$.
Then there exist $m_0\in\mathbb{N}^\star$ and $C>0$ such that for any $m\geqslant m_0$ and $ n\geqslant 1$,
\begin{align*}
  \Omega_{(n+1)m,b}-\Omega_{nm,b}\geqslant \frac{C}{(nm)^3}.
\end{align*}
In addition,
\begin{align*}
  \lim_{n\to \infty}\Omega_{nm,b}=\int_{\T}K_0\big(2b |\sin\tfrac{\eta}{2}|\big) \cos \eta\,\dd \eta
  - b^{-1} \int_0^{2\pi} \int_0^b \partial_{\rho_1}G_1(b,0,\rho,\eta) \, \rho \dd \rho \dd \eta.
\end{align*}
\end{lemma}

\begin{proof}[Proof of Lemma \ref{lem:Omega-pertubative}]
  Since the kernel  $K_1$ belongs to $C^3_{\mathrm{loc}}(\mathbf{D}^2)$ and
$\eta\mapsto K_1(b,be^{i\eta})$  to $C^3(\mathbb{T})$.
Then, using integration by parts we infer
\begin{equation}\label{eq:G1-est}
\begin{split}
  \Big|\int_{\T}K_1(b,b e^{i\eta})\cos (nm\eta) \dd \eta\Big|
  =&\Big|\frac{1}{(nm)^3}\int_0^{2\pi}\partial^3_\eta \big(K_1(b,b e^{i\eta})\big) \sin (nm\eta)\dd \eta\Big|  \\
  \leqslant &\frac{C_2}{(nm)^3},
\end{split}
\end{equation}
with some $C_2>0$.
Hence, in view of \eqref{eq:Omega-whole}, \eqref{def:bar-Omeg}, \eqref{eq:td-Omeg-lbd} and \eqref{eq:G1-est}, we find
\begin{align*}
  \Omega_{(n+1)m,b}-\Omega_{nm,b}= & \,\Omega_{(n+1)m,b}^0 - \Omega_{nm,b}^0
  + \Omega_{(n+1)m,b}^1 - \Omega_{nm,b}^1 \\
  \geqslant &  \sum_{k =nm}^{(n+1)m-1}\frac{c_*'}{k^3} - \frac{ 2 C_2}{(nm)^3}\\
  \geqslant & \frac{c_*' }{(n+1)^3 m^2} - \frac{ C_2 }{(nm)^3}\cdot
\end{align*}
Choosing some $m_0>\frac{20 C_2}{c_*'}$, we show the first result on the lower bound of $ \Omega_{(n+1)m,b}-\Omega_{nm,b}$.
\\[0.5mm]
Next, using Riemann-Lebesgue's lemma combined with Lemma \ref{lem:Omega} allow to get the convergence result. This ends the proof of the desired result.
\end{proof}

\noindent Now, we are in a position to show the main result on the spectral study of $\partial_r F(\Omega,0)$,
by showing the validity of all  the requirements in  Crandall-Rabinowitz's theorem. The function spaces that will be used below are described in \eqref{def:Xm} and \eqref{def:Ym}.
\begin{proposition}\label{propos:bifurcation}
Assume that either the assumptions of Theorem $\ref{thm:main}$ or those of Theorem $\ref{thm:perturbative}$ are satisfied.
Then the following statements hold true.
\begin{enumerate}[(1)]
\item
The kernel of $\partial_rF(\Omega,0):\mathbf{X}_m\to \mathbf{Y}_m$ is non-trivial if and only if $\Omega=\Omega_{\ell m,b}$
for some $\ell \in \mathbb{N}^\star$.
In this case, it is a one-dimensional vector space  generated by $\theta\mapsto \cos (\ell m\theta)$.
\item
The range of $\partial_r F(\Omega_{\ell m,b},0)$ is closed and is of co-dimension one.
It is given by
\begin{align*}
  Range(\partial_r F(\Omega_{\ell m,b},0))=\Big\{r\in C^{1-\alpha}(\mathbb{T}): r(\theta)
  =\sum_{n\geqslant 1, n\neq \ell}a_n \sin (nm\theta),a_n\in \R\Big\}.
\end{align*}
\item
Transversality condition:
\begin{align*}
  \partial_\Omega \partial_r F(\Omega_{\ell m,b},0)(\cos (\ell m\theta)) \not\in R(\partial_r F(\Omega_{\ell m,b},0)).
\end{align*}
\end{enumerate}
\end{proposition}

\begin{proof}[Proof of Proposition \ref{propos:bifurcation}]
\textbf{(1)} The proof of statement $(1)$ is a direct consequence of  \eqref{eq:F-h-general} and the strict monotonicity of $n\in\N^\star\mapsto \Omega_{nm,b}$, seen in Lemmas \ref{lem:Omega}.
\\	
\textbf{(2)} From \eqref{eq:F-h-general}, it is obvious to see that
\begin{align*}
  R(\partial_rF(\Omega_{\ell m,b},0))\subset\Big\{ r\in C^{1-\alpha}(\mathbb{T}): r(\theta)
  =\sum_{n\geqslant 1, n\ne \ell}a_{n}\sin (nm\theta),a_n\in \R\Big\}.
\end{align*}
Next we prove the converse inclusion relationship. For any $r\in C^{1-\alpha}(\T)$ satisfying
\begin{align*}
  r(\theta)=\sum\limits_{n\geqslant 1\atop  n\ne \ell}b_n\sin (nm\theta),
\end{align*}
we have to find some $h\in \mathbf{X}_m$ such that $\partial_{r}F(\Omega_{\ell m,b},0)h = r$. In view of \eqref{eq:F-h-general}, we formally get
\begin{align*}
 h(\theta)= \sum_{n\geqslant1\atop n\neq \ell}\frac{b_n}{(\Omega_{nm,b}-\Omega_{\ell m,b})nm}\cos (nm\theta),
\end{align*}
and we need to prove that $h\in C^{2-\alpha}(\T).$ First, we write
\begin{align}\label{eq:Omeg-targ}
 h(\theta)
  = \sum_{k\geqslant1\atop k\neq \ell m} \frac{1}{\Omega_{k,b}- \Omega_{\ell m, b}} \widetilde{b}_k \cos(k\theta),
\end{align}
where
\begin{align*}
  \widetilde{b}_k \triangleq
  \begin{cases}
    \frac{b_{k/m}}{k},\quad & \textrm{for}\;\; k \in  m\mathbb{N}^\star \\
    0,\quad & \textrm{for}\;\; k \notin  m\mathbb{N}^\star.
  \end{cases}
\end{align*}
Notice that one easily gets that
\begin{align*}
  \theta\mapsto \sum_{k\geqslant 1\atop k\neq m\ell} \widetilde{b}_k \cos(k\theta)
  \in C^{2-\alpha}(\T).
\end{align*}
Since $(\Omega_{k,b})_{k\in \mathbb{N}^\star}$ is strictly increasing with respect to $k$, we have
\begin{align*}
  \sup_{k\neq m\ell} \frac{1}{|\Omega_{k,b} - \Omega_{\ell m,b}|}<\infty,
\end{align*}
and
\begin{align*}
  \sup_{k\neq m\ell\atop k\neq m\ell-1 } \Big|\frac{1}{(\Omega_{k+1,b} - \Omega_{\ell m, b}) (\Omega_{k,b} - \Omega_{\ell m,b})} \Big|
  <\infty.
\end{align*}
In order to show $h\in C^{2-\alpha}(\T)$, by applying Lemma \ref{lem:multiplier-lemma},
we only need to prove that
\begin{align}\label{eq:cond-Omeg}
  \sup_{k\geqslant 1} k |\Omega_{k+1,b}-\Omega_{k,b}| < \infty.
\end{align}
Indeed, if the case \eqref{case:K-1} is considered, by virtue of \eqref{def:tild-Omeg}, \eqref{lambda_n},
the monotonicity property of $|K_0|$ and Lemma \ref{lem:int}, we infer from integration by parts  that
\begin{align*}
  |\Omega_{k+1,b} - \Omega_{k,b}|
  = &\, 2 \Big|\int_0^\pi K_0(2 b\sin \eta)e^{2i k\eta}(e^{2i\eta}-1)\dd \eta\Big| \\
  \leqslant & \frac{4}{k}\int_0^{\frac\pi2} |K_0(2b\sin \eta )|\dd \eta
  +\frac{4 b}{k}\int_0^{\frac\pi2}|K'_0(2 b\sin \eta)|\, |e^{2i\eta} -1| \dd \eta \\
  \leqslant & \frac{4}{k} \int_0^{\frac{\pi}{2}} |K_0(\tfrac{4}{\pi} b \eta)| \dd \eta
  +\frac{4 b}{k} \int_0^{\frac{\pi}{2}} |K'_0(\tfrac{4}{\pi} b \eta)| \eta \dd \eta\\
  \leqslant & \frac{C}{k},
\end{align*}
where in the last line, we have applied  \eqref{cond:K0} and Lemma \ref{lem-high-deriv}.
For   the general case \eqref{case:K-2}, we combine    \eqref{eq:Omega-whole}, \eqref{eq:G1-est} and the above inequality,
leading to  \eqref{eq:cond-Omeg}. Hence, we conclude that  $h\in C^{2-\alpha}(\T)$ and the proof of the range characterization follows immediately.
\\[0.5mm]
\textbf{(3)} Due to the fact
$\partial_\Omega \partial_r F(\Omega_{\ell m,b},0)h = h'$,
we find
\begin{align*}
  \partial_\Omega \partial_r F(\Omega_{\ell m,b},0)\cos (\ell m\theta) & = -\ell m \sin (\ell m \theta) \\
  & \not\in Range\big(\partial_rF(\Omega_{\ell m,b},0)\big),
\end{align*}
as claimed. This ends the proof of Proposition \ref{propos:bifurcation}.
\end{proof}

\section{Applications to geopghyscial flows}\label{sec:example}
In this section, we will examine special cases of \eqref{eq:ASE}-\eqref{eq:psi} covering crucial  models encountered in geophysical flows.
Through this exploration, we will  observe that  our comprehensive framework  often leads to   the known  results on the construction of V-states in the simply connected cases. Furthermore, we will derive new identities on special functions as a byproduct of our asymptotic description of the spectrum seen in Corollary \ref{cor:asym}.\subsection{2D Euler equation in the whole space}
Consider the 2D incompressible Euler equation in the whole plane. It corresponds to  the equation \eqref{eq:ASE}
with $\mathbf{D} = \mathbb{R}^2$ and $\psi = (-\Delta)^{-1}\omega$. Equivalently,
the stream function $\psi$ satisfies \eqref{eq:psi} with
\begin{align*}
  K(\mathbf{x},\mathbf{y})= K_0 (|\mathbf{x}-\mathbf{y}|) = - \tfrac{1}{2\pi} \log |\mathbf{x} -\mathbf{y}|.
\end{align*}
Although $K_0(t) = -\frac{1}{2\pi} \log t$, $t>0$ does not have a definite sign,
the function $-K'_0(t)=\frac{1}{2\pi}\frac{1}{t} $ is completely monotone
which has the following representation
\begin{align*}
  - K_0'(t) = \frac{1}{2\pi} \frac{1}{t} = \frac{1}{2\pi} \int_0^\infty e^{-tx}\dd x = \int_0^\infty e^{-tx} \dd \mu(x) ,
\end{align*}
that is, the associated non-negative measure $\mu$ is given by $\dd \mu(x) = \frac{1}{2\pi} \dd x$. Moreover,
$K_0$  satisfies the assumption  \eqref{cond:K0} with any $\alpha\in(0,1)$.
Thus the assumptions $(\mathbf{A}1)$-$(\mathbf{A}2)$ are verified and Theorem \ref{thm:main}
can be applied in this case. This gives Burbea result proved in \cite{Burbea82}. On the other hand,
in view of \eqref{spectral:Omega-whole} (or \eqref{eq:Omega-whole-s}) and through   straightforward calculus, using  for instance the identity   \cite[4.384]{GR15} we have
\begin{equation}\label{eq:lambd-n-euler}
\begin{split}
  \lambda_{n,1} = \int_0^{2\pi} K_0\big(|2\sin \tfrac{\eta}{2} |\big) \cos (n\eta) \dd \eta
  = & -\frac{1}{2\pi}\int_0^{2\pi}\log \big(\sin \tfrac{\eta}{2}\big) \cos (n\eta) \dd \eta \\
  =& \, \frac{1}{2n}
\end{split}
\end{equation}
and
\begin{align}\label{spectrum:Euler}
  \Omega_{n,1}^0 = \lambda_{1,1} - \lambda_{n,1} = \tfrac{1}{2} \big(1-\tfrac{1}{n}\big).
\end{align}
This is identical  to the result in \cite{Burbea82,HMV13}.
Using Corollary \ref{cor:asym}, we deduce that
\begin{align*}
  \lambda_{n,1}=\sum_{k=0}^{N}\frac{A_k}{n^{2k+1}}+\varepsilon_{n,N},
\end{align*}
where $A_k$ is independent of $n$ given by (the function  $\Psi_k$ is defined by \eqref{eq:psi-rela1}-\eqref{eq:psi-rela2})
\begin{align}\label{def:ck}
  A_k=\frac{1}{\pi}\int_0^{\infty}\frac{\Psi_k(x)}{x}\dd x,\quad k\in \mathbb{N},
\end{align}
and
\begin{align*}
  |\varepsilon_{n,N}|\leqslant C_{N,\delta} \frac{1}{n^{2N+\frac{5}{3}}}
  \int_0^\infty \frac{x^{\delta-1}}{1+ \frac{x}{n}} \dd x
  \leqslant C_{N,\delta}' \frac{1}{n^{2N + \frac{5}{3}-\delta}},\quad \delta\in (0,\tfrac{1}{3}).
\end{align*}
From \eqref{eq:lambd-n-euler}, we  infer that
\begin{align}\label{eq:ck}
  A_0=\frac{1}{2},\quad A_k=0,\quad \forall k\in \mathbb{N}^\star.
\end{align}
Note that the relations $A_0 =\frac{1}{2}$ and $A_1=A_2=0$ can be easily justified from the formula \eqref{def:ck}, 

\subsection{gSQG equation in the whole space}
The generalized surface quasi-geostrophic equation  in the plane, denoted by  gSQG equation, corresponds to
$$K(\mathbf{x},\mathbf{y}) = K_0(|\mathbf{x}-\mathbf{y}|) = {c_\beta} |\mathbf{x} - \mathbf{y}|^{-\beta},\quad c_\beta= \tfrac{ \Gamma(\frac{\beta}{2})}{\pi2^{2-\beta}\Gamma(1-\frac{\beta}{2})},\quad \beta\in(0,1)\cdot $$
Obviously, the function $t\in(0,\infty)\mapsto K_0(t)={c_\beta}t^{-\beta}$ satisfies the fact that $-K_0' $
is completely monotone with
\begin{align*}
  -K_0'(t) = {\beta c_\beta} t^{-\beta -1} = \frac{c_\beta}{ \Gamma(\beta)} \int_0^\infty e^{-tx} x^\beta
  \dd x = \int_0^\infty e^{-tx} \dd \mu(x),
\end{align*}
with the nonnegative measure $\mu$ given by $\dd \mu(x) = \frac{c_\beta}{ \Gamma(\beta)} x^\beta \dd x$.
Besides, the condition \eqref{cond:K0} holds true  for any $\alpha\in (0,1-\beta]$.
Consequently,  Theorem \ref{thm:main} can be applied in this case leading to the result of \cite{HH15}, with $b=1$.
\vskip1mm
\noindent Now, let us discuss  some identities that will mainly follow from Corollary \ref{cor:asym}.
The explicit computation of the spectrum, which will be detailed below, was conducted in \cite{HH15}.
For the sake of completeness, we shall  outline the main steps.
By using \eqref{lambda_n} and the following identity, see for example \cite[page 4]{MOS66},
\begin{align}\label{id:ZH}
  \forall \beta>-1,\;\forall \gamma\in \R,\quad \int_0^\pi(\sin \eta)^\beta  e^{i\gamma\eta}\dd\eta
  =\frac{\pi e^{i\frac{\gamma\pi}{2}}\Gamma(\beta+1)}{2^\beta\Gamma(1+\frac{\beta+\gamma}{2})\Gamma(1+\frac{\beta-\gamma}{2})},
\end{align}
we deduce that
\begin{align*}
  \lambda_{n,1} = \frac{c_\beta}{2\pi}\int_0^{2\pi}\frac{1}{|2\sin \frac{\eta}{2}|^\beta}\cos (n\eta)\dd \eta
  = & \frac{c_\beta}{\pi}\int_0^{\pi}\frac{1}{|2\sin \eta|^\beta}\cos (2n\eta)\dd \eta \\
  =&\frac{\Gamma(\frac{\beta}{2})}{2^{1-\beta}\Gamma(1-\frac{\beta}{2})}
  \frac{(-1)^{n}\Gamma(1-\beta)}{\Gamma(n+1-\frac{\beta}{2})\Gamma(1-n-\frac{\beta}{2})} \\
  = & \frac{\Gamma(1-\beta)}{2^{1-\beta}\Gamma^2(1-\frac{\beta}{2})}
  \frac{\Gamma(n+\frac{\beta}{2})}{\Gamma(n+1-\frac{\beta}{2})},
\end{align*}
where in the last line we have used the identity that (using the relation $\Gamma(1+\mathbf{z})=\mathbf{z}\Gamma(\mathbf{z})$)
\begin{align*}
  \frac{(-1)^{n}}{\Gamma(1-n-\frac{\beta}{2})}
  =\frac{\Gamma(\frac{\beta}{2}+n)}{\Gamma(1-\frac{\beta}{2})\Gamma(\frac{\beta}{2})}\cdot
\end{align*}
Thus from \eqref{spectral:Omega-whole} we have
\begin{align}\label{spectrum:gSQG}
  \Omega_{n,1}^0 = \lambda_{1,1} - \lambda_{n,1} = \frac{\Gamma(1-\beta)}{2^{1-\beta}\Gamma^2(1-\frac{\beta}{2})}
  \bigg(\frac{\Gamma(1+\frac{\beta}{2})}{\Gamma(2-\frac{\beta}{2})}
  -\frac{\Gamma(n+\frac{\beta}{2})}{\Gamma(n+1-\frac{\beta}{2})}\bigg),
\end{align}
which recovers the rotating angular velocity of gSQG equation proposed in \cite{CCG16b,HH15}.
\\[0.5mm]
Thanks to \eqref{eq:lamb-n}, Lemma \ref{lem:bound-varphi} and the fact that (e.g. see the identity \cite[3.241.2]{GR15})
\begin{align*}
  \int_0^{\infty}\frac{s^\beta}{1+s^2}\dd s=\frac{\pi}{2}\frac{1}{\sin \big(\frac{1+\beta}{2}\pi\big)}
  = \frac{\pi}{2} \frac{1}{\cos (\frac{\beta \pi}{2})},
\end{align*}
we can deduce that
\begin{align}\label{eq:lambda-gSQG}
  \frac{ 4n^2}{4n^2 +1}\frac{A_{\beta,0}}{n^{1-\beta}}\leqslant \lambda_{n,1}
  \leqslant \frac{4n^2}{4n^2 -1}\frac{A_{\beta,0}}{n^{1-\beta}},
\end{align}
with
\begin{align}\label{eq:c0-gSQG}
  A_{\beta,0} = \frac{2c_\beta}{ \Gamma(\beta)}\int_0^\infty \frac{s^\beta }{1 + s^2 } \dd s
  = \frac{\Gamma(\frac{\beta}{2})}{2^{2-\beta} \Gamma(\beta) \Gamma(1-\frac{\beta}{2})}
  \frac{1}{\cos (\frac{\beta \pi}{2})}
  = \frac{\Gamma(1-\beta)}{2^{1-\beta}\Gamma^2(1-\frac{\beta}{2})},
\end{align}
where in the last inequality we used the refection formula of Gamma function
\begin{align*}
  \Gamma(x) \Gamma (1-x) = \frac{\pi}{\sin (\pi x)}, \quad x\notin\Z.
\end{align*}
\vskip0.5mm
\noindent Using Proposition \ref{prop:asypt} and Corollary \ref{cor:asym},
we can easily deduce the  formula of $\lambda_{n,1}$:  
\begin{align*}
  \lambda_{n,1} = \frac{2c_\beta}{\Gamma(\beta)}
  \sum_{k=0}^N\frac{1}{n^{2k+1}}\int_0^\infty \Psi_k(\tfrac{ x}{n}) x^{\beta-1} \dd x + \varepsilon_{n,N}
  = \sum_{k=0}^{N}\frac{A_{\beta,k}}{n^{2k+1-\beta}} + \varepsilon_{n,N},
\end{align*}
where $\Psi_k$ is given by \eqref{eq:psi-rela1}-\eqref{eq:psi-rela2} and
\begin{align*}
  A_{\beta,k}  = \frac{2c_\beta}{ \Gamma(\beta)}\int_0^\infty \Psi_k(x) x^{\beta-1}\dd x
\end{align*}
where
\begin{align*}
  |\varepsilon_{n,N}|\leqslant \frac{c_\beta C_N}{ \Gamma(\beta)} \frac{1}{n^{2N+\frac{5}{3}}}
  \int_0^\infty \frac{x^{\beta-1}}{1+\frac{x}{n}}
  \dd x \leqslant \frac{C_{N,\beta}}{n^{2N+\frac{5}{3}-\beta }}\cdot
\end{align*}
Note that for $k=0$, $A_{\beta,0}$ has the explicit formula shown by \eqref{eq:c0-gSQG}.
Therefore, we infer from  the formula \eqref{spectrum:gSQG}, the following  asymptotic expansion of the Wallis quotient
\begin{align}\label{eq:wallis1}
  \frac{\Gamma(n+\frac{\beta}{2})}{\Gamma(n + 1 - \frac{\beta}{2})} = \frac{1}{n^{1-\beta}}
  + \sum_{k=1}^N \frac{A_{\beta,k}}{A_{\beta,0}} \frac{1}{n^{2k+1-\beta}} + O\Big(\frac{1}{n^{2N+\frac{5}{3}-\beta}} \Big).
\end{align}
On the other hand, recall  that the Wallis quotient  has the following expansion formula, see for instance \cite[p. 34]{Luke69} or
 \cite[Eq. (6.4)]{BE11},
\begin{align*}
  \frac{\Gamma(\mathbf{z}+a)}{\Gamma(\mathbf{z}+1-a)}
  = \mathbf{z}^{2a-1} \sum_{k=0}^N \frac{B^{(2a)}_{2k}(a) (1-2a)_{2k}}{(2k)!} \mathbf{z}^{-2k}
  +O(\mathbf{z}^{-2(N+1)+ 2a -1}),\quad |\arg \mathbf{z}|<\pi,
\end{align*}
where $B^{(s)}_n(t)$ stands for the generalized Bernoulli polynomials given by the following generating function
\begin{align*}
  \frac{x^s e^{tx}}{(e^x -1)^s} = \sum_{n=0}^\infty B^{(s)}_n(t) \frac{x^n}{n!},
\end{align*}
and $(t)_n$ is the Pochhammer symbol defined as
\begin{equation*}
  (t)_n \triangleq
\begin{cases}
  t(t+1)\cdots (t+n-1),\quad &\textrm{if}\;\;n\in\mathbb{N}^\star,\\
  1,\quad & \textrm{if}\;\; n=0.
\end{cases}
\end{equation*}
Thus we also have
\begin{align}\label{eq:wallis2}
  \frac{\Gamma(n+ \frac{\beta}{2})}{\Gamma(n+1 -\frac{\beta}{2})}
  = \sum_{k=0}^N \frac{B^{(\beta)}_{2k}(\frac{\beta}{2}) (1-\beta)_{2k}}{(2k)!} \frac{1}{n^{2k+1-\beta}}
  +O(n^{-2(N+1)+\beta -1}).
\end{align}
By comparing \eqref{eq:wallis1} and \eqref{eq:wallis2}, we deduce the following interesting identity:
\begin{align*}
  \frac{B^{(\beta)}_{2k}(\frac{\beta}{2}) (1-\beta)_{2k}}{(2k)!} A_{\beta,0} = A_{\beta,k},\quad \forall k\in\mathbb{N}.
\end{align*}
\vskip0.5mm
\noindent In addition, owing to \eqref{eq:bound-diff-phi-n} and \eqref{eq:lamb-n}, we find
\begin{align}\label{eq:lambda-gSQG-diff}
  \lambda_{n,1}-\lambda_{n+1,1}\approx \frac{2c_\beta}{ \Gamma(\beta)}
  \int_0^{\infty}\frac{(2n+1)x^\beta}{(n^2+x^2) \big((n+1)^2+x^2\big)} \dd x  \approx_\beta \frac{1}{n^{2-\beta}}.
\end{align}
Finally, applying Lemma \ref{lem:convex-lem} with $f(x) = \frac{c_\beta}{ \Gamma(\beta)} x^{\beta-1}$
gives  the convexity of $(\lambda_{n,1})_{n\geqslant 2}$, that is,
\begin{align*}
  \lambda_{n+1,1}+\lambda_{n-1,1}-2\lambda_{n,1}\geqslant 0,\quad \forall n\geqslant 2.
\end{align*}

\subsection{QGSW equation in the whole space}
Consider the QGSW equation in the whole plane, then it reduces to the equation \eqref{eq:ASE}
with $\mathbf{D}=\mathbb{R}^2$ and the stream function $\psi = ( -\Delta + \varepsilon^2)^{-1}\omega$, with $\varepsilon >0$ the deformation radius.
According to \cite{DHR19}, the kernel involved  \eqref{eq:psi} takes the form
\begin{align*}
  K(\mathbf{x},\mathbf{y}) = K_0(|\mathbf{x} -\mathbf{y}|) 
  = \frac{1}{2\pi} \mathbf{K}_0(\varepsilon |\mathbf{x} -\mathbf{y}|),
\end{align*}
where $\mathbf{K}_0$ is the modified Bessel function defined in Subsection \ref{subsec:Bessel}.
In view of \eqref{eq:Knu-Lap},
\begin{align*}
  K_0(t)=\frac{1}{2\pi}\mathbf{K}_0(\varepsilon t) = \frac{1}{2\pi} \int_1^\infty \frac{e^{- \varepsilon x t}}{\sqrt{x^2 -1}} \dd x,
\end{align*}
we obviously note that $-K_0'$ is completely monotone and by change of variables
\begin{align*}
  -K_0'(t) = \frac{1}{2\pi} \int_1^\infty \frac{\varepsilon x e^{-\varepsilon x t}}{\sqrt{x^2 -1}} \dd x
  = \frac{1}{2\pi} \int_\varepsilon^\infty e^{- tx} \frac{x }{\sqrt{x^2 -\varepsilon^2}}\dd x = \int_0^\infty e^{-tx} \dd \mu(x),
\end{align*}
with the nonnegative measure $\mu$ given by
\begin{align*}
  \dd \mu(x)=\frac{1}{2\pi}\frac{x}{\sqrt{x^2-\varepsilon^2}}\mathbf{1}_{\{x>\varepsilon\}} \dd x.
\end{align*}
Besides,  for $0\leqslant \alpha<1$,
\begin{align*}
  \int_0^{a_0} | K_0(t)| t^{-\alpha + \alpha^2} \dd t
  & = \frac{1}{2\pi} \int_1^\infty \frac{1}{\sqrt{x^2 -1}} \int_0^{a_0} e^{-\varepsilon x t} t^{-\alpha +\alpha^2} \dd t \dd x \\
  & \leqslant C_\alpha \int_1^\infty \frac{1}{\sqrt{x^2-1}} (\varepsilon x)^{-(1-\alpha +\alpha^2)} \dd x <\infty,
\end{align*}
which ensures that the condition \eqref{cond:K0} is verified. Hence, Theorem \ref{thm:main} can be applied in this case with any $\alpha\in (0,1)$ yielding to the result of \cite{DHR19}. Now, let us explore some other consequences.
\\[0.5mm]
By using \eqref{spectral:Omega-whole} and the identity \eqref{eq:Nicholson},
we can easily recover the result in \cite{DHR19}, namely,
\begin{align}\label{eq:Omeg-QGSW}
  \Omega_{n,1}^0 = \lambda_{1,1} - \lambda_{n,1} & = \frac{1}{2\pi} \int_{-\pi}^\pi \mathbf{K}_0\big(|2\varepsilon \sin \tfrac{\eta}{2}|\big)\cos \eta\dd \eta
  - \frac{1}{2\pi} \int_{-\pi}^\pi \mathbf{K}_0\big(|2 \varepsilon \sin \tfrac{\eta}{2}|\big)\cos(n\eta) \dd \eta \nonumber  \\
  & = \frac{2}{\pi} \int_0^{\frac{\pi}{2}} \mathbf{K}_0(2\varepsilon \sin \eta) \cos (2\eta) \dd \eta
  - \frac{2}{\pi} \int_0^{\frac{\pi}{2}} \mathbf{K}_0(2\varepsilon \sin \eta) \cos(2 n\eta) \dd \eta \nonumber \\
  & = - \frac{2}{\pi} \int_0^{\frac{\pi}{2}} \mathbf{K}_0(2\varepsilon \cos \eta) \cos (2\eta) \dd \eta
  - \frac{2(-1)^n}{\pi} \int_0^{\frac{\pi}{2}} \mathbf{K}_0(2\varepsilon \cos \eta) \cos(2 n\eta) \dd \eta \nonumber \\
  & = \mathbf{I}_1(\varepsilon)\mathbf{K}_1(\varepsilon)-\mathbf{I}_{n}(\varepsilon)\mathbf{K}_{n}(\varepsilon).
\end{align}
\vskip0.5mm
\noindent Lemma \ref{lem:bound-varphi} and \eqref{eq:lamb-n} yield
\begin{align*}
  \frac{4n^2}{4n^2 + 1} \frac{1}{\pi}
  \int_\varepsilon^{\infty}\frac{1}{\sqrt{x^2-\varepsilon^2}}
  \frac{x}{n^2+x^2}\dd x
  \leqslant \lambda_{n,1} \leqslant \frac{4n^2}{4n^2 -1} \frac{1}{\pi}
  \int_\varepsilon^{\infty}\frac{1}{\sqrt{x^2-\varepsilon^2}}
  \frac{x}{n^2+x^2}\dd x .
\end{align*}
From the explicit value
\begin{align*}
  \int_{\varepsilon}^{\infty}\frac{1}{\sqrt{x^2-\varepsilon^2}} \frac{x}{n^2+x^2}\dd x
  = \frac{\pi}{2} \frac{1}{\sqrt{n^2 + \varepsilon^2}},
\end{align*}
we find that for $n\in\mathbb{N}^\star$,
\begin{align}\label{eq:lamb-QGSW-bdd}
  \frac{2n^2}{4n^2+1} \frac{1}{\sqrt{n^2 +\varepsilon^2}} \leqslant
  \lambda_{n,1} \leqslant \frac{2n^2}{4n^2-1} \frac{1}{\sqrt{n^2 +\varepsilon^2}}.
\end{align}
The inequality \eqref{eq:bound-diff-phi-n} implies that for $n\in\mathbb{N}^\star$,
\begin{align}\label{eq:lamb-QGSW-diff}
  \lambda_{n,1}-\lambda_{n+1,1} & \approx \int_\varepsilon^\infty \frac{x}{\sqrt{x^2 -\varepsilon^2}}
  \frac{2n+1}{(n^2+x^2) \big((n+1)^2+x^2\big)} \dd x \nonumber \\
  & \approx \int_\varepsilon^\infty \frac{x}{\sqrt{x^2 -\varepsilon^2}} \frac{1}{n^2 + x^2} \dd x
  - \int_\varepsilon^\infty \frac{x}{\sqrt{x^2 -\varepsilon^2}} \frac{1}{(n+1)^2 + x^2} \dd x \nonumber \\
  & \approx \frac{1}{\sqrt{n^2+\varepsilon^2}} - \frac{1}{\sqrt{(n+1)^2 + \varepsilon^2}}\cdot
\end{align}
According to Corollary \ref{cor:asym}, we infer that
\begin{equation}\label{eq:lamb-QGSW}
\begin{split}
  \lambda_{n,1} = \mathbf{I}_n(\varepsilon) \mathbf{K}_n(\varepsilon)
  & = \frac{1}{\pi} \sum_{k=0}^{N}\frac{1}{n^{2k+1}} \int_\varepsilon^\infty \Psi_k(\tfrac{ x}{n})
  \frac{1}{\sqrt{x^2 -\varepsilon^2}}\dd x
  +\varepsilon_{n,N} \\
  & = \frac{1}{\pi} \sum_{k=0}^{N}\frac{1}{n^{2k+1}} \int_{\frac{\varepsilon}{n}}^\infty \Psi_k(x)
  \frac{1}{\sqrt{x^2 -\varepsilon^2/n^2}}\dd x
  + \varepsilon_{n,N},
\end{split}
\end{equation}
where $\Psi_k$ is given by \eqref{eq:psi-rela1}-\eqref{eq:psi-rela2} and
\begin{align*}
  |\varepsilon_{n,N}| \leqslant
  \frac{C_N}{n^{2N+\frac{5}{3}}}\int_\varepsilon^\infty \frac{1}{1+\frac{x}{n}}
  \frac{1}{\sqrt{x^2 -\varepsilon^2}}   \dd x
  \leqslant \frac{C_{N,\varepsilon} (\log n+1)}{n^{2N+\frac{5}{3}}},
\end{align*}
with some constant $C_{N,\varepsilon}>0$ independent of $n$. Note that the first term on the right-hand side of \eqref{eq:lamb-QGSW}
is $\frac{1}{2 \sqrt{n^2 + \varepsilon^2}}$, and direct calculations give $\Psi_1(x) = \frac{x(x^4 -6x^2 +1)}{4(1+x^2)^4}$,
\begin{align*}
  \int_{\frac{\varepsilon}{n}}^\infty \Psi_1(x) \frac{1}{\sqrt{x^2 - \varepsilon^2/n^2}} \dd x
  & = \frac{n^3}{(\sqrt{n^2+\varepsilon^2})^3} \Big(\frac{\pi}{16} - \frac{3\pi}{8} \frac{n^2}{n^2+\varepsilon^2}
  + \frac{5 \pi}{16} \frac{n^4}{(n^2 + \varepsilon^2)^2}\Big) \\
  & = \frac{n^3}{(\sqrt{n^2+\varepsilon^2})^3} \Big( -\frac{\pi}{8} \frac{\varepsilon^2}{n^2 + \varepsilon^2}
  + \frac{5\pi}{16} \frac{\varepsilon^4}{(n^2 +\varepsilon^2)^2}\Big),
\end{align*}
and
\begin{align*}
  \int_{\frac{\varepsilon}{n}}^\infty \Psi_2(x) \frac{1}{\sqrt{x^2 -\varepsilon^2/n^2}} \dd x
  = O\Big(\frac{\varepsilon^2}{n^2 +\varepsilon^2} \Big).
\end{align*}
Thus we have
\begin{align}\label{eq:lamb-QGSW2}
  \lambda_{n,1} = \mathbf{I}_n(\varepsilon) \mathbf{K}_n(\varepsilon)
  = \frac{1}{2\sqrt{n^2+\varepsilon^2}}
  + O\Big(\frac{1}{n^5} \Big) .
\end{align}
\vskip0.5mm
\noindent Finally, let us make a  remark on the convexity of the spectrum $(\lambda_{n,1})_{n\geqslant 2}$.
First,  Lemma \ref{lem:convex-lem} does not apply to this case due to jump of the measure density at $\varepsilon$,
but through some numerical experiments one can conjecture that for $n\geqslant 2$,
\begin{align*}
  \lambda_{n+1,1}+\lambda_{n-1,1}-2\lambda_{n,1}\geqslant 0.
\end{align*}
So far we do not know how to rigorously prove this result, but as a simple application of \eqref{eq:lamb-QGSW2},
we can show the convexity result  for every $n\geqslant n_\varepsilon$ with some $n_\varepsilon\in\mathbb{N}$ sufficiently large.

\subsection{Euler-$\alpha$ equation in the whole space}\label{subsec:Euler-alph}
The Euler-$\alpha$ equation is a regularization of 2D Euler equation
and it has been introduced in the context of averaged fluid models, see \cite{HMR98a,HMR98b}.
By considering its vorticity form in the whole plane, it corresponds to the equation \eqref{eq:ASE}
with $\mathbf{D} =\mathbb{R}^2$ and the stream function $\psi = (-\Delta)^{-1} \omega - ( -\Delta + \frac{1}{\alpha^2})^{-1}\omega$, for $\alpha>0$.\\
The kernel involved in \eqref{eq:psi} takes the form
\begin{align*}
  K(\mathbf{x},\mathbf{y}) = K_0 (|\mathbf{x}-\mathbf{y}|) =
  -\tfrac{1}{2\pi} \log |\mathbf{x}-\mathbf{y}| - \tfrac{1}{2\pi} \mathbf{K}_0(\tfrac{1}{\alpha}|\mathbf{x}-\mathbf{y}|).
\end{align*}
Thus,  $K_0(t) = -\frac{1}{2\pi}\log|t|-\frac{1}{2\pi }\mathbf{K}_0(\frac{1}{\alpha}|t|)$ satisfies that
\begin{align*}
  -K'_0(t)=\frac{1}{2\pi}\int_0^{\infty}e^{-tx}
  \Big(1-\frac{ x \alpha\mathbf{1}_{\{x> \frac{1}{\alpha}\}}}{\sqrt{x^2\alpha^2-1}}\Big)\dd x,
\end{align*}
which implies that  $-K'_0$ is not completely monotone, and  Theorem \ref{thm:main} cannot be applied for any symmetry $m\in\N^\star$. However, this theorem occurs for higher symmetry $m$.
Indeed, using \eqref{spectral:Omega-whole}, \eqref{spectrum:Euler}, \eqref{eq:Omeg-QGSW} and
\eqref{eq:lamb-QGSW2}, we can deduce that $(\Omega_{n,1})$ is strictly increasing for every $n\geqslant n_\alpha$ with some
$n_\alpha\in\N$ large enough. Hence, we may check that all the assumptions of Crandall-Rabinowitz's theorem work well.
Note that, in a recent work \cite{Rou23b} the strict monotonicity of $(\Omega_{n,1})$ is satisfied  for all the range $n\in \mathbb{N}^\star$,
and the author obtained the existence of $m$-fold symmetric V-states for the Euler-$\alpha$ equation.
Actually, the monotonicity follows directly from  the explicit formula of the spectrum which  takes the following form
\begin{align*}
  \Omega_{n,1}=\frac{2n-1}{2n}-\Big(\mathbf{I}_1\big(\tfrac{1}{\alpha}\big)
  \mathbf{K}_1\big(\tfrac{1}{\alpha}\big)-\mathbf{I}_n
  \big(\tfrac{1}{\alpha}\big) \mathbf{K}_{n}\big(\tfrac{1}{\alpha}\big)\Big).
\end{align*}

\subsection{2D Euler equation in the unit disc}

Consider the 2D incompressible Euler equation in the vorticity form in the unit disc $\mathbb{D}$ with rigid boundary condition
(the non-penetration boundary condition),
that is, the equation \eqref{eq:ASE} in $\mathbf{D} = \mathbb{D}$
with the stream function $\psi$ solving the Dirichlet problem in the unit disc
\begin{align}\label{eq:Euler-bdr-cd}
  - \Delta \psi = \omega,\quad \textrm{in}\; \mathbb{D},\qquad \psi |_{\partial \mathbb{D}} =0.
\end{align}
It is classical that the expression formula \eqref{eq:psi} holds with the Green function $K$ given by
\begin{align*}
  K(\mathbf{x},\mathbf{y})= - \tfrac{1}{2\pi} \log \tfrac{|\mathbf{x} - \mathbf{y}|}{ |1-\mathbf{x} \overline{\mathbf{y}}
  |} = - \tfrac{1}{2\pi} \log |\mathbf{x} -\mathbf{y}|
  + \tfrac{1}{2\pi} \log|1-\mathbf{x} \overline{\mathbf{y}}
  |,\quad \mathbf{x},\mathbf{y}\in \mathbb{D} .
\end{align*}
Clearly, $t\in(0,\infty)\mapsto K_0(t) = - \frac{1}{2\pi} \log t$ satisfies the assumptions $(\mathbf{A}1)$-$(\mathbf{A}2)$ with $\alpha\in (0,1)$,
and the perturbative kernel $(\mathbf{x},\mathbf{y})\in\D^2\mapsto K_1(\mathbf{x},\mathbf{y}) =
\frac{1}{2\pi} \log|1-\mathbf{x} \overline{\mathbf{y}}
 |$
is smooth and satisfies the assumptions $(\mathbf{A}3)$-$(\mathbf{A}4)$.
Hence, Theorem \ref{thm:perturbative} can be applied to the study of V-states
around the Rankine vortices $\mathbf{1}_{b\mathbb{D}}$ ($0<b<1$) leading to  the bifurcation  for large symmetry $m$.
Actually, as we shall see below, we are able to retrieve all the symmetry $m\geqslant 1$.
This allows to replicate  the full result obtained in \cite{DHHM}.
To start, we note that  the quantity
\begin{align*}
  G_1(\rho_1,\theta,\rho_2,\eta)
  = K_1(\rho_1 e^{i\theta},\rho_2 e^{i\eta}) =
  \tfrac{1}{2\pi} \log \big( |\rho_1\rho_2 e^{i\theta} - e^{i\eta} |\big),
\end{align*}
satisfies the property that $\partial_{\rho_1} G_1(\rho_1,\theta,\rho_2,\eta)
= \frac{\rho_2}{\rho_1} \partial_{\rho_2} G_1(\rho_1,\theta,\rho_2,\eta)$,
and owing to \eqref{spectral:Omega-general}, \eqref{eq:lambd-n-euler} and the following fact (see e.g. 4.397 of \cite{GR15})
\begin{align*}
  \int_0^{2\pi} \log(1- 2 a \cos \eta + a^2) \cos (n\eta) \dd \eta = - \frac{2\pi}{n} a^n,\quad |a|<1,
\end{align*}
\begin{align*}
  \int_0^{2\pi} \log(1- 2 a \cos \eta + a^2) \dd \eta =0,\quad |a|\leqslant 1,
\end{align*}
we obtain through integration by parts
\begin{align*}
  \Omega_{n,b} & = \frac{n-1}{2n} -b^{-1} \int_0^{2\pi} \int_0^b \partial_{\rho_1}G_1(b,0,\rho,\eta) \rho \dd \rho \dd \eta
  - \int_0^{2\pi} G_1(b,0,b,\eta)\cos (n\eta)\dd \eta \\
  & = \frac{n-1}{2n} -  b^{-2} \int_0^{2\pi}\int_0^b
  \partial_\rho G_1(b,0,\rho,\eta) \rho^2 \dd \rho \dd \eta
  - \frac{1}{4\pi} \int_0^{2\pi} \log \big(b^4 +1 - 2 b^2\cos\eta\big)\cos (n\eta)\dd \eta \\
  & = \frac{n-1 + b^{2n}}{2n} -  \int_0^{2\pi} G_1(b,0,b,\eta) \dd \eta
  + 2 b^{-2} \int_0^{2\pi} \int_0^b G_1(b,0,\rho,\eta) \rho \dd \rho \dd \eta,
\end{align*}
implying that
\begin{align*}
  \Omega_{n,b}
  & = \frac{n-1 + b^{2n}}{2n} - \frac{1}{4\pi} \int_0^{2\pi} \log(b^4 + 1 - 2b^2 \cos \eta) \dd \eta \\
  & \quad +  \frac{b^{-2}}{2\pi} \int_0^b \int_0^{2\pi} \log(b^2\rho^2 +1 -2 b\rho \cos \eta) \dd \eta \,\rho \dd \rho \\
  & = \frac{n-1 + b^{2n}}{2n}\cdot
\end{align*}
This formula coincides  with the rotating angular velocity established in \cite{DHHM}. Direct calculation shows that
$(\Omega_{n,1})_{n\in \mathbb{N}^\star}$
is strictly increasing, thus we can remove the restriction on $m$
in \mbox{Theorem \ref{thm:perturbative}} and recover the existence of $m$-fold symmetric V-states with $m\in\mathbb{N}^\star$
for the 2D Euler equation in the unit disc as in \cite[Theorem 1]{DHHM}.

\subsection{gSQG equation in the unit disc}

If we consider the gSQG equation in the unit disc $\mathbb{D}$ with rigid boundary,
it corresponds to the equation \eqref{eq:ASE} with $\mathbf{D} = \mathbb{D}$ and the stream function $\psi$ solving
\begin{align*}
  \psi = (- \Delta)^{-1 + \frac{\beta}{2}} \omega,\quad \textrm{in}\; \mathbb{D},\qquad \psi |_{\partial \mathbb{D}} = 0,
  \quad \beta\in (0,1).
\end{align*}
Equivalently, $\psi$ satisfies
\begin{equation}\label{eq:psi-exp-qg}
	\begin{split}
		\psi(\mathbf{x})=\frac{1}{\Gamma(1-\frac{\beta}{2})}
		\int_0^{\infty} t^{-\frac{\beta}{2}} e^{t\Delta} \omega(\mathbf{x}) \dd t
		=\int_{\mathbf{D}}K(\mathbf{x},\mathbf{y}) \omega(\mathbf{y}) \dd \mathbf{y},
	\end{split}
\end{equation}
where the Laplacian $\Delta$ is defined on $\mathbb{D}$
with Dirichlet boundary condition.
According to \cite[Lemma 2.3]{HXX23}, the spectral Green function $K$ satisfies
\begin{align*}
  K(\mathbf{x},\mathbf{y}) = K_0(|\mathbf{x} - \mathbf{y}|) + K_1(\mathbf{x},\mathbf{y})
  =  \tfrac{\Gamma(\frac{\beta}{2})}{2^{2-\beta} \pi\Gamma(1-\frac{\beta}{2})} |\mathbf{x} - \mathbf{y}|^{-\beta}
  + K_1(\mathbf{x},\mathbf{y}),
\end{align*}
and $K_1\in C^\infty(\mathbb{D}\times \mathbb{D})$. In view of Lemma 2.4 of \cite{HXX23},
$K_1(\mathbf{x},\mathbf{y})$ satisfies the assumptions $(\mathbf{A}3)$-$(\mathbf{A}4)$.
Hence, Theorem \ref{thm:perturbative} can be applied in this case to show the existence of $m$-fold
symmetric rotating patch solutions around trivial solution $\mathbf{1}_{b\mathbb{D}}$ ($0<b<1$)
with sufficiently large $m$, which is one of the main result in \cite{HXX23}.
On the other hand, by virtue of Lemma \ref{lem:disc-exp} and \eqref{eq:psi-exp-qg}, we have
\begin{align*}
  K(\mathbf{x},\mathbf{y})=\sum_{n\in\NN, k\in\mathbb{N}^\star} x_{n,k}^{\alpha-2}
  \Big( \phi_{n,k}^{(1)}(\mathbf{x})\phi_{n,k}^{(1)}(\mathbf{y}) + \phi_{n,k}^{(2)}(\mathbf{x})\phi_{n,k}^{(2)}(\mathbf{y}) \Big).
\end{align*}
For $\mathbf{x}= \rho_1 e^{i\theta}\in \mathbb{D}$, $\mathbf{y}=\rho_2 e^{i\eta}\in\mathbb{D}$, and using the notation \eqref{def:G},
we also have
\begin{align}\label{eq:K-ss2}
  K(\mathbf{x},\mathbf{y}) = G(\rho_1,\theta,\rho_2,\eta)
  = \sum_{n\in\mathbb{N}, k\in\mathbb{N}^\star} x_{n,k}^{\alpha-2} A_{n,k}^2 J_n(x_{n,k}\rho_1) J_n(x_{n,k}\rho_2)
  \cos \big(n(\theta-\eta)\big).
\end{align}
Recall that in Subsection \ref{subsec:lin} the spectrum $\Omega_{n,b} = - V[0] - \Lambda_{n,b}$ with
\begin{align}\label{eq:V0-Lamb}
  V[0] = b^{-1}\int_{0}^{2\pi}\int_0^b\Big(\nabla_\mathbf{x}K(be^{i\theta},\rho e^{i\eta})
  \cdot e^{i\theta}\Big)\rho \dd \rho \dd \eta,\quad \Lambda_{n,b}=\int_0^{2\pi} K(b,be^{i\eta})e^{in\eta}\dd \eta,
\end{align}
we can argue as \cite{HXX23} to show that
\begin{align*}
  V[0] = - 2\sum_{ k\geqslant 1}x_{0,k}^{\alpha-2}\frac{J_1^2\big(x_{0,k} b \big)}{J_1^2(x_{0,k} )},
  \quad
  \Lambda_{n,b} = 2\sum_{k\geqslant 1}x_{m,k}^{\alpha-2}\frac{J_m^2(x_{m,k}b)}{J_{m+1}^2(x_{m,k})}\cdot
\end{align*}
By using Sneddon's formula, Lemma 5.1 of \cite{HXX23} proves the strict monotonicity of $n\mapsto \Omega_{n,b}$ in either small $b$ case
or small $\alpha$ case, and it further implies the existence of $m$-fold symmetric V-states around $\mathbf{1}_{b\mathbb{D}}$ ($0<b<1$)
in both cases.

\subsection{QGSW equation in the unit disc}
Consider the QGSW model in the unit disc $\mathbb{D}$ with rigid boundary, and it corresponds to the equation \eqref{eq:ASE}
with $\mathbf{D} = \mathbb{D}$ and the relationship between $\psi$ and $\omega$ can be expressed by
\begin{align*}
  \psi=(-\Delta+\varepsilon^2)^{-1}\omega,
\end{align*}
which denotes the unique solution to the following Dirichlet problem,
\begin{align*}
  (-\Delta+\varepsilon^2)\psi=\omega\; \text{ in }\mathbb{D},\quad
  \psi|_{\partial\mathbb{D}}=0.
\end{align*}
In order to describe the associated  Green function,  we need to solve the equation for every
$\mathbf{x}\in \mathbb{D}$,
\begin{align*}
  -\Delta_\mathbf{y}K(\mathbf{x},\mathbf{y})+\varepsilon^2K(\mathbf{x},\mathbf{y})=\delta_\mathbf{x}(\mathbf{y})\;
  \text{ in }\mathbb{D},\quad
  {K(\mathbf{x},\cdot)|_{\partial \mathbb{D}}=0,}
\end{align*}
where $\delta_{\mathbf{x}}(\mathbf{y})$ is the Dirac measure centered at the point $\mathbf{x}$.
According to the spectral theory of elliptic problems, for example \cite[Sec. 6.2 and Sec. 6.3]{Evans10},
we infer that $(-\Delta+\varepsilon^2)^{-1}$ is well-defined and bounded from $L^2(\mathbb{D})$ to $H^2(\mathbb{D})$.
In addition, we can split the kernel as follows
$$K(\mathbf{x},\mathbf{y})=K_0(|\mathbf{x}-\mathbf{y}|)+K_1(\mathbf{x},\mathbf{y})\quad\hbox{with}\quad
K_0(|\mathbf{x} - \mathbf{y}|) = \tfrac{1}{2\pi}\mathbf{K}_0(\varepsilon|\mathbf{x} -\mathbf{y}|)
$$
and $K_1$ solves the elliptic problem
\begin{align*}
  -\Delta_\mathbf{y}K_1(\mathbf{x},\mathbf{y}) + \varepsilon^2 K_1(\mathbf{x},\mathbf{y})=0, \text{ in }\mathbb{D},\quad
  K_1(\mathbf{x},\mathbf{y})|_{\mathbf{y}\in \partial \mathbb{D}} = -\tfrac{1}{2\pi} \mathbf{K}_0(\varepsilon|\mathbf{x}-\mathbf{y}|).
\end{align*}
Since $\mathbf{K}_0$ is smooth except at $\mathbf{x}=0$,
by the classical regularity theory of elliptic PDE,
we have that $\mathbf{y}\in \D \mapsto K_1(\mathbf{x},\mathbf{y})$ is smooth for any $\mathbf{x}\in \mathbb{D}$.
Following exactly the same argument in \cite[p. 39]{Evans10}, we find $K(\mathbf{x},\mathbf{y}) = K(\mathbf{y},\mathbf{x})$.
Thus $K_1$ belongs to $C^\infty(\mathbb{D}\times \mathbb{D})$ and
moreover the geometric properties \eqref{Sym01} and \eqref{Sym1} in $\mathbb{D}$
can be easily checked by arguing as \cite[Lemma 2.4]{HXX23}.
Hence, for every $\varepsilon>0$ and $b\in (0,1)$,
we can apply Theorem \ref{thm:perturbative} to show the existence of $m$-fold symmetric V-states
for the QGSW equation in the unit disc $\mathbb{D}$ with $m$ large enough.
\\[0.5mm]
On the other hand, according to the  work developed on the unit disc $\D$, we actually  obtain an explicit formula for
$K$ expressed by series in terms of the eigenvalue-eigenfunction pairs of spectral problem \eqref{def:phi-j2}.
According to the spectral theory of second order elliptic PDE, e.g. see \cite[Sec. 6.5]{Evans10},
the eigenfunctions $\big(\phi^{(1)}_{n,k},\phi^{(2)}_{n,k}\big)_{n\in \N,k\in \mathbb{N}^\star}$ in Lemma \ref{lem:disc-exp}
form an orthonormal basis in $L^2(\D)$ and belong to $H_0^1(\D)$.
Then via a simple calculation, we infer that
\begin{align*}
  \psi(\mathbf{x})=\int_{\D}\sum_{n\in \N,k\in \mathbb{N}^\star}
  \frac{1}{x^2_{n,k}+\varepsilon^2}\Big(\phi^{(1)}_{n,k}(\mathbf{x})\phi^{(1)}_{n,k}(\mathbf{y})
  +\phi^{(2)}_{n,k}(\mathbf{x})\phi^{(2)}_{n,k}(\mathbf{y})\Big) \omega(\mathbf{y})\dd \mathbf{y}.
\end{align*}
Comparing with \eqref{eq:psi}
leads to
\begin{align}
  K(\mathbf{x},\mathbf{y}) = \sum_{n\in \N,k\in \mathbb{N}^\star}
  \frac{1}{x^2_{n,k}+\varepsilon^2}\Big(\phi^{(1)}_{n,k}(\mathbf{x})\phi^{(1)}_{n,k}(\mathbf{y})
  + \phi^{(2)}_{n,k}(\mathbf{x}) \phi^{(2)}_{n,k}(\mathbf{y})\Big).
\end{align}
Similarly to \eqref{eq:K-ss2}, we find
\begin{align}\label{eq:K-exp-QGSW-disc}
  K(\mathbf{x},\mathbf{y}) = G(\rho_1,\theta,\rho_2,\eta)
  = \sum_{n\in\mathbb{N}, k\in\mathbb{N}^\star} \frac{1}{x_{n,k}^2 + \varepsilon^2} A_{n,k}^2 J_n(x_{n,k}\rho_1) J_n(x_{n,k}\rho_2)
  \cos \big(n(\theta-\eta)\big),
\end{align}
with $A_{n,k}$ given by \eqref{eq:Ank}.
Concerning  the spectrum $\Omega_{n,b} = - V[0] - \Lambda_{n,b}$, in view of \eqref{eq:V0-Lamb} and \eqref{eq:K-exp-QGSW-disc},
and arguing as \cite[Eq. (113)]{HXX23}, we can show that
\begin{align}\label{eq:V0-QGSW-disc}
  V[0]=-2\sum_{k\in \mathbb{N}^\star}\frac{1}{x_{0,k}^2+\varepsilon^2}\frac{J_1^2(bx_{0,k})}{J_1^2(x_{0,k})},
\end{align}
and
\begin{align}\label{eq:lambda-QGSW-disc}
  \Lambda_{n,b} & = \int_0^{2\pi} \bigg(\sum_{\ell\in \mathbb{N},k\in \mathbb{N}^\star}
  \frac{1}{x_{\ell,k}^2 +\varepsilon^2} A_{\ell,k}^2 J_\ell^2(b x_{\ell,k})  \cos(\ell\eta)  \bigg) e^{i n\eta}\dd \eta \nonumber\\
  & = 2\sum_{k\in \mathbb{N}^\star}\frac{1}{x_{n,k}^2+\varepsilon^2}\frac{J_n^2(bx_{n,k})}{J_{n+1}^2(x_{n,k})}\cdot
\end{align}
\vskip0.5mm
\noindent Interestingly, as the Sneddon's formula used in \cite{HXX23}, there is also a suitable  summation formula for \eqref{eq:lambda-QGSW-disc}.
Choosing $X=Y=b$, $\nu=n$ and $\mathbf{z}=\varepsilon$ in the Kneser-Sommerfeld expansion \eqref{eq:formula-sum-Bessel},
we obtain
\begin{align}\label{eq:lambda-QGSW-disc2}
  \Lambda_{n,b} 
  = \mathbf{I}_n(b\varepsilon)\mathbf{K}_n(b\varepsilon) -
  \frac{\mathbf{K}_n(\varepsilon)}{\mathbf{I}_n(\varepsilon)}\mathbf{I}^2_n(b\varepsilon).
\end{align}
\noindent As $\varepsilon \rightarrow 0$, noting that $\sum_{k=1}^\infty \frac{1}{x_{0,k}^2} \frac{J_1^2(b x_{0,k})}{J_1^2(x_{0,k})} = \frac{1}{4} $
(using Sneddon's formula,
e.g. see \cite[Eq. (29)]{HXX23}), and applying  the asymptotics of $\mathbf{I}_n(x)$
and $\mathbf{K}_n(x)$ in \eqref{eq:InKn-asymp}, we deduce that the spectrum $\Omega_{n,b}$ in QGSW equation in the unit disc
satisfies
\begin{align*}
  \lim_{\varepsilon \rightarrow 0} \Omega_{n,b} = 2 \sum_{k=1}^\infty \frac{1}{x_{0,k}^2} \frac{J_1^2(b x_{0,k})}{J_1^2(x_{0,k})}
  - \lim_{\varepsilon \rightarrow 0} \Lambda_{n,b} = \frac{1}{2} - \frac{1 - b^{2n}}{2n},
\end{align*}
which coincides with the spectrum of 2D Euler equation in the unit disc.
\\[0.5mm]
The monotonicity of $(\Lambda_{n,b})_{n\in\mathbb{N}^\star}$ given by \eqref{eq:lambda-QGSW-disc2}
for every $b\in (0,1)$ and $\varepsilon >0$ is a crucial property and seems not easy to achieve.
Below, we show that for every $\varepsilon >0$ and $b\in (0,b_*)$ with some small $b_*\in (0,\frac{1}{2})$ depending only on $\varepsilon$,
such a sequence $(\Lambda_{n,b})_{n\in\mathbb{N}^\star}$ is strictly increasing with respect to $n$.
Notice that
\begin{align*}
  \Lambda_{n,b} - \Lambda_{n+1,b} & = \Big(\mathbf{I}_n(b\varepsilon) \mathbf{K}_n(b\varepsilon)
  - \mathbf{I}_{n+1}(b\varepsilon) \mathbf{K}_{n+1}(b\varepsilon)  \Big)
  \bigg( 1 - \frac{\mathbf{I}_n(\varepsilon)\mathbf{K}_n(\varepsilon)}{\mathbf{I}_n(b\varepsilon)\mathbf{K}_n(b\varepsilon)}
  \frac{\mathbf{I}_n^2(b\varepsilon)}{\mathbf{I}_n^2(\varepsilon)}\bigg) \\
  & \quad - \mathbf{I}_{n+1}(b\varepsilon) \mathbf{K}_{n+1}(b\varepsilon)
  \bigg( \frac{\mathbf{I}_n(\varepsilon)\mathbf{K}_n(\varepsilon)}{\mathbf{I}_n(b\varepsilon)\mathbf{K}_n(b\varepsilon)}
  \frac{\mathbf{I}_n^2(b\varepsilon)}{\mathbf{I}_n^2(\varepsilon)}
  - \frac{\mathbf{I}_{n+1}(\varepsilon)\mathbf{K}_{n+1}(\varepsilon)}{\mathbf{I}_{n+1}(b\varepsilon)\mathbf{K}_{n+1}(b\varepsilon)}
  \frac{\mathbf{I}_{n+1}^2(b\varepsilon)}{\mathbf{I}_{n+1}^2(\varepsilon)} \bigg) .
\end{align*}
By using \eqref{eq:lamb-QGSW-bdd}, \eqref{eq:lamb-QGSW-diff} and the following fact
\begin{align*}
  \forall n\in \mathbb{N}^\star,\; x>0,\; b\in(0,1),\qquad \mathbf{K}_n(x)>0,\quad\textrm{and}\quad 0< \mathbf{I}_n(bx) \leqslant b^n \mathbf{I}_n(x),
\end{align*}
we deduce that
\begin{align*}
  \frac{\mathbf{I}_n(\varepsilon)\mathbf{K}_n(\varepsilon)}{\mathbf{I}_n(b\varepsilon)\mathbf{K}_n(b\varepsilon)}
  \leqslant \frac{4n^2 +1}{4n^2-1} \frac{\sqrt{n^2 + (b\varepsilon)^2}}{\sqrt{n^2 + \varepsilon^2}} \leqslant \frac{5}{3},
\end{align*}
and
\begin{align*}
  \Lambda_{n,b} - \Lambda_{n+1,b}
  & \geqslant \frac{1}{4} \bigg(\frac{1}{\sqrt{n^2 + (b\varepsilon)^2}} - \frac{1}{\sqrt{(n+1)^2+ (b\varepsilon)^2}} \bigg)
  \Big(1- \frac{5}{3} b^{2n} \Big)
  - \frac{8}{17} \frac{1}{\sqrt{(n+1)^2 + (b\varepsilon)^2}} \frac{5}{3} b^{2n} \\
  & \geqslant \frac{2n + 1}{\sqrt{n^2 + (b\varepsilon)^2} ((n+1)^2 + (b\varepsilon)^2) }
  \bigg( \frac{1}{8} - \frac{5}{6} b^{2n} - \frac{40}{51} \frac{(n+1)^2 + (b\varepsilon)^2}{2n+1} b^{2n} \bigg).
\end{align*}
Thus, by taking  $b\leqslant \frac{1}{2}$, and setting $\sup_{n\in\mathbb{N}^\star}\frac{((n+1)^2 + (\varepsilon/2)^2)}{2n+1} \frac{1}{2^n}
= C(\varepsilon)$, we deduce that
\begin{align*}
  \Lambda_{n,b} - \Lambda_{n+1,b} \geqslant  \frac{2n + 1}{\sqrt{n^2 + (b\varepsilon)^2} ((n+1)^2 + (b\varepsilon)^2) }
  \bigg( \frac{1}{8} - \frac{5}{6} b^{2n} - \frac{40}{51} C(\varepsilon) b^n \bigg) .
\end{align*}
Hence, there exists a small constant $b_*\in (0,\frac{1}{2})$ depending only on $\varepsilon$
so that for every $\varepsilon>0$ and $b\in (0,b_*)$
the sequence $(\Lambda_{n,b})_{n\in \mathbb{N}^\star}$ is strictly increasing with respect to $n$.
With this property at hand, and for every $\varepsilon>0$ and $b\in (0,b_*)$,
we can show the existence of $m$-fold ($m\in \mathbb{N}^\star$) symmetric rotating solutions around $\mathbf{1}_{b \mathbb{D}}(\mathbf{x})$
for the QGSW equation in the unit disc. This result is completely new, in contrast to the models discussed before.

\section{Tools}\label{sec:tools}
In this section we shall collect some useful results used along the paper.
\subsection{Completely monotone functions}
This subsection is devoted to outlining  some properties of completely monotone functions.
We start with the following definition.
\begin{definition}\label{def:cmf}
A function $f:(0,\infty)\to \R$ is said to be completely monotone  if it  is of class $C^{\infty}$ and it satisfies
\begin{align*}
  (-1)^nf^{(n)}(t)\geqslant 0\qquad \forall t>0, \quad \forall  n\in \N.
\end{align*}
\end{definition}
\vskip0.5mm
\noindent The typical example is $f(t)= t^{-\alpha},$ with $\alpha\geqslant 0$.
One can refer for instance to \cite{RS10} for various examples of completely monotone functions.
\\[0.5mm]
The following result is fundamental in the theory of completely monotone functions. It gives a useful characterization through Laplace transform of Borel measure. For more details,  see for example Theorem 1.4 in \cite{RS10}.
\begin{theorem}[Bernstein's theorem]\label{lem:bernstein}
Let $f:(0,\infty)\mapsto \R$ be a completely monotone function. Then it is the Laplace transform of a unique nonnegative measure
$\mu$ on $[0,\infty)$, that is,
\begin{align*}
  \forall\, t>0,\quad f(t)=\int_{0}^{\infty}e^{-tx}\dd \mu(x)\triangleq \mathcal{L}(\mu)(t).
\end{align*}
Conversely, whenever $\mathcal{L}(\mu)(t)<\infty$ for every $t>0$, the function $t\mapsto \mathcal{L}(\mu)$ is a completely monotone function.
\end{theorem}
\noindent
The next goal  is to discuss useful pointwise estimates on  completely monotone functions.
\begin{lemma}\label{lem-high-deriv}
 The following assertions hold true.
\begin{enumerate}[(1)]
\item Let $f:(0,\infty)\mapsto \R$ be a completely monotone function. Then, for any $n\in\N$ and $\alpha\in(0,1)$ we have
\begin{align*}
  \forall\, t>0,\quad t^n|f^{(n)}(t)|\leqslant (\tfrac{n}{1-\alpha})^n f(\alpha{t}).
\end{align*}
\item Consider $f:(0,\infty)\mapsto \R$ such that $-f^\prime$   is completely monotone. Then, we have
\begin{align*}
  \forall\,0<t_1\leqslant t_2, \quad  0\leqslant f(t_1)-f(t_2)\leqslant (t_1-t_2)f^\prime(t_1).
\end{align*}
\end{enumerate}
\end{lemma}

\begin{proof}[Proof of Lemma \ref{lem-high-deriv}]
(1) By differentiation, we get
\begin{align*}
  t^nf^{(n)}(t)=(-1)^n\int_{0}^{\infty}(tx)^ne^{-tx}\dd \mu(x).
\end{align*}
Now, we use the inequality
\begin{align*}
  \forall s\geqslant 0, \quad  s^n\leqslant (\tfrac{n}{1-\alpha})^n e^{(1-\alpha)s},
\end{align*}
in order to get
\begin{align*}
  t^n|f^{(n)}(t)|&\leqslant (\tfrac{n}{1-\alpha})^n \int_{0}^{\infty}e^{-\alpha tx}\dd \mu(x)\\
  &\leqslant (\tfrac{n}{1-\alpha})^n f(\alpha t).
\end{align*}
(2) Using the identity \eqref{Expression-K} yields for any $a>0$ and $t>0$,
\begin{align*}
  f(t)  =&f(a)+\int_0^{\infty}\frac{e^{-t x}-e^{-ax}}{x}\dd {\mu}(x).
\end{align*}
Let $0<t_1\leqslant t_2$, then
\begin{align*}
  f(t_1)-f(t_2)=\int_{0}^{\infty}e^{-t_1x}\,\frac{1-e^{-(t_2-t_1)x}}{x}\dd \mu(x).
\end{align*}
At this stage we use the inequality
\begin{align*}
  \forall\, s\geqslant0,\quad 0\leqslant 1-e^{-s}\leqslant s,
\end{align*}
which implies that
\begin{align*}
  0\leqslant f(t_1)-f(t_2)\leqslant (t_2-t_1)\int_{0}^{\infty}e^{-t_1x}\dd \mu(x)=(t_1-t_2)f^\prime(t_1).
\end{align*}
This achieves the proof of the  desired result.
\end{proof}

\noindent Next, we intend to discuss the propagation of higher regularity/integrability of completely monotone functions, that will be used later.
\begin{lemma}\label{lem:int}
The following statements hold true.
\begin{enumerate}[(1)]
\item Assume that $f$ is a completely monotone function satisfying
\begin{equation}\label{eq:f-cond1}
  \int_0^{t_0} |f(t)| t^\beta \dd t <\infty,\quad \textrm{for some}\; \beta\in (-1,\infty)\;\textrm{and }\; t_0>0,
\end{equation}
then we have
\begin{align*}
  \int_0^{t_0}|f^{(k)}(t)| t^{k+\beta}\dd t \leqslant C_{k,\beta}\int_0^{t_0}f(t)t^\beta\dd t,\quad \forall k\in \mathbb{N}.
\end{align*}
\item
Assume that $f$ is a smooth function satisfying \eqref{eq:f-cond1} and $f'$ is with constant sign, then we have
\begin{align}\label{es:f-weig}
  \int_0^{t_0}|f'(t)| t^{1+\beta}\dd t\leqslant
  (1+\beta) \int_0^{t_0}|f(t)| t^\beta \dd x + |f(t_0)|t_0^{1+\beta}.
\end{align}
\end{enumerate}
\end{lemma}

\begin{proof}[Proof of Lemma \ref{lem:int}]
(1)
By virtue of Lemma \ref{lem-high-deriv}-(1) and the decreasing of $f$, we infer
\begin{align*}
	\int_0^{t_0}|f^{(k)}(t)| t^{k+\beta}\dd t\leqslant C_k\int_0^{t_0}f(\tfrac{t}{2})t^{\beta}\dd t\leqslant C_{k,\beta}\int_0^{t_0}f(t)t^\beta\dd t.
\end{align*}
Without loss of generality, we may suppose that $f'$ is non-positive, then $f$ is non-increasing.
If $\lim_{t\to 0^+}f(t)<\infty$, then $\lim_{t\to 0^+}f(t) t^{1+\beta}=0$. However
 if $\lim_{t\to 0^+}f(t)=\infty$, then $f(t)>0$ for sufficiently small $t>0$, and thus
\begin{align*}
  0\leqslant \lim_{t\to 0^+}f(t)t^{1+\beta} \leqslant \lim_{t\to 0^+}(1+\beta)\int_0^t f(s)s^{\beta}\dd s=0.
\end{align*}
Using  integration by parts we see that
\begin{equation*}
	\begin{split}
		-\int_0^{t_0}  f'(t) t^{1+\beta} \dd t
		& = - t^{1+\beta}_0 f(t_0) + \lim_{t\to 0^+} t^{1+\beta} f(t)  + (1+\beta) \int_0^{t_0} f(t) t^\beta \dd t \\
		& \leqslant  |f(t_0)|t_0^{1+\beta}+ (1+\beta) \int_0^{t_0} |f(t)| t^\beta \dd t,
	\end{split}
\end{equation*}
which yields the desired estimate \eqref{es:f-weig}.
\end{proof}
Next, we shall discuss  a result  which will be used frequently in the paper.
\begin{lemma}\label{lem:int2}
If $-f'$ is a completely monotone function on $(0,\infty)$ and $f$ satisfies \eqref{eq:f-cond1},
then for any $\alpha\in [0,1]$, $m,n\in\mathbb{N}^\star$, $t_0, c_1, c_2>0$, we have
\begin{align}\label{eq:int2-1}
  \int_0^{t_0}|f^{(m)}(c_1t)|^{\alpha}|f^{(n)}(c_2t)|^{1-\alpha}t^{m\alpha + n(1-\alpha)+\beta }\dd t
  \leqslant C\bigg(\int_0^{t_0}|f(t)| t^{\beta} \dd t
  + |f(t_0)| t_0^{1+\beta}\bigg),
\end{align}
and
\begin{align}\label{eq:int2-2}
  \int_0^{t_0}|f(c_1 t)|^\alpha|f^{(n)}(c_2t)|^{1-\alpha}t^{n(1-\alpha)+\beta }\dd t
  \leqslant C\bigg(\int_0^{(c_1\vee 1) t_0}|f(t)| t^\beta \dd t
  + |f(t_0)| t_0^{1+\beta}\bigg),
\end{align}
with $c_1\vee 1 \triangleq \max\{c_1,1\}$ and the constant $C>0$ depends on $m,n,\alpha,\beta,c_1,c_2$.
\end{lemma}

\begin{proof}[Proof of Lemma \ref{lem:int2}]
First, by virtue of Lemma \ref{lem:int},
we have that for $k\in \mathbb{N}^\star$,
\begin{align}\label{eq:fk-der-int}
  \int_0^{t_0}|f^{(k)}(t)|t^{k+\beta}\dd t\leqslant C \int_0^{t_0} \big(-f'(t)\big)t^{1+\beta} \dd t
 \leqslant C \Big(\int_0^{t_0}|f(t)|t^\beta \dd t + t_0^{1+\beta} | f(t_0)|\Big).
\end{align}
Since $(-1)^k f^{(k)}$ is non-negative and non-increasing,
then we get for $k\in\mathbb{N}^\star$ and  $c_1>0$,
\begin{equation}\label{eq:int-fk-fact}
\begin{split}
  \int_0^{t_0} |f^{(k)}(c_1 t)| t^{k+\beta} \dd t &\leqslant
  \begin{cases}
    \int_0^{t_0} (-1)^k f^{(k)}(t) t^{k+\beta} \dd t,\quad & \textrm{if}\;\; c_1\geqslant 1, \\
    c_1^{-k-1-\beta} \int_0^{c_1 t_0} (-1)^k f^{(k)}(t) t^{k+\beta} \dd t,
    \quad & \textrm{if}\;\; c_1 \leqslant 1,
  \end{cases} \\
  & \leqslant \max\big\{1, c_1^{-k-1-\beta}\big\} \int_0^{t_0}  |f^{(k)}(t)| t^{k+\beta} \dd t,
\end{split}
\end{equation}
and
\begin{align*}
  \int_0^{t_0} |f(c_1 t)| t^\beta \dd t = c_1^{\beta-1} \int_0^{c_1 t_0} |f(t)| t^\beta \dd t.
\end{align*}
Choosing $c=\tfrac{1}{2}\min\{c_1,c_2\}$, by Lemma \ref{lem-high-deriv}-(1) and the fact that $|f'(t)|$ is non-increasing, we have
\begin{align*}
  |f^{(m)}(c_1t)|^{\alpha}|f^{(n)}(c_2t)|^{1-\alpha}t^{m\alpha + n(1-\alpha)+\beta}\leqslant C_{m,n}|f'(ct)|t^{\beta}.
\end{align*}
Then \eqref{eq:int2-1} is a consequence of \eqref{eq:fk-der-int} and \eqref{eq:int-fk-fact}.
\\[0.5mm]
Now,  we move  to the proof of  \eqref{eq:int2-2}.
Using H\"older's inequality and Lemma \ref{lem-high-deriv}, we can deduce that for every $n\in\mathbb{N}^\star$,
\begin{align*}
  &\int_0^{t_0}|f(c_1t)|^{\alpha}|f^{(n)}(c_2t)|^{1-\alpha}t^{n(1-\alpha)+\beta }\dd t \\
  &\leqslant \bigg(\int_0^{t_0}|f(c_1t)t^{\beta}|\dd t\bigg)^\alpha
  \bigg(\int_0^{t_0}|f^{(n)}(c_2t)|t^{n+\beta}\dd t \bigg)^{1-\alpha}\\
  & \leqslant C \bigg(\int_0^{c_1t_0}|f(t) | t^{\beta} \dd t \bigg)^\alpha
  \bigg(\int_0^{t_0}|f^{(n)}(t)| t^{n+\beta}\dd t \bigg)^{1-\alpha} \\
  & \leqslant C \bigg(|t_0^{1+\beta}f(t_0)| +
  \int_0^{(c_1\vee 1)t_0}|f(t)t^{\beta}|\dd t\bigg),
\end{align*}
which corresponds to \eqref{eq:int2-2}.
\end{proof}

\subsection{Boundedness property of some operators on the torus}

In this subsection, we give  useful estimates for the following integral operator
\begin{align}\label{eq:int-operator}
  \mathcal{T}f(\theta)\triangleq \int_{\T}\mathbb{K}(\theta,\eta)f(\eta)\dd \eta,
\end{align}
where $\T=\R/{2\pi\Z}$ is the  torus, $\mathbb{K}:\T\times \T\to \C$ is the kernel function,
and $f$ is a $2\pi$-periodic function.
\begin{lemma}\label{lem:int-operator}
Let $\alpha\in (0,1)$, $n\in \mathbb{N}^\star$.
Assume the existence of  $C > 0$ and functions $H_1(\cdot)$, $\cdots$, $H_{n+1}(\cdot)$ satisfying
\begin{equation}\label{assum:K-H}
\begin{aligned}
  \int_{\T} H_k\big(\big|\sin \tfrac{\eta}{2}\big|\big)\dd \eta\leqslant C, \quad \forall k=1,2,\cdots, n,\\
  \int_{\T}  \big|H_{n}\big(|\sin \tfrac{\eta}{2}|\big)\big|^\alpha\,
  \big| H_{n+1} \big(|\sin \tfrac{\eta}{2}|\big)\big|^{1-\alpha} \dd \eta \leqslant C,
\end{aligned}
\end{equation}
such that $\mathbb{K}:\T\times \T\to \C$ satisfy
the following properties.
\begin{enumerate}[(1)]
\item $\mathbb{K}$ is measurable on $\T\times \T \setminus \{(\theta,\theta),\theta\in \T\}$ and
\begin{align*}
  |\mathbb{K}(\theta,\theta+ \eta)|\leqslant H_1\big(\big|\sin \tfrac{\eta}{2}\big|\big).
\end{align*}
\item For each $\eta\in \T$, the mapping $\theta\mapsto \mathbb{K}(\theta,\theta+\eta)$ is $n$-times differentiable in $\T$ and
\begin{align*}
  |\partial^{k}_\theta \big(\mathbb{K}(\theta,\theta+\eta)\big)|\leqslant H_{k+1}\big(\big|\sin \tfrac{\eta}{2} \big| \big),\quad  \forall k=1,\cdots,n.
\end{align*}
\end{enumerate}
Then the linear integral operator $\mathcal{T}$ given by \eqref{eq:int-operator}
is continuous from $C^{n-\alpha}(\mathbb{T})$ to $C^{n-\alpha}(\mathbb{T})$ and
\begin{align}\label{eq:T-Hold-es}
  \lVert \mathcal{T}f\rVert_{C^{n-\alpha}(\mathbb{T})}\le C_nC\lVert f\rVert_{C^{n-\alpha}(\mathbb{T})}.
\end{align}
\end{lemma}

\begin{proof}[Proof of Lemma \ref{lem:int-operator}]
The proof is by the induction method.
Making change of variables gives
\begin{align}\label{eq:Tf-exp2}
  \mathcal{T}f(\theta) = \int_{\T}\mathbb{K}(\theta,\theta+\eta)f(\theta+\eta)\dd \eta.
\end{align}
First, we start with  the  case $n=1$.
From \eqref{assum:K-H} we have
\begin{align*}
  |\mathcal{T}f(\theta)|\leqslant & \lVert f\rVert_{L^{\infty}} \int_{\T} |\mathbb{K}(\theta,\theta+\eta)|\dd \eta \\
  \leqslant & 
  \lVert f\rVert_{L^{\infty}} \int_{\T} H_1\big(\big|\sin \tfrac{\eta}{2} \big|\big)\dd \eta
  \leqslant C \lVert f\rVert_{L^{\infty}}.
\end{align*}
By using interpolation inequalities together with the mean value theorem, we infer
\begin{align*}
  & |\mathbb{K}(\theta_1,\theta_1+\eta)-\mathbb{K}(\theta_2,\theta_2+\eta)| \\
  & \leqslant \Big(|\mathbb{K}(\theta_1,\theta_1+\eta)|^{\alpha}+|\mathbb{K}(\theta_2,\theta_2+\eta)|^{\alpha}\Big)
  |\mathbb{K}(\theta_1,\theta_1+\eta)-\mathbb{K}(\theta_2,\theta_2+\eta)|^{1-\alpha} \\
  & \leqslant \Big(|\mathbb{K}(\theta_1,\theta_1+\eta)|^{\alpha}+|\mathbb{K}(\theta_2,\theta_2+\eta)|^{\alpha}\Big)
  \Big(\int_0^1 \big|\partial_{\theta_\tau}\big(\mathbb{K}(\theta_\tau,\theta_\tau +\eta)\big)\big| \dd \tau \Big)^{1-\alpha}
  |\theta_1-\theta_2|^{1-\alpha} \\
  & \leqslant 2\big|H_1\big(|\sin \tfrac{\eta}{2}|\big)\big|^\alpha\,
  \big| H_2 \big(|\sin \tfrac{\eta}{2}|\big)\big|^{1-\alpha} |\theta_1-\theta_2|^{1-\alpha},
\end{align*}
with $\theta_\tau = \tau \theta_1 + (1-\tau) \theta_2$. It follows that
\begin{align*}
  & |\mathcal{T}f(\theta_1)-\mathcal{T}f(\theta_2)| \leqslant\int_{\T}|\mathbb{K}(\theta_1,\theta_1+\eta)||f(\theta_1+\eta)-f(\theta_2+\eta)|\dd \eta\\
  & \quad +\int_{\T} |\mathbb{K}(\theta_1,\theta_1+\eta)-\mathbb{K}(\theta_2,\theta_2+\eta)|f(\theta_2+\eta)\dd \eta.
\end{align*}
Therefore,
\begin{align*}
   |\mathcal{T}f(\theta_1)-\mathcal{T}f(\theta_2)|
  & \leqslant|\theta_1-\theta_2|^{1-\alpha} \lVert f\rVert_{{C}^{1-\alpha}}
  \int_{\T}H_1\big(\big|\sin \tfrac{\eta}{2}\big|\big)\dd \eta\\
  &\quad + 2|\theta_1-\theta_2|^{1-\alpha}\lVert f\rVert_{\infty}\int_{\T}\big|H_1\big(|\sin \tfrac{\eta}{2}|\big)\big|^\alpha\,
  \big| H_2 \big(|\sin \tfrac{\eta}{2}|\big)\big|^{1-\alpha}\dd\eta \\
  & \leqslant 2C |\theta_1-\theta_2|^{1-\alpha}\lVert f\rVert_{C^{1-\alpha}}.
\end{align*}
Hence, combining the above estimates yields the desired inequality \eqref{eq:T-Hold-es} with $n=1$.
\\[0.5mm]
Now assuming that Lemma \ref{lem:int-operator} is true for $n=j$
and for the operator $\mathcal{T}$ given by \eqref{eq:Tf-exp2},
we prove that it also holds for $n=j+1$. Observe that
\begin{align}\label{eq:base-estiamte-T-j}
  \partial_{\theta}(\mathcal{T}f)(\theta) = \int_{\TT}\partial_{\theta}(\mathbb{K}(\theta,\theta+\eta))f(\theta+\eta) \dd \eta
  +\int_{\TT}\mathbb{K}(\theta,\theta+\eta)\partial_{\theta}f(\theta+\eta)\dd \eta.
\end{align}
In view of the fact that
\begin{align*}
  \int_{\TT}\big|H_{j}\big(|\sin \tfrac{\eta}{2}|\big)\big|^\alpha\,
  \big| H_{j+1} \big(|\sin \tfrac{\eta}{2}|\big)\big|^{1-\alpha}\dd \eta
  \leqslant \int_{\TT}\big|H_{j}\big(|\sin \tfrac{\eta}{2}|\big)\big|\dd \eta
  + \int_{\TT}\big|H_{j+1}\big(|\sin \tfrac{\eta}{2}|\big)\big|\dd \eta
  \leqslant C,
\end{align*}
and by the inductive hypothesis, we have
$\lVert \mathcal{T}(f)\rVert_{C^{j-\alpha}(\T)}\leqslant C_jC\lVert f\rVert_{C^{j-\alpha}(\T)}$ and
\begin{align*}
  \Big\lVert\int_{\TT}\mathbb{K}(\theta,\theta+\eta)\partial_{\theta}f(\theta+\eta)\dd \eta\Big\rVert_{C^{j-\alpha}(\T)}
  = \|\mathcal{T}(\partial_\eta f)\|_{C^{j-\alpha}(\T)}\leqslant C_jC\lVert f\rVert_{C^{j+1-\alpha}(\T)}.
\end{align*}
Noting that $\widetilde{\mathbb{K}}(\theta, \theta+\eta) = \partial_\theta(\mathbb{K}(\theta,\theta+\eta))$ satisfies
\begin{align*}
  |\partial_\theta^k\widetilde{\mathbb{K}}(\theta,\theta+ \eta)|\leqslant H_{k+2}\big(\big|\sin \tfrac{\eta}{2}\big|\big)
  \triangleq \widetilde{H}_{k+1}\big(|\sin\tfrac{\eta}{2}|\big), \quad \forall k=0,1,\cdots,j,
\end{align*}
and $\widetilde{H}_k$ ($k=1,2,\cdots,j+1$) satisfies \eqref{assum:K-H} with $n=j$ and $\widetilde{H}_k$ in place of $H_k$,
we use the induction hypothesis to deduce that
\begin{align*}
  \Big\lVert \int_{\TT}\partial_{\theta}(\mathbb{K}(\theta,\theta+\eta))f(\theta+\eta)\dd \eta\Big\rVert_{C^{j-\alpha}(\T)}
  & = \Big\lVert \int_{\TT} \widetilde{\mathbb{K}}(\theta,\theta+\eta) f(\theta+\eta)\dd \eta\Big\rVert_{C^{j-\alpha}(\T)} \\
  & \leqslant C_{j}C\lVert f\rVert_{C^{j-\alpha}(\T)}.
\end{align*}
Hence, we prove that
\begin{align*}
  \|\mathcal{T}f\|_{C^{j+1-\alpha}(\T)} = \|\mathcal{T}f\|_{C^{j-\alpha}(\T)} +
  \lVert \partial_{\theta}\mathcal{T}f\rVert_{C^{j-\alpha}(\T)}\leqslant 3C_jC\lVert f\rVert_{C^{j+1-\alpha}(\TT)}.
\end{align*}
The induction method guarantees that Lemma \ref{lem:int-operator} holds for every $n\in \mathbb{N}^\star$ and $\alpha\in(0,1)$.
\end{proof}

\noindent In the study of the linearized operator done before, we used the following Mikhlin multiplier type theorem
for an operator defined on a periodic function, see for instance   \cite[Theorem 4.5]{AB04}.
\begin{lemma}\label{lem:multiplier-lemma}
Given $\{a_n\}_{n\in \mathbb{Z}}$ and $h\in L^1(\T)$, and define the operator
\begin{align*}
  T h(\theta)=\sum_{n\in \mathbb{Z}} a_n \widehat{h}(n)e^{in\theta},
\end{align*}
where $\widehat{h}(n) = \int_{\T} h(\theta) e^{-in\theta} \dd \theta$ is the $n$-th Fourier coefficient of $h$.
Assume that
\begin{align*}
  \sup_{n\in \Z}|a_n|<\infty,\quad\textrm{and}\quad \sup_{n\in \Z}|n(a_{n+1}-a_n)|<\infty,
\end{align*}
then the operator $T$ is bounded in $C^{k+\alpha}(\T)$, for any $k\in \N$ and $\alpha\in (0,1)$.
\end{lemma}

\subsection{Bessel functions and Hankel transform}\label{subsec:Bessel}

In this subsection we collect some useful properties about Bessel functions and Hankel transform.
We recall for instance from  \cite[Chapter 3]{Wat96} that
\begin{align}
  J_\nu(\mathbf{z})=\sum_{n=0}^{\infty}\frac{(-1)^n(\frac{\mathbf{z}}{2})^{\nu+2n}}{n!(\Gamma(\nu+n+1))},\quad \forall \mathbf{z},\nu \in \C,\\
  J_{\nu -1}(\mathbf{z})-J_{\nu+1}(\mathbf{z})=2J'_{\nu}(\mathbf{z}),\quad \forall \mathbf{z},\nu \in \C,\label{bessel:recurrence} \\
  \frac{\dd}{\dd \mathbf{z}}(\mathbf{z}^{\nu}J_{\nu}(\mathbf{z}))=\mathbf{z}^\nu  J_{\nu-1}(\mathbf{z}),\quad \forall \mathbf{z},\nu \in \C.\label{bessel:derivative}
\end{align}
In particular, when $\nu = n$ is an integer, then we have according to  \cite[Chapter 2]{Wat96}
\begin{align*}
  J_n(\mathbf{z})& = \frac{1}{2\pi}\int_0^{2\pi}\cos(n\theta - \mathbf{z}\sin \theta)\dd \theta,\quad \forall \mathbf{z}\in \C, \\
  J_{-n}(\mathbf{z}) & = (-1)^n J_n(\mathbf{z}),\quad \forall \mathbf{z} \in \C.
\end{align*}
Next, we shall introduce  Bessel functions of imaginary argument also called modified Bessel functions of first and second kind, see for instance \cite[p. 66]{MOS66},
\begin{equation}\label{I-m-q}
  \mathbf{I}_\nu(\mathbf{z})=\sum_{n=0}^\infty \frac{\left(\frac{\mathbf{z}}{2}\right)^{\nu+2n}}{n!\Gamma(\nu+n+1)},\quad \nu\in \C,\, |\mbox{arg}(\mathbf{z})|<\pi,
\end{equation}
and
\begin{align*}
  \mathbf{K}_\nu(\mathbf{z})=\frac{\pi}{2}\frac{\mathbf{I}_{-\nu}(\mathbf{z})-
  \mathbf{I}_\nu(\mathbf{z})}{\sin(\nu\pi)},\quad\nu\in\mathbb{C}\setminus\mathbb{Z},\,|\mbox{arg}(\mathbf{z})|<\pi.
\end{align*}
When $\nu=j\in\mathbb{Z}$, $\mathbf{K}_{j}$ is defined through the formula
$\mathbf{K}_j(\mathbf{z})=\displaystyle\lim_{\nu\rightarrow j}\mathbf{K}_\nu(\mathbf{z})$. From   \cite[8.432.1]{GR15}, we recall the following integral representation
\begin{align}\label{eq:Knu-exp}
  \mathbf{K}_\nu(\mathbf{z}) = \int_0^{\infty} e^{-\mathbf{z}\cosh s} \cosh(\nu s) \dd s,\quad
  \forall \nu \in \C,\,|\mathrm{arg}(\mathbf{z})|<\tfrac{\pi}{2}.
\end{align}
Another useful identity that can be found in \cite[8.432.3]{GR15} deals with the representation in terms of Laplace transform,\begin{align}\label{eq:Knu-Lap}
  \mathbf{K}_\nu(\mathbf{z})= \frac{(\frac{\mathbf{z}}{2})^\nu \Gamma(\frac{1}{2})}{\Gamma(\nu+\frac{1}{2})}
  \int_1^{\infty} e^{-s\mathbf{z}} (s^2-1)^{\nu -\frac{1}{2}} \dd s>0,\quad \mathrm{Re}\,\big(\nu+\tfrac{1}{2}\big)>0,\,
  |\mathrm{arg}(\mathbf{z})|<\tfrac{\pi}{2}.
\end{align}
For $\mathbf{I}_n(x)$ and $\mathbf{K}_n(x)$, we have the asymptotic expansion of small argument (e.g. see \cite[p. 375]{AbS64})
\begin{align}\label{eq:InKn-asymp}
  \forall n\in \mathbb{N}^\star,\quad \mathbf{I}_n(x) \stackrel{x\rightarrow 0}\sim \frac{(\frac{1}{2}x)^n}{\Gamma(n+1)},
  \quad \textrm{and}\quad \mathbf{K}_n(x) \stackrel{x\rightarrow 0}\sim \frac{\Gamma(n)}{2 (\frac{1}{2}x)^n}\cdot
\end{align}
The following Nicholson's integral representation of $\mathbf{I}_n(\mathbf{z})\mathbf{K}_n(\mathbf{z})$
is useful in the sequel, see for instance \cite[p. 441]{Wat96}. For $n\in\mathbb{N}$,
\begin{align}\label{eq:Nicholson}
  \mathbf{I}_n(\mathbf{z})\mathbf{K}_n(\mathbf{z})
  =\frac{2(-1)^n}{\pi}\int_0^{\frac{\pi}{2}}\mathbf{K}_0(2\mathbf{z}\cos \theta)\cos (2n\theta)\dd \theta.
\end{align}
\vskip0.5mm
\noindent The following  useful result states that the eigenvalues and eigenfunctions of the spectral Laplacian $-\Delta$
on the unit disc $\mathbb{D}\subset \mathbb{R}^2$ have precise expression formula through Bessel functions
(e.g. see Section 5.5 of Chapter V in \cite{CH09}).
\begin{lemma}\label{lem:disc-exp}
The eigenvalues and the eigenfunctions solving the spectral problem
\begin{align}\label{def:phi-j2}
  \hbox{for}\; j\geqslant1,\,\;\;-\Delta\phi_j = \lambda_j \phi_j,\quad \phi_j |_{\partial \mathbb{D}} =0,\quad
  \int_{\mathbb{D}}\phi_j^2(x) \dd x=1.
\end{align}
are described by double index families $(\lambda_{n,k})_{n\in \mathbb{N},k\in \mathbb{N}^\star}$
and $\big((\phi_{n,k}^{(1)}, \phi_{n,k}^{(2)})\big)_{n\in\mathbb{N},k\in\mathbb{N}^\star}$ such that
\begin{align}
  \lambda_{n,k}=x_{n,k}^2, \quad \phi_{n,k}^{(1)}(x)=J_n(x_{n,k}|x|)A_{n,k}\cos(n\theta),
  \quad \phi_{n,k}^{(2)}(x)=J_n(x_{n,k}|x|)A_{n,k}\sin(n\theta),
\end{align}
where
\begin{align}\label{eq:Ank}
  \pi A_{0,k}^2=\frac{1}{J_{1}^2(x_{0,k})}\quad\hbox{and}\quad
  \pi A_{n,k}^2=\frac{2}{J_{n+1}^2(x_{n,k})},\quad\forall n\in\mathbb{N}^\star,
\end{align}
and $J_n$  denotes the Bessel function of order $n$  and $(x_{n,k})_{k\in\mathbb{N}^\star}$ are its zeroes.
\end{lemma}
\noindent We also have the following Kneser-Sommerfeld expansion (e.g. see \cite[Eq. (12)]{Mar22} or \cite[p. 134]{MOS66})
involving the zeros of Bessel functions:
\begin{align}\label{eq:formula-sum-Bessel}
  \sum_{k=1}^\infty \frac{1}{\mathbf{z}^2+x^2_{\nu,k}}\frac{J_\nu(X x_{\nu,k})J_\nu(Y x_{\nu,k})}{J^2_{\nu+1}(x_{\nu,k})}
  =  \frac{1}{2}\frac{\mathbf{I}_\nu(X\mathbf{z})}{\mathbf{I}_\nu(\mathbf{z})}
  \Big(\mathbf{I}_\nu(z) \mathbf{K}_\nu(Y\mathbf{z})-\mathbf{K}_\nu(\mathbf{z})\mathbf{I}_\nu(Y\mathbf{z})\Big),
\end{align}
where $(x_{\nu,k})_{k\in \mathbb{N}^\star}$ are $k$-th zeros of $J_{\nu}(x)$ on the positive real axis and
$\nu \in \C\setminus \{-\mathbb{N}^\star\}$, $0\leqslant X\leqslant Y\leqslant 1$, $\mathbf{z}\in \C$.
\\[0.5mm]
In what follows we shall discuss some basic properties of the Hankel transform, and we refer the readers for instance to
 \cite[Chap. 9]{PA96}. First,
recall that the $\nu$-th order Hankel transform of $f:(0,\infty)\to \R$ is defined as
\begin{align}\label{def:HankelTrans}
  \forall \, r>0,\quad   \mathcal{H}_{\nu}f(r)\triangleq \int_0^{\infty}x f(x)J_\nu(r x)\dd x.
\end{align}
This transformation  is well-defined for example when $f$ is piecewise continuous and subject to the integrability condition $\displaystyle{\int_0^\infty|f(r)|\sqrt{r}\dd r}$. Furthermore,
under the following assumptions that $f$ is of class $C^2$ and
\begin{align}\label{cond:Hankel}
  \lim_{x\to \infty} x^{\frac{1}{2}}f(x)=0,\quad \lim_{x\to \infty}x^{\frac{1}{2}}f'(x)=0,\quad \lim_{x\to 0}xf(x)=0,
\end{align}
we have
\begin{align}\label{hankel:diff}
  \mathcal{H}_\nu \Big(\Big(\frac{\dd^2 }{\dd x^2}+\frac{1}{x}\frac{\dd }{\dd x}
  -\frac{{\nu}^2}{x^2}\Big)f(x)\Big)=-r^2\mathcal{H}_{\nu}f(r).
\end{align}
For $\nu>-\frac{1}{2}$, we also have
\begin{align}\label{hankel:inverse}
  \mathcal{H}_{\nu}^2f(x)=f(x).
\end{align}

\subsection{Crandall-Rabinowitz's theorem}

The Crandall-Rabinowitz theorem from the local bifurcation theory plays a fundamental role in our paper, and for the proof we refer to \cite{C-R71}.
\begin{theorem}[Crandall-Rabinowitz's theorem]\label{thm:C-R}
Let $X$ and $Y$ be two Banach spaces, $V$ a neighborhood of $0$ in $X$ and let
$F : \R \times V \to Y$ be with the following  properties:
\begin{enumerate}
\item $F (\lambda, 0) = 0$ for any $\lambda\in \R$.
\item The partial derivatives $\partial_\lambda F$, $\partial_x F$ and $\partial_\lambda \partial_{x}F$ exist and are continuous.
\item $N(\mathcal{L}_0)$ and $Y/R(\mathcal{L}_0)$ are one-dimensional.
\item {\it Transversality assumption}: $\partial_\lambda \partial_x F(0, 0)x_0 \not\in R(\mathcal{L}_0)$, where
\begin{align*}
  N(\mathcal{L}_0) = \mathrm{span}\{x_0\}, \quad \mathcal{L}_0\triangleq \partial_x F(0,0).
\end{align*}
\end{enumerate}
If $Z$ is any complement of $N(\mathcal{L}_0)$ in $X$, then there is a neighborhood $U$ of $(0,0)$ in $\R \times X$,
an interval $(-a,a)$, and continuous functions $\varphi: (-a,a) \to \R$,
$\psi: (-a,a) \to Z$ such that $\varphi(0) = 0$, $\psi(0) = 0$ and
\begin{align*}
  F^{-1}(0)\cap U=\Big\{\big(\varphi(\xi), \xi x_0+\xi\psi(\xi)\big)\,:\,\vert \xi\vert<a\Big\}
  \cup\Big\{(\lambda,0)\,:\, (\lambda,0)\in U\Big\}.
\end{align*}
\end{theorem}

\bibliographystyle{plain}

 \end{document}